\documentclass{preprint}

\usepackage[utf8]{inputenc}
\usepackage[full]{textcomp}
\usepackage[osf]{newtxtext} 
\usepackage[title,titletoc]{appendix}
\usepackage{tikz-cd}

\usepackage{graphicx}
\usepackage{caption}
\usepackage{subcaption}
\graphicspath{{pictures/}}
\usepackage{float}

\usepackage{hyperref}
\usepackage{amsmath, bm}
\usepackage{amssymb}
\usepackage{mathtools}
\usepackage{amsfonts}
\usepackage{mhequ}
\usepackage{mhenvs}
\usepackage{mhsymb}
\usepackage{mathrsfs}
\usepackage{microtype}
\usepackage{wasysym}
\usepackage{centernot}
\usepackage{booktabs}
\usepackage{tikz}
\usetikzlibrary{decorations}
\usetikzlibrary{decorations.pathmorphing}
\usepackage{colortbl}
\usepackage{enumitem}
\usepackage{scalerel}
\usepackage{longtable}
\usepackage{comment}

\makeatletter
\def\DeclareSymbol#1#2#3{%
	\expandafter\gdef\csname MH@symb@#1\endcsname{%
	\tikz[baseline=#2,scale=0.15,draw=symbols,line join=round]{#3}}%
	\expandafter\gdef\csname MH@symb@#1s\endcsname{\scalebox{0.75}{%
	\tikz[baseline=#2,scale=0.15,draw=symbols,line join=round]{#3}}}%
	\expandafter\gdef\csname MH@symb@#1ss\endcsname{\scalebox{0.65}{%
	\tikz[baseline=#2,scale=0.15,draw=symbols,line join=round]{#3}}}%
	}
\def\<#1>{\ifthenelse{\boolean{mmode}}{\mathchoice{\csname MH@symb@#1\endcsname}{\csname MH@symb@#1\endcsname}{\csname MH@symb@#1s\endcsname}{\csname MH@symb@#1ss\endcsname}}{\csname MH@symb@#1\endcsname}}
\makeatother

\DeclareMathAlphabet{\mcb}{U}{BOONDOX-calo}{m}{n}
\SetMathAlphabet{\mcb}{bold}{U}{BOONDOX-calo}{b}{n}

\DeclareSymbol{0}{-2.4}{\node[dot] {};}

\definecolor{connection}{rgb}{0.7,0.1,0.1}

\tikzset{
root/.style={circle,fill=black!50,inner sep=0pt, minimum size=3mm},
        dot/.style={circle,fill=black,inner sep=0pt, minimum size=1.2mm},
        sdot/.style={circle,fill=black,inner sep=0pt,minimum size=.5mm},
        dotred/.style={circle,fill=black!50,inner sep=0pt, minimum size=2mm},
        var/.style={circle,fill=black!10,draw=black,inner sep=0pt, minimum size=3mm},
        kernel/.style={semithick,shorten >=2pt,shorten <=2pt},
        kernel1/.style={draw=black,thick},
        kernels/.style={snake=zigzag,shorten >=2pt,shorten <=2pt,segment amplitude=1pt,segment length=4pt,line before snake=2pt,line after snake=5pt,},
        rho/.style={densely dashed,semithick,shorten >=2pt,shorten <=2pt},
           testfcn/.style={dotted,semithick,shorten >=2pt,shorten <=2pt},
           tau/.style={circle,inner sep=1pt,draw=black,fill=white,text=black,thin},
        renorm/.style={shape=circle,fill=white,inner sep=1pt},
        labl/.style={shape=rectangle,fill=white,inner sep=1pt},
        xic/.style={very thin,circle,fill=symbols,draw=black,inner sep=0pt,minimum size=1.2mm},
        xi/.style={very thin,circle,fill=blue!10,draw=black,inner sep=0pt,minimum size=1.2mm},
        xix/.style={crosscircle,fill=blue!10,draw=black,inner sep=0pt,minimum size=1.2mm},
	xib/.style={very thin,circle,fill=blue!10,draw=black,inner sep=0pt,minimum size=1.6mm},
	xie/.style={very thin,circle,fill=green!50!black,draw=black,inner sep=0pt,minimum size=1.6mm},
	xid/.style={very thin,circle,fill=symbols,draw=black,inner sep=0pt,minimum size=1.6mm},
	xibx/.style={crosscircle,fill=blue!10,draw=black,inner sep=0pt,minimum size=1.6mm},
	kernels2/.style={very thick,draw=connection,segment length=12pt},
	not/.style={thin,circle,fill=symbols,draw=connection,fill=connection,inner sep=0pt,minimum size=0.5mm},
	>=stealth,
  }

\newtheorem{example}[lemma]{Example}
\newtheorem{exercise}[lemma]{Exercise}

\newcommand{\End}{\textnormal{End}}

\newcommand{\reg}{\textnormal{reg}}
\newcommand{\ex}{\textnormal{ex}}

\newcommand{\mcO}{\mathcal{O}}
\newcommand{\mcN}{\mathcal{N}}
\newcommand{\mcY}{\mathcal{Y}}
\newcommand{\mcZ}{\mathcal{Z}}
\newcommand{\mcM}{\mathcal{M}}
\newcommand{\mcQ}{\mathcal{Q}}

\newcommand{\mcI}{\mathcal{I}}
\newcommand{\mcT}{\mathcal{T}}
\newcommand{\mcC}{\mathcal{C}}
\newcommand{\mcP}{\mathcal{P}}
\newcommand{\mcD}{\mathcal{D}}

\newcommand{\mcH}{\mathcal{H}}
\newcommand{\mcF}{\mathcal{F}}
\newcommand{\mcR}{\mathcal{R}}

\newcommand{\mcB}{\mathcal{B}}
\newcommand{\mcL}{\mathcal{L}}

\newcommand{\mcV}{\mathcal{V}}

\newcommand{\mcA}{\mathcal{A}}

\newcommand{\mcE}{\mathcal{E}}

\newcommand{\cbF}{\mcb{F}}

\def\R{\mathbb{R}}
\def\N{\mathbb{N}}

\newcommand{\F}{\mathbf{F}}

\newcommand{\e}{\mathbf{e}}

\newcommand{\bfPi}{\mathbf{\Pi}}

\newcommand{\id}{\textnormal{id}}

\newcommand{\GL}{\mathbf{GL}}

\newcommand{\bone}{\mathbf{1}}

\newcommand{\mfp}{\mathfrak{p}}
\newcommand{\mfo}{\mathfrak{o}}
\newcommand{\mfn}{\mathfrak{n}}
\newcommand{\mfe}{\mathfrak{e}}

\newcommand{\mfG}{\mathfrak G}
\newcommand{\mfN}{\mathfrak{N}}
\newcommand{\mfB}{\mathfrak{B}}

\newcommand{\s}{\mathfrak s}
\newcommand{\mfL}{\mathfrak L}
\newcommand{\mft}{\mathfrak t}

\newcommand{\mfT}{\mathfrak{T}}

\newcommand{\mfF}{\mathfrak{F}}

\newcommand{\mfK}{\mathfrak K}

\newcommand{\mrd}{\mathop{}\!\mathrm{d}}

\newcommand\restr{\mathord{\upharpoonright}}

\newcommand{\fork}{\begin{tikzpicture}  [scale=0.2,baseline=-0.1cm]
\node [dot] (uu) at (-2, 0) {};
\node [dot] (m) at (-2, -1) {};
\draw (m) to (uu);
\node [dot] (uur) at (-1,1) {};
\draw (uu) to  (uur);
\node [dot] (uul) at (-3,1) {};
\draw (uu) to  (uul);
\node [dot] (uum) at (-2,1) {};
\draw (uu) to (uum);
\draw[decorate,decoration={zigzag,segment length=0.5mm, amplitude=.3mm},thick] (uul) to  (-3,2);
\draw[decorate,decoration={zigzag,segment length=0.5mm, amplitude=.3mm},thick] (uum) to  (-2,2);
\draw[decorate,decoration={zigzag,segment length=0.5mm, amplitude=.3mm},thick] (uur) to  (-1,2);
\end{tikzpicture}}

\newcommand{\she}{\begin{tikzpicture}  [scale=0.2,baseline=0.1cm]
\node [dot] (root) at (0, 0) {};
\node [dot] (um) at (0, 1) {};
\draw (root) to (um);
\draw[decorate,decoration={zigzag,segment length=0.5mm, amplitude=.3mm},thick] (um) to  (0,2);
\end{tikzpicture}}

\newcommand{\shesq}{\begin{tikzpicture}  [scale=0.2,baseline=0.1cm]
\node [dot] (root) at (0, 0) {};
\node [dot] (ul) at (-1, 1) {};
\node [dot] (ur) at (0, 1) {};
\draw[decorate,decoration={zigzag,segment length=0.5mm, amplitude=.3mm},thick] (ul) to  (-1,2);
\draw[decorate,decoration={zigzag,segment length=0.5mm, amplitude=.3mm},thick] (ur) to  (0,2);
\draw (root) to  (ul);
\draw (root) to  (ur);
\end{tikzpicture}}

\newcommand{\shecube}{\begin{tikzpicture}  [scale=0.2,baseline=0.1cm]
\node [dot] (root) at (0, 0) {};
\node [dot] (ul) at (-1, 1) {};
\node [dot] (um) at (0, 1) {};
\node [dot] (ur) at (1, 1) {};
\draw[decorate,decoration={zigzag,segment length=0.5mm, amplitude=.3mm},thick] (ul) to  (-1,2);
\draw[decorate,decoration={zigzag,segment length=0.5mm, amplitude=.3mm},thick] (um) to  (0,2);
\draw[decorate,decoration={zigzag,segment length=0.5mm, amplitude=.3mm},thick] (ur) to  (1,2);
\draw (root) to  (ul);
\draw (root) to  (ur);
\draw (root) to  (um);
\end{tikzpicture}}

\newcommand{\intshe}{\begin{tikzpicture}  [scale=0.2,baseline=0.2cm]
\node [dot] (root) at (0, 0) {};
\node [dot] (u) at (0, 1) {};
\node [dot] (ur) at (1, 2) {};
\draw[decorate,decoration={zigzag,segment length=0.5mm, amplitude=.3mm},thick] (ur) to  (1,3);
\draw (root) to  (u);
\draw (u) to  (ur);
\end{tikzpicture}}

\newcommand{\intsheext}{\begin{tikzpicture}  [scale=0.2,baseline=0.2cm]
\node [dot] (root) at (0, 0) {};
\node [dot,label={[yshift=-0.1cm,xshift=0.2cm]135:{\scriptsize $\bone^{p}$}}] (u) at (0, 1) {};
\node [dot] (ur) at (1, 2) {};
\draw[decorate,decoration={zigzag,segment length=0.5mm, amplitude=.3mm},thick] (ur) to  (1,3);
\draw (root) to  (u);
\draw (u) to  (ur);
\end{tikzpicture}}

\newcommand{\intone}{\begin{tikzpicture}  [scale=0.2]
\node [dot] (root) at (0, 0) {};
\node [dot] (u) at (0, 1) {};
\draw (root) to  (u);
\end{tikzpicture}}

\newcommand{\intonesq}{\begin{tikzpicture}  [scale=0.2]
\node [dot,red] (root) at (0, 0) {};
\node [dot,red] (ur) at (1, 1) {};
\node [dot,red] (ul) at (-1, 1) {};
\draw[red] (root) to  (ur);
\draw[red] (root) to  (ul);
\end{tikzpicture}}

\newcommand{\graftone}{\begin{tikzpicture}  [scale=0.2,baseline=0.1cm]
\node [dot] (root) at (0, 0) {};
\node [dot] (u) at (1, 1) {};
\draw (root) to  (u);
\node [dot,red] (root1) at (-1, 1) {};
\node [dot,red] (ur) at (0, 2) {};
\node [dot,red] (ul) at (-2, 2) {};
\draw[red] (root1) to  (ur);
\draw[red] (root1) to  (ul);
\draw (root) to (root1);
\end{tikzpicture}}

\newcommand{\grafttwo}{\begin{tikzpicture}  [scale=0.2,baseline=0.2cm]
\node [dot] (root) at (0, 0) {};
\node [dot] (u) at (0, 1) {};
\draw (root) to  (u);
\node [dot,red] (root1) at (0, 2) {};
\node [dot,red] (ur) at (1, 3) {};
\node [dot,red] (ul) at (-1, 3) {};
\draw[red] (root1) to  (ur);
\draw[red] (root1) to  (ul);
\draw (u) to (root1);
\end{tikzpicture}}

\newcommand{\intoneext}{\begin{tikzpicture}  [scale=0.2,baseline=0.1cm]
\node [dot] (root) at (0, 0) {};
\node [dot,label={[yshift=-0.1cm]90:{\scriptsize $\bone^{q}$}}] (u) at (0, 1) {};
\draw (root) to  (u);
\end{tikzpicture}}

\newcommand{\intshek}{\begin{tikzpicture}  [scale=0.2,baseline=0.2cm]
\node [dot] (root) at (0, 0) {};
\node [dot] (u) at (0, 1) {};
\node [dot] (ur) at (1, 2) {};
\draw[decorate,decoration={zigzag,segment length=0.5mm, amplitude=.3mm},thick] (ur) to  (1,3);
\draw (root) to node[ midway,left]{\scriptsize $k$} (u);
\draw (u) to  (ur);
\end{tikzpicture}}

\newcommand{\intonek}{\begin{tikzpicture}  [scale=0.2,baseline=0cm]
\node [dot] (root) at (0, 0) {};
\node [dot] (u) at (0, 1) {};
\draw (root) to node[ midway,left]{\scriptsize $k$} (u);
\end{tikzpicture}}

\newcommand{\bigtree}{\begin{tikzpicture}  [scale=0.2,baseline=-0.1cm]
\node [dot] (uu) at (-2, 0) {};
\node [dot] (m) at (-2, -1) {};
\draw (m) to (uu);
\node [dot] (uur) at (-1,1) {};
\draw (uu) to  (uur);
\node [dot] (uul) at (-3,1) {};
\draw (uu) to  (uul);
\node [dot] (uum) at (-2,1) {};
\draw (uu) to (uum);
\draw[decorate,decoration={zigzag,segment length=0.5mm, amplitude=.3mm},thick] (uul) to  (-3,2);
\draw[decorate,decoration={zigzag,segment length=0.5mm, amplitude=.3mm},thick] (uum) to  (-2,2);
\draw[decorate,decoration={zigzag,segment length=0.5mm, amplitude=.3mm},thick] (uur) to  (-1,2);
\node [dot] (ul) at (-3, 0) {};
\node [dot] (ur) at (-1, 0) {};
\draw (m) to (ul);
\draw (m) to (ur);
\draw[decorate,decoration={zigzag,segment length=0.5mm, amplitude=.3mm},thick] (ul) to  (-4,1);
\draw[decorate,decoration={zigzag,segment length=0.5mm, amplitude=.3mm},thick] (ur) to  (0,1);
\end{tikzpicture}}

\newcommand{\noise}{\begin{tikzpicture}  [scale=0.2]
\node [dot] (root) at (0, 0) {};
\draw[decorate,decoration={zigzag,segment length=0.5mm, amplitude=.3mm},thick] (root) to  (0,1);
\end{tikzpicture}}

\colorlet{ng}{green!50!black}

\newcommand{\bigtreecolour}{\begin{tikzpicture}  [scale=0.2,baseline=-0.2cm]
\node [dot, ng] (uu) at (-2, 0) {};
\node [dot, red,label={[yshift=0.1cm]-90:{\scriptsize $X^a$}}] (m) at (-2, -1) {};
\draw (m) to (uu);
\node [dot,ng] (uur) at (-1,1) {};
\draw[ng] (uu) to  (uur);
\node [dot] (uul) at (-3,1) {};
\draw (uu) to  (uul);
\node [dot,ng] (uum) at (-2,1) {};
\draw[ng] (uu) to (uum);
\draw[decorate,decoration={zigzag,segment length=0.5mm, amplitude=.3mm},thick] (uul) to  (-3,2);
\draw[ng,decorate,decoration={zigzag,segment length=0.5mm, amplitude=.3mm},thick] (uum) to  (-2,2);
\draw[ng,decorate,decoration={zigzag,segment length=0.5mm, amplitude=.3mm},thick] (uur) to  (-1,2);
\node [dot,red] (ul) at (-3, 0) {};
\node [dot,red] (ur) at (-1, 0) {};
\draw[red] (m) to (ul);
\draw[red] (m) to (ur);
\draw[red,decorate,decoration={zigzag,segment length=0.5mm, amplitude=.3mm},thick] (ul) to  (-4,1);
\draw[red,decorate,decoration={zigzag,segment length=0.5mm, amplitude=.3mm},thick] (ur) to  (0,1);
\end{tikzpicture}}

\newcommand{\shesqred}{\begin{tikzpicture}  [scale=0.2,baseline=0.1cm]
\node [dot,red] (root) at (0, 0) {};
\node [dot,red] (ul) at (-1, 1) {};
\node [dot,red] (ur) at (1, 1) {};
\draw[red,decorate,decoration={zigzag,segment length=0.5mm, amplitude=.3mm},thick] (ul) to  (-2,2);
\draw[red,decorate,decoration={zigzag,segment length=0.5mm, amplitude=.3mm},thick] (ur) to  (2,2);
\draw[red] (root) to  (ul);
\draw[red] (root) to  (ur);
\end{tikzpicture}}

\newcommand{\shesqgreen}{\begin{tikzpicture}  [scale=0.2,baseline=0.1cm]
\node [dot,ng] (root) at (0, 0) {};
\node [dot,ng] (ur) at (1, 1) {};
\node [dot,ng] (um) at (0, 1) {};
\draw[ng,decorate,decoration={zigzag,segment length=0.5mm, amplitude=.3mm},thick] (ur) to  (1,2);
\draw[ng,decorate,decoration={zigzag,segment length=0.5mm, amplitude=.3mm},thick] (um) to  (0,2);
\draw[ng] (root) to  (ur);
\draw[ng] (root) to  (um);
\end{tikzpicture}}

\newcommand{\intshepa}{\begin{tikzpicture}  [scale=0.2,baseline=0.1cm]
\node [dot,label={[yshift=0.1cm]-90:{\scriptsize $\bone^{p}X^{a-b}$}}] (root) at (0, 0) {};
\node [dot,label={[yshift=-1mm,xshift=-1mm]45:{\scriptsize $\bone^{q}$}}] (u) at (0, 1) {};
\node [dot] (ul) at (-1, 2) {};
\draw[decorate,decoration={zigzag,segment length=0.5mm, amplitude=.3mm},thick] (ul) to  (-1,3);
\draw (root) to  (u);
\draw (u) to node[sloped,midway,below]{\scriptsize $m$} (ul);
\end{tikzpicture}}

\newcommand{\shesqgreenm}{\begin{tikzpicture}  [scale=0.2,baseline=0.1cm]
\node [dot,ng,label={[yshift=0.1cm]-90:{\scriptsize $X^m$}}] (root) at (0, 0) {};
\node [dot,ng] (ur) at (1, 1) {};
\node [dot,ng] (um) at (0, 1) {};
\draw[ng,decorate,decoration={zigzag,segment length=0.5mm, amplitude=.3mm},thick] (ur) to  (1,2);
\draw[ng,decorate,decoration={zigzag,segment length=0.5mm, amplitude=.3mm},thick] (um) to  (0,2);
\draw[ng] (root) to  (ur);
\draw[ng] (root) to  (um);
\end{tikzpicture}}

\newcommand{\shesqredb}{\begin{tikzpicture}  [scale=0.2,baseline=0.1cm]
\node [dot,red,label={[yshift=0.1cm]-90:{\scriptsize $X^{b}$}}] (root) at (0, 0) {};
\node [dot,red] (ul) at (-1, 1) {};
\node [dot,red] (ur) at (1, 1) {};
\draw[red,decorate,decoration={zigzag,segment length=0.5mm, amplitude=.3mm},thick] (ul) to  (-2,2);
\draw[red,decorate,decoration={zigzag,segment length=0.5mm, amplitude=.3mm},thick] (ur) to  (2,2);
\draw[red] (root) to  (ul);
\draw[red] (root) to  (ur);
\end{tikzpicture}}
\usepackage[english]{babel}

\newcommand{\Span}{\mathrm{span}}

\makeatletter
\newcommand*\fprod{\mathpalette\fprod@{.5}}
\newcommand*\fprod@[2]{\mathbin{\vcenter{\hbox{\scalebox{#2}{$\m@th#1\bullet$}}}}}
\makeatother

\newcommand{\Deltam}{\Delta^{\!-}}
\newcommand{\Deltap}{\Delta^{\!+}}
\newcommand{\Deltapm}{\Delta^{\!\pm}}

\newcommand\graft{\curvearrowright}

\newcommand\hgraft{\hat\curvearrowright}
\newcommand\bgraft{\bar\curvearrowright}
\renewcommand\root{\mathrm{root}}
\newcommand\poly{\mathrm{poly}}
\newcommand\nonroot{\mathrm{non\textnormal{-}root}}

\usepackage{stmaryrd}
\def\fancynorm#1{{\talloblong #1 \talloblong}}

\begin{document}
	
	\date{\today}
	\title{Hopf and pre-Lie algebras in regularity structures}
	\author{Ilya Chevyrev}
	\institute{The University of Edinburgh, \email{ichevyrev@gmail.com}}
	\date{\today}
	\titleindent=0.65cm
	
	\maketitle
	
\begin{abstract}
These lecture notes aim to present the algebraic theory of regularity structures as developed in~\cite{Hairer14,BHZ19,BCCH21}. The main aim of this theory is to build a systematic approach to renormalisation of singular SPDEs; together with complementary analytic results, these works give a general solution theory for a wide class of semilinear parabolic singular SPDEs in the subcritical regime. We demonstrate how `positive' and `negative' renormalisation can be described using interacting Hopf algebras, and how the renormalised non-linearities in SPDEs can be computed. For the latter, of crucial importance is a pre-Lie structure on non-linearities on which the negative renormalisation group acts through pre-Lie morphisms. To show the main aspects of these results without introducing many notations and assumptions, we focus on a special case of the general theory in which there is only one equation and one noise. These lectures notes are an expansion of the material presented at a minicourse with the same title at the Master Class and Workshop ``Higher Structures Emerging from Renormalisation'' at the ESI, Vienna, in November 2021.
\end{abstract}
	
\setcounter{tocdepth}{2}
	\tableofcontents	
	
\section{Introduction}\label{lec1}

These lecture notes aim to present the main algebraic structures which appear in the theory of regularity structures designed to solve singular stochastic partial differential equations (SPDEs) of the form
\begin{equation}\label{eq:SPDE}
(\partial_t -\mcL) u=  F(u,\nabla u,\ldots, \xi,\nabla\xi,\ldots)\;.
\end{equation}
Here $u$ is a function (or distribution) 
defined on $\R_+\times \R^{d-1}$
taking values in $\R^n$, $\partial_t-\mcL$ is a parabolic operator (the $t$-coordinate is the $d$-coordinate),
and $\xi$ is a noise term, e.g. space or space-time white noise.

A number of equations of interest in stochastic analysis and quantum field theory are of this form and have been solved using regularity structures.
These include stochastic (ordinary) differential equations (SDEs), in which case regularity structures specialises to a form of rough path theory~\cite{Lyons,FV10, FrizHairer20},
the KPZ equation~\cite{Hairer13,BGHZ21}, the Parabolic Anderson Model (PAM)~\cite{Hairer14,CFG17}, the $\Phi^4_d$ equations~\cite{DPD2,ChandraWeber17,MoinatWeber20},
the dynamic sine-Gordon model~\cite{HairerShen16,CHS18}
and the stochastic Yang--Mills equations~\cite{CCHS20,CCHS22,Chevyrev22}.

These notes are based on material primarily from~\cite{Hairer14,BHZ19,BCCH21}, although there are some differences in conventions, results, and proofs that we try to point out along the way.
These works, together with~\cite{CH16}, which focuses to the analytic aspects of renormalisation,
provide a general way to solve a wide class subcritical, semilinear, parabolic SPDEs of the form~\eqref{eq:SPDE}.
We will not go into the full generality of this theory, focusing primarily on the case that $u$ is an $\R$-valued function (or distribution)
and there is only one $\R$-valued noise $\xi$.
The general theory, however, is able to handle systems of equations with multiple noises and taking values in vector spaces (see in particular~\cite[Sec.~5]{CCHS20} for an extension to vector-valued noises and solutions based on a functorial construction).

Furthermore, a number of works have
extended or provided alternative proofs for the theory developed in~\cite{Hairer14,BHZ19,BCCH21}, including treating quasilinear equations~\cite{GH19,Gerencser20}, avoiding so-called extended decorations~\cite{BB21a,BB21b} (see Section~\ref{sec:neg_renorm}),
finding a relation to deformations of pre-Lie products~\cite{BM20},
and extending the framework to settings with inhomogeneous noises~\cite{GHM22}.

We also mention that several alternative approaches exist to solve equations of the form~\eqref{eq:SPDE}.
This includes the theory of paracontrolled distributions~\cite{GIP15}, developed at the same time as regularity structures,
and the renormalisation group method of~\cite{Kupiainen2016} (see also~\cite{Duch21,Duch22}).
An alternative theory has been initiated in~\cite{OW} and developed further in~\cite{Otto_18, Otto_21,Otto_21II,Otto_22} that bears similarity to regularity structures with some advantages in treating certain equations (e.g. quasilinear equations).

The general algebraic machinery developed in~\cite{Hairer14,BHZ19,BCCH21}
is strongly motivated by B-series appearing in numerical analysis.
The reader will in particular see a resemblance between the Butcher group~\cite{Butcher72,hairer74} (the origins of which date back to Cayley)
and the work of Connes--Kreimer~\cite{CK98, CK2000,CK2001} on renormalisation in quantum field theory (although both regularity structures and the work of Connes--Kreimer deal with renormalisation, the link with the Connes--Kreimer Hopf algebra is rather through Butcher series than through QFT).
In particular, a crucial role is played by
interacting Hopf algebras,
which is an algebraic structure identified in~\cite{CEFM11}. See also~\cite{CHV10} for a description of this interaction at the level of groups.
The B-series based on decorated trees identified in~\cite{BCCH21} have since been applied to numerical schemes for dispersive PDEs~\cite{Bruned_Schratz_22}.

\begin{remark}
Since we wish to emphasise the algebraic aspects,
we will not be too careful about analytic / probabilistic considerations of actually solving such equations,
e.g. we will hardly speak about the regularity / integrability of boundary conditions, which are of course important when analytically solving for $u$ in~\eqref{eq:SPDE}.
We will always assume the spatial component is periodic,
but see~\cite{MateMartin19,MateMartin21,HairerPardoux21} for results on other types of boundary conditions.
\end{remark}

\subsection{Motivation: SPDEs and power counting}
\label{sec:motivation}

When trying to solve~\eqref{eq:SPDE} in the singular regime, the main problem which appears is
\textit{multiplication of distributions}.
In the rest of this subsection, we discuss this problem, how it ties into the equation,
and potential ways to solve it.
A brief outline of the structure of the lecture notes is given at the end of this subsection.

We start by recalling a classical result from harmonic analysis, often referred to as Young's product theorem~\cite{Young36}.
\begin{proposition}\label{prop:Young}
Let $\alpha,\beta \in \R$. Denote by $C^\alpha=C^\alpha(\R^d)$ the H{\"o}lder--Besov space of periodic distributions on $\R^d$ of regularity $\alpha$.
Then multiplication $C^\infty\times C^\infty \to C^\infty$ extends to a continuous bilinear map $C^\alpha\times C^\beta\to C^{\alpha\wedge \beta}$,
if and only if $\alpha+\beta>0$.
\end{proposition}

\begin{notation}
We will use $\alpha-$ to denote a number slightly below $\alpha\in\R$, i.e. $\alpha- = \alpha-\kappa$ for $\kappa>0$ arbitrarily small but fixed.
This will make power-counting much simpler.
\end{notation}
To see how to apply Proposition~\ref{prop:Young} to SPDEs,
consider first the ODE with multiplicative noise
\begin{equation}\label{eq:SDE}
\partial_t u = f(u)\xi\;,
\end{equation}
posed on $u\colon \R\to \R^n$, where $\xi\colon \R \to \R^m$ and $f\colon\R^n\to L(\R^m,\R^n)$ is smooth.
We often think of $\xi = \frac{\mrd W}{\mrd t}$,
where $W$ is the `driving path'.

Assume $\xi\in C^{\alpha}$ for some $\alpha\in\R$.
One expects $u$ to have at best one more derivative than $\xi$, i.e. we can try to solve for $u\in C^{\alpha+1}\Rightarrow f(u)\in C^{\alpha+1}$.
Therefore we can use Proposition~\ref{prop:Young} to make analytic sense of the product $f(u)\xi$ if and only if $\alpha+\alpha+1>0$,
i.e. $\alpha>-\frac12$.
In this case, we can indeed solve for~\eqref{eq:SDE} with $u\in C^{\alpha+1}$ --
a variant of this result was shown
by Lyons~\cite{Lyons-Young}
as a precursor to his introduction of rough paths~\cite{Lyons}.
In particular, there is a deterministic solution theory for SDEs driven by fractional Brownian motion with Hurst parameter $H>\frac12$ since sample paths of this process are almost surely in
$C^{H-}$.

However, this deterministic analysis comes short of solving~\eqref{eq:SDE} for $\xi=\frac{\mrd W}{\mrd t}$ where $W$ is a sample paths of Brownian motion as we have $\xi\in C^{\alpha}$ with $\alpha<-\frac12$ in this case.

\begin{remark}\label{rem:integral_ill-defined}
One may wonder if H{\"o}lder--Besov spaces $C^\alpha$ are not the correct function space to use -- perhaps if we looked at a different space (e.g. some more general Orlicz or Besov space) we can recover continuity of the product, even in the regime of Brownian motion.
Unfortunately, another result of Lyons~\cite{TerryPath} is that there exists no Banach space containing smooth paths and allowing to approximate almost every realisation of Brownian motion to which $(y,x)\mapsto \int_0^1 y_t \mrd x_t$ extends as a continuous bilinear map (see also~\cite[Sec.~1.5.1]{MR2314753}).
Since $\int_0^1 y_t \mrd x_t$ is a simple example of an ODE driven by $(x,y)$,
this suggests that no amount of deterministic \textit{linear} functional analysis is able to recover It{\^o} calculus.
(See~\cite{Chevyrev22Heat} for an ill-posedness result of a similar flavour for a non-linear heat equation with rough initial conditions.)
\end{remark}

\begin{remark}
In case $n=1$, one can actually solve for $u$ as $u_t = e^{X_t f}(u_0)$, where $X_t = \int_0^t \xi_s\mrd s$ and $e^{xf}$ for $x\in \R$ denotes the flow for the ODE $\dot y_t = xf(y)$.
Hence the product $f(u)\xi = f(e^{X_t f}(u_0))\partial_t X_t = \partial_t \exp(X_t f)(u_0)$
is perfectly well-defined.
This explicit solution formula is the reason SDEs with a $1$-dimensional driver (or more generally when the drivers are paired with commuting vector fields)
are stable under the uniform norm on driving paths.
\end{remark}

\subsubsection{Stochastic quantisation of \texorpdfstring{$\Phi^4_{d-1}$}{Phi4}}
Another famous example, this time with a multi-dimensional domain, is the $\Phi^4_{d-1}$ stochastic quantisation equation
\begin{equation}\label{eq:Phi4d}
(\partial_t-\Delta) u = -u^3 + \xi\;,\qquad u\colon \R_+\times \R^{d-1}\to \R\;,
\end{equation}
where $\xi$ is a (space-time) white noise on $\R^d$.
At this point, let us recall the Schauder estimates for parabolic equations and the Kolmogorov-type theorem on the regularity of white noise.\footnote{These results require us to work with the parabolic rather than Euclidean scaling, see the start of Section~\ref{subsec:subcrit}}
\begin{theorem}[Schauder estimates]
If $(\partial_t-\Delta)v=\zeta$, where $\zeta\in C^\alpha$,
then $v\in C^{\alpha+2}$.
\end{theorem}
\begin{theorem}[Kolmogorov regularity]
Almost surely $\xi\in C^{-\frac d2-\frac12-}$.
\end{theorem}
Applying these results, we see that we can at best solve for $u$ in $C^{\frac32-\frac d2-}$.
If $d=2$, we can therefore solve for $u$ in $C^{1/2-}$,
but if $d\geq 3$, the cube $u^3$
does not make analytic sense.
We will now explore this example in more detail as motivation for the later sections.

\subsubsection{\texorpdfstring{$d=3$}{d=3}: Da Prato--Debussche trick}\label{subsubsec:DPD}

If $d=3$, we can actually give a meaning to the solution of~\eqref{eq:Phi4d} using the so called Da Prato--Debussche trick~\cite{DPD2}.
Consider $v$ and $X$ such that
\[
(\partial_t -\Delta)v = \xi\;,\qquad u=v+X\;.
\]
Then Schauder estimates imply $v\in C^{0-}$ and $X$ solves
\begin{equation}\label{eq:X_equation}
(\partial_t -\Delta)X = -(v+X)^3 = -(v^3 + 3v^2X+3vX^2 + X^3)\;.
\end{equation}
We notice that, if we can make sense of $v^3$ and $v^2$ as functions, say, in $C^{0-}$,
then we expect $X$ to be in $C^{2-}$, which renders the products $v^2X$ and $vX^2$ well-posed (with plenty of room to spare).
Luckily, Wick renormalisation indeed allows us to give a sensible definition to $v^3$ and $v^2$.

\begin{proposition}\label{prop:Wick}
Let $v_\eps=v*\delta_\eps$ for $\eps>0$, where $\delta_\eps$ is a smooth mollifier converging to the Dirac delta as $\eps\to 0$.
Then there exist deterministic constants $C_\eps\in \R$ for $\eps>0$ and random space-time distributions $v^{:2:}$, $v^{:3:}$ (called Wick powers)
such that
$v^2_\eps-C_\eps \to v^{:2:}$
and $v^3-3C_\eps v \to v^{:3:}$
in $C^{0-}$ in probability
as $\eps\to 0$.
Moreover,  $C_\eps\to \infty$ as $\eps\to0$.
\end{proposition}
The idea is now to replace $v^3$ and $v^2$ in~\eqref{eq:X_equation} by $v^{:3:}$ and $v^{:2:}$ respectively.
As indicated above, we are then able to solve for $X$ in $C^{2-}$ and \textit{define} the solution $u$ to~\eqref{eq:Phi4d} as $u\eqdef v+X \in C^{0-}$.
Furthermore, we can obtain the following approximation result.
\begin{theorem}\label{thm:DPD}
Following the notation from Proposition~\ref{prop:Wick},
let $\xi_\eps=\xi*\delta_\eps$.
Then the solutions to
\[
(\partial_t-\Delta)u_\eps = -u_\eps^3 + 3C_\eps u_\eps + \xi_\eps
\]
converge to a random space-time distribution $u$ in $C^{0-}$ in probability as $\eps\to0$.
\end{theorem}
Note that since $C_\eps \to \infty$, our $u$ does not really solve a PDE. It is instead obtained as a limit of solutions to \textit{renormalised} PDEs,
and is a continuous function not of $\xi$ alone, but of the triple $(\xi,v^{:2:},v^{:3:})\in C^{2-}\times C^{0-}\times C^{0-}$.
The distributions $v^{:2:}$ and $v^{:3:}$ are of course measurable functions of $\xi$ (at least once a suitable choice for $C_\eps$ is fixed),
but they are not \textit{continuous} functions.

\subsubsection{\texorpdfstring{$d=4$}{d=4}}
\label{subsubsec:d4}

Let us now consider~\eqref{eq:Phi4d} with $d=4$.
This time $v\in C^{-1/2-}$ and the Wick powers $v^{:3:}$ and $v^{:2:}$ are more singular:
the same result as~\ref{prop:Wick} holds \textit{except} that $v^{:3:}\in C^{-3/2-}$ and $v^{:2:}\in C^{-1-}$.
Performing the Da Prato--Debussche trick again, we see that $X$ is at best in $C^{1/2-}$,
which renders the products $v^{:2:}X$ and $vX^2$ ill-posed.
At this point, it is tempting to try the trick again by singling out the worst term in the equation for $X$:
define $w$ and $Y$ by
\[
(\partial_t -\Delta) w = v^{:3:}\;,\qquad X=w+Y\;.
\]
Then $Y$ solves
\begin{equation}\label{eq:Y_equation}
(\partial_t-\Delta) Y = -(3v^{:2:}(w+Y) + 3v(w+Y)^2 + (w+Y)^3)
\end{equation}
and Schauder estimates imply that $w\in C^{1/2-}$.
Furthermore, the product $v^{:2:}w$ is an explicit function of $\xi$, much like $v^{:2:}$ and $v^{:3:}$,
and there is good hope we can make sense of it (and the other products of $v$ and $w$) through stochastic analysis.
Indeed, after another renormalisation, one can make sense of $v^{:2:}w\in C^{-1-}$ (which is natural since $v^{:2:}\in C^{-1-}$ and multiplication by a more regular function should not increase regularity).

We therefore expect $Y \in C^{1-}$ so it seems we have gained something.
However, the product $v^{:2:}Y$ is still (borderline) ill-posed.
Let us try the trick one more time, defining
\[
(\partial_t-\Delta)a = v^{:2:}w\;,\qquad Y = a+Z\;.
\]
Then $a\in C^{1-}$ and $Z$ solves
\[
(\partial_t-\Delta)Z = -(3v^{:2:}(a+Z) + \ldots)
\]
where $\ldots$ denotes further terms.
We see that, even if we can make sense of $v^{:2:}a$ again, we expect its regularity to be at best that of $v^{:2:}$, which $C^{-1-}$.
This again implies $v^{:2:}Z$ is ill-defined.
It is not difficult to see that further iterations of the Da Prato--Debussche trick will not remove this problem, so we are seemingly stuck.
At this point, let us make an important remark.

\begin{remark}\label{rem:motivation}
If we could make sense of $v^{:2:}w$ in $C^{-1/2-}$ instead of $C^{-1-}$ which would correspond roughly to multiplication behaving like $C^\alpha\times C^\beta\to C^{\alpha+\beta}$ instead of $C^{\alpha}\times C^{\beta}\to C^{\alpha\wedge \beta}$,
then we would be in much better shape.
Indeed, after making sense of the remaining products of $w$ and $v$, we should be able to solve for $Y$ in 
$C^{3/2-}$, with the most singular product $v^{:2:}Y$ now well-defined by Proposition~\ref{prop:Young}.

Of course multiplication does not behave like $C^\alpha\times C^\beta\to C^{\alpha+\beta}$.
However, if instead of $w$ we consider the family of functions
\[
\{w_x\}_{x\in \R^d}\;,\qquad w_x = w-w(x)\;,
\]
then we can show that the product $ v^{:2:}w_x$, again after renormalisation, behaves like it were in $C^{-1/2-}$
\textit{locally around $x$}\footnote{One
needs to test against approximations of the Dirac delta centred at $x$ to make this precise, see Definition~\ref{def:pre_model}.}
This parallels the (obvious) fact that if $f$ and $g$ are functions such that $|f(y)| \leq A|y-x|^\alpha$ and $|g(y)|\leq B|y-x|^\beta$ for some $\alpha,\beta\in\R$,
then $|(fg)(y)|\leq AB |y-x|^{\alpha+\beta}$.

The idea now is that we should try to describe $Y$ \textit{locally} around every point $x$ as a function of the finite collection of explicit stochastic objects $v^{:2:}w_x, vw_x, \ldots$ that appear on the right-hand side of~\eqref{eq:Y_equation},
plus a remainder which behaves better locally around $x$.
One should then be able to make sense of the all the singular products and solve for $Y$.
Our final solution $u$ to~\eqref{eq:Phi4d} would then be defined as $u=v+w + Y$.
\end{remark}
The above remark is far from precise, but one of the achievements of regularity structures is to make this argument rigorous.
Indeed, the following result is shown in~\cite[Theorem~1.15]{Hairer14} which extends Theorem~\ref{thm:DPD}
to $d=4$.
\begin{theorem}\label{thm:Phi43}
Let $\xi_\eps=\xi*\delta_\eps$ be a mollification of $\xi$.
Then there exist deterministic constants $\bar C_\eps\in \R$
such that
solutions to
\[
(\partial_t-\Delta)u_\eps = - u_\eps^3 + \bar C_\eps u_\eps + \xi_\eps
\]
converge to a random space-time distribution $u$ in $C^{-1/2-}$ in probability as $\eps\to 0$.
\end{theorem}

\begin{remark}
Unlike in Theorem~\ref{thm:DPD}, the constant $\bar C_\eps$ in Theorem~\ref{thm:canonical_model} is \textit{not} the constant $3C_\eps$ obtained from the Wick renormalisation $v^2_\eps - C_\eps \to v^{:2:}$.
\end{remark}
As one may expect, proving such a result requires some algebraic considerations.
First, it is not clear how the distributions $v^{:2:}w_x$ and $v^{:2:}w_y$ relate to each other for different points $x,y$
and what information we need about them to solve for $Y$.
Second, it is not clear how the renormalisation of singular products $v^{2} \mapsto v^{:2:}$ akin to Proposition~\ref{prop:Wick} (assuming it can be done at all!)
interacts with the recentring procedure $v^{:2:}w \mapsto v^{:2:}w_x$.
Third, the existence and computation of the renormalisation constant $\bar C_\eps$ is not clear (even for $d=3$, if $3C_\eps$ in Proposition~\ref{prop:Wick} happened to be
$4C_\eps$,
then a statement like in Theorem~\ref{thm:DPD} would be false as one would require an additional non-local renormalisation term $C_\eps v_\eps$).

In the rest of these lecture notes, we will
show the kind of algebraic structures which appear in the solutions of these issues.
We will specifically discuss
\begin{itemize}
\item positive renormalisation (Section~\ref{sec:pos_renorm}), which addresses the recentring $w\mapsto w_x$, $v^{:2:}w \mapsto v^{:2:}w_x$,  etc.

\item negative renormalisation (Section~\ref{sec:neg_renorm}), which addresses the transformations $v^{2}_\eps\mapsto v^{:2:}_\eps \eqdef v^2-C_\eps$ etc.
and how they interact with positive renormalisation, and

\item renormalisation of SPDEs (Section~\ref{sec:renorm_SPDE}), where we identify the counterterms $\bar C_\eps u_\eps$ etc. for a general equation.
\end{itemize}

Sections~\ref{sec:pos_renorm} and~\ref{sec:neg_renorm} are primarily based on~\cite{Hairer14,BHZ19} and Section~\ref{sec:renorm_SPDE} is based on~\cite{BCCH21}.
In the rest of this section, we set up some preliminary definitions and results used in the sequel.

\subsection{Picard iterations and trees}

A potential way to solve~\eqref{eq:SPDE} is to first consider smooth $\xi$ and rewrite the equation in mild formulation as
\[
u = G u_0 + G*(F(u,\nabla u,\ldots \xi,\nabla \xi,\ldots))\;,
\]
where $G=(\partial_t-\mcL)^{-1}$ is the Green's function of $\partial_t -\mcL$, $*$ denotes space-time convolution, and $(Gu_0)(t,\cdot) = e^{t\mcL}u_0$ is the harmonic extension of the initial condition.

It is then natural to define the map
\begin{equ}
\mcM(u)\eqdef G u_0 + G*(F(u,\nabla u,\ldots \xi,\nabla\xi,\ldots))\;,
\end{equ}
so that the solution to~\eqref{eq:SPDE} is a fixed point $\mcM(u)=u$.
Under favourable conditions, namely if $\mcM$ is a contraction, the iterations $u^{(0)}\eqdef0$, $u^{(n)} = \mcM(u^{(n-1)})$ converge to a unique fixed point $u$ as $n\to\infty$.

Writing $F(u,\xi)$ instead of $F(u,\nabla u,\ldots \xi,\nabla \xi,\ldots)$,
let us \textit{formally} Taylor expand
\begin{equ}
F(u,\xi) = \sum_{\alpha} \frac{D^\alpha F(0)}{\alpha!}(u,\xi)^k\;.
\end{equ}
Then each iteration
\[
u^{(0)}=0\;,\qquad u^{(1)}=Gu_0\;,\qquad u^{(2)}=Gu_0 + G(F(u^{(1)},\xi))\;,\qquad \ldots
\]
can be formally written as a series of functions which are multilinear in the jets of $Gu_0$ and $\xi$, and involve iterated convolutions with $G$.
For example, one of the functions in $u^{(3)}$ could be a multiple of $G(Gu_0 (G*\partial_i\xi)^2\xi)$.

This naturally leads to a collection of `abstract symbols' (trees) which will be used as placeholders for these functions.
This abstraction helps to detach the terms that appear in this `formal' fixed point with the concrete functions these trees will represent.
% (such as $v^{:2:}w$ in Section~\ref{subsubsec:d4}).

\begin{definition}\label{def:types}
Let $\mfL = \{\Xi,\mcI\}$ be a $2$-element set which we call the set of \textit{types},
and define $\mcE = \mfL\times \N^d$, which we call the set of \textit{edge types}, where $\N=\{0,1,\ldots\}$.
We will often write $(\mft,p)\in\mcE$ as $\mft_p$
and $(\mft,0)=\mft_0\in\mcE$ as just $\mft$.
\end{definition}
One should think of $\Xi$ and $\mcI$ above as formal symbols corresponding to noise and integration respectively.
\begin{definition}\label{def:trees_mfT}
Consider a tuple $\tau= (T,\mfn,\mfe)$ where
\begin{itemize}
\item $T$ is a rooted tree, i.e. a connected graph without cycles, with node set $N_T=N_\tau$, edge set $E_T=E_\tau$,
and a distinguished node $\rho=\rho_\tau\in N_T$ called the root;

\item $\mfe\colon E_T\to \mcE$, which we call the edge decoration;

\item $\mfn\colon N_T\to\N^d$, which we call the polynomial decoration.
\end{itemize}
We call the edges $e\in E_T$ with $\mfe(e) \in \{\mcI\}\times\N^d$ (resp. $\mfe(e) \in\{\Xi\}\times\N^d$) \textit{kernel} edges (resp. \textit{noise} edges).
We write $x\leq y$ to mean that $x$ is on the unique path from $y$ to $\rho$. 
If $(x,y)\in E_T$ is a noise edge with $x\leq y$, we call $y$ a \textit{noise node}.
 
Let $\mfT$\label{def:mfT} denote the set of all such tuples $\tau=T^\mfn_\mfe\eqdef (T,\mfn,\mfe)$ for which, if $y$ is a noise node, then
\begin{itemize}
\item $y$ is a leaf, i.e. $y$ has degree $1$, and

\item $\mfn(y)=0$.
\end{itemize}
We denote $\CV = \Span_\R(\mfT)$.\label{CV_page_ref}
\end{definition}

When we draw trees, we will always place the root at the bottom.
When important, we will write the decorations next to edges or nodes.
For example, the following trees are in $\mfT$:
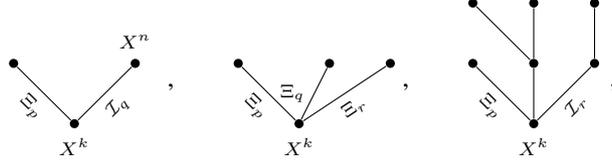
\begin{equ} 
\begin{tikzpicture}  [scale=0.8,baseline=-0.3cm] 
    \node [dot,label=-90:{\scriptsize $X^k$}] (m) at (-2, -1) {};
    \node [dot,label=90:{\scriptsize $X^n$}] (ml) at (-1, 0) {};
    \node [dot] (new) at (-3, 0) {};
    \draw (m) to node[sloped, midway,below]{\scriptsize $\CI_q$} (ml);
    \draw (m) to node[sloped, midway,below] {\scriptsize $\Xi_p$} (new);
\end{tikzpicture}
\;,\qquad
\begin{tikzpicture}  [scale=0.8,baseline=-0.3cm]
    \node [dot] (uu) at (-1.5, 0) {};
    \node [dot,label=-90:{\scriptsize $X^k$}] (m) at (-2, -1) {};
    \node [dot] (ml) at (-0.5, 0) {};
    \node [dot] (new) at (-3, 0) {};
    \draw (m) to node[sloped, midway,below]{\scriptsize $\Xi_r$} (ml);
\draw (m) to node[midway,left]{\scriptsize $\Xi_q$} (uu);
    \draw (m) to node[sloped, midway,below] {\scriptsize $\Xi_p$} (new);
\end{tikzpicture}
\;,\qquad
\begin{tikzpicture}  [scale=0.8,baseline=-0.3cm]
    \node [dot] (uu) at (-2, 0) {};
    \node [dot,label=-90:{\scriptsize $X^k$}] (m) at (-2, -1) {};
    \node [dot] (ml) at (-1, 0) {};
    \node [dot] (new) at (-3, 0) {};
    \draw (m) to node[sloped, midway,below]{\scriptsize $\mcI_r$} (ml);
\draw (m) to  (uu);
\node [dot] (uur) at (-2,1) {};
\node [dot] (uul) at (-3,1) {};
\node [dot] (uurr) at (-1,1) {};
\draw (uu) to (uur);
\draw (uu) to (uul);
\draw (ml) to (uurr);
    \draw (m) to node[sloped, midway,below] {\scriptsize $\Xi_p$} (new);
\end{tikzpicture}
\;,
\end{equ}
where unmarked edges have decoration $\mcI_0$ and unmarked nodes have decoration $X^0$.
However, for $k\neq 0$, the following two trees are not in $\mfT$ because the first has a noise node $y$ with $\mfn(y)=k\neq 0$, while the second has a noise node $y$ that is not a leaf:
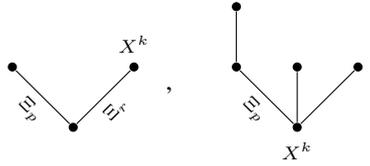
\begin{equ}
\begin{tikzpicture}  [scale=0.8,baseline=-0.3cm]
    \node [dot] (m) at (-2, -1) {};
    \node [dot,label=90:{\scriptsize $X^k$}] (ml) at (-1, 0) {};
    \node [dot] (new) at (-3, 0) {};
    \draw (m) to node[sloped, midway,below]{\scriptsize $\Xi_r$} (ml);
    \draw (m) to node[sloped, midway,below] {\scriptsize $\Xi_p$} (new);
\end{tikzpicture}
\;,\qquad
\begin{tikzpicture}  [scale=0.8,baseline=-0.3cm]
    \node [dot] (uu) at (-2, 0) {};
    \node [dot,label=-90:{\scriptsize $X^k$}] (m) at (-2, -1) {};
    \node [dot] (ml) at (-1, 0) {};
    \node [dot] (new) at (-3, 0) {};
    \draw (m) to  (ml);
\draw (m) to  (uu);
\node [dot] (uul) at (-3,1) {};
\draw (new) to (uul);
    \draw (m) to node[sloped, midway,below] {\scriptsize $\Xi_p$} (new);
\end{tikzpicture}
\;.
\end{equ}

\begin{remark}
These trees are \textit{combinatorial}, meaning that
we do not place an order on the edges leaving a node;
we thus identify any two trees which differ by a graph isomorphism that preserves the roots and decorations.
For example, the following two trees are treated as equal:
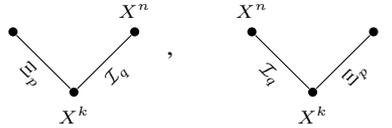
\begin{equ} 
\begin{tikzpicture}  [scale=0.8,baseline=-0.3cm] 
    \node [dot,label=-90:{\scriptsize $X^k$}] (m) at (-2, -1) {};
    \node [dot,label=90:{\scriptsize $X^n$}] (ml) at (-1, 0) {};
    \node [dot] (new) at (-3, 0) {};
    \draw (m) to node[sloped, midway,below]{\scriptsize $\CI_q$} (ml);
    \draw (m) to node[sloped, midway,below] {\scriptsize $\Xi_p$} (new);
\end{tikzpicture}
\;,\qquad
\begin{tikzpicture}  [scale=0.8,baseline=-0.3cm] 
    \node [dot,label=-90:{\scriptsize $X^k$}] (m) at (-2, -1) {};
    \node [dot] (ml) at (-1, 0) {};
    \node [dot,label=90:{\scriptsize $X^n$}] (new) at (-3, 0) {};
    \draw (m) to node[sloped, midway,below]{\scriptsize $\Xi_p$} (ml);
    \draw (m) to node[sloped, midway,below] {\scriptsize $\CI_q$} (new);
\end{tikzpicture}
\;.
\end{equ}
\end{remark}
\begin{convention}[Multisets vs. functions]\label{rem:multisets}
For a set $A$,
we will frequently identify functions in $\N^A$
with multisets of $A$ in the obvious way.
In particular, the zero element $0\in\N^A$ corresponds to the empty multiset $\emptyset$.
Likewise $k\in \N^d$ will often be treated as a multiset of $[d]\eqdef\{1,\ldots, d\}$\label{[d]_page_ref}.
\end{convention}
We will write every tree $\tau\in\mfT$ symbolically as follows.
Let $\mfN \subset \N^{\N^d}$ denote the set of functions $\N^d\to \N$ with finite support.\label{mfN_page_ref}
According to Remark~\ref{rem:multisets},
we will implicitly treat every $L\in\mfN$ as a multiset of $\N^d$ with finitely many elements.
We then write every $\tau\in\mfT$ in one of the following two ways
\begin{equ}\label{eq:general_tree}
\tau=  X^k \Xi_{L}\prod_{j\in J} \mcI_{m_j}[\tau_j] = X^k \prod_{l\in L}\Xi_{l}\prod_{j\in J} \mcI_{m_j}[\tau_j]\;,
\end{equ}
where $k, m_j \in \N^d$, $L\in\mfN$, and $\tau_j\in\mfT$.
Here $k=\mfn(\rho_\tau)$ is the polynomial decoration at the root,
$L$ is the multiset such that there are $\# L$ noise edges incident on $\rho_\tau$ and have edge decorations $\{(\Xi,l)\}_{l\in L}$,
and $J$ is a finite (possibly empty) index set (say $J\subset \N$ for concreteness)
such that there are $\# J$ kernel edges
incident on $\rho_\tau$ and have edge decorations  $\{(\mcI,m_j)\}_{j\in J}$ and terminate respectively on the roots of the trees $\{\tau_j\}_{j\in  J}\subset\mfT$.
This representation is unique up to changing the index set $J$ in a way that the multiset $\{(m_j,\tau_j)\}_{j\in J}\subset\N^d\times\mfT$ remains unchanged.

\begin{example}
\begin{equ}
\begin{tikzpicture}  [scale=0.8,baseline=-0.3cm]
    \node [dot,label=90:{\scriptsize $X^b$}] (uu) at (-2, 0) {};
    \node [dot,label=-90:{\scriptsize $X^a$}] (m) at (-2, -1) {};
    \node [dot] (ml) at (0, 0) {};
    \node [dot] (new) at (-4, 0) {};
    \draw (m) to node[sloped, midway,below]{\scriptsize $\Xi_\beta$} (ml);
\draw (m) to node[midway,left]{\scriptsize $\mcI_p$} (uu);
\node [dot] (uur) at (-1,1) {};
\node [dot,label=90:{\scriptsize $X^c$}] (uul) at (-3,1) {};
%\node [dot] (uurr) at (-1,1) {};
\draw (uu) to node[sloped, midway,below]{\scriptsize $\Xi_\gamma$} (uur);
\draw (uu) to node[sloped, midway,below]{\scriptsize $\mcI_q$} (uul);
%\draw (ml) to (uurr);
    \draw (m) to node[sloped, midway,below] {\scriptsize $\Xi_\alpha$} (new);
\end{tikzpicture}
\quad = \quad
X^a \Xi_\alpha \Xi_\beta \mcI_p[X^b\Xi_\gamma \mcI_q[X^c]]\;.
\end{equ}
\end{example}
Notation~\eqref{eq:general_tree} is suggestive of the fact that there is a natural associative and commutative product on $\mfT$.
Indeed, we define for $\tau,\bar\tau\in\mfT$ the \textit{tree product} $\tau\bar \tau\in\mfT$ by joining the roots of $\tau$ and $\bar\tau$
and adding the polynomial decorations at the roots (all other decorations remain unchanged).
In symbols, the tree product is written
\begin{equ}
\Big(X^k \Xi_L \prod_{j\in J} \mcI_{m_j} [\tau_j]\Big)
\Big(X^{\bar k} \Xi_{\bar L} \prod_{j\in \bar J} \mcI_{m_j} [\tau_j]\Big)
\eqdef X^{k+\bar k} \Xi_{L+\bar L}\prod_{j\in J\sqcup \bar J} \mcI_{m_j} [\tau_j]\;,
\end{equ}
where we suppose $J,\bar J$ are disjoint.
For example,
\begin{equ} 
\begin{tikzpicture}  [scale=0.8,baseline=-0.3cm] 
    \node [dot,label=-90:{\scriptsize $X^k$}] (m) at (-2, -1) {};
    \node [dot,label=90:{\scriptsize $X^n$}] (ml) at (-1, 0) {};
    \node [dot] (new) at (-3, 0) {};
    \draw (m) to node[sloped, midway,below]{\scriptsize $\CI_q$} (ml);
    \draw (m) to node[sloped, midway,below] {\scriptsize $\Xi_p$} (new);
\end{tikzpicture}
\quad
\begin{tikzpicture}  [scale=0.8,baseline=-0.3cm]
    \node [dot] (uu) at (-2, 0) {};
    \node [dot,label=-90:{\scriptsize $X^m$}] (m) at (-2, -1) {};
    \node [dot] (ml) at (-1, 0) {};
%    \node [dot] (new) at (-3, 0) {};
    \draw (m) to node[sloped, midway,below]{\scriptsize $\Xi_r$} (ml);
\draw (m) to  (uu);
\node [dot] (uur) at (-1.5,1) {};
\node [dot] (uul) at (-2.5,1) {};
%\node [dot] (uurr) at (-1,1) {};
\draw (uu) to (uur);
\draw (uu) to (uul);
%\draw (ml) to (uurr);
%    \draw (m) to node[sloped, midway,below] {\scriptsize $\Xi_q$} (new);
\end{tikzpicture}
\quad = \quad 
\begin{tikzpicture}  [scale=0.8,baseline=-0.3cm]
    \node [dot] (uu) at (-1, 0) {};
    \node [dot,label=-90:{\scriptsize $X^{k+m}$}] (m) at (-2, -1) {};
    \node [dot] (ml) at (0, 0) {};
    \node [dot] (new) at (-4, 0) {};
    \draw (m) to node[sloped, midway,below]{\scriptsize $\Xi_r$} (ml);
\draw (m) to  (uu);
\node [dot] (uur) at (-0.5,1) {};
\node [dot] (uul) at (-1.5,1) {};
%\node [dot] (uurr) at (0,1) {};
\draw (uu) to (uur);
\draw (uu) to (uul);
%\draw (ml) to (uurr);
\node [dot,label=90:{\scriptsize $X^n$}] (poly) at (-2, 0) {};
\draw (m) to node[ midway,left] {\scriptsize $\mcI_q$} (poly);
    \draw (m) to node[sloped, midway,below] {\scriptsize $\Xi_p$} (new);
\end{tikzpicture}
\end{equ}
Note that the tree product is additive in the number of edges but not the number of nodes.

The unit under the product is $\bone\in\mfT$,\label{bone_page_ref} the tree with one node $\rho$ (the root) and zero polynomial decoration $\mfn(\rho)=0$.
We will correspondingly identify $\bone = X^0 = \Xi_\emptyset$ (but $\bone \neq \Xi_{\{0\}}$!) and  $\bone=\prod_{j\in J} \mcI_{m_j}[\tau_j]$ if $J$ is empty.
The products in~\eqref{eq:general_tree} should thus be interpreted as tree products.
For example, $X^k$ will denote the tree with one node $\rho$ and polynomial decoration $\mfn(\rho)=k$.
The notation $
\Xi_L = \prod_{l\in L} \Xi_l
$
in~\eqref{eq:general_tree}
is consistent with the interpretation of the product as a tree product.
Here we use the shorthand $\Xi_l=\Xi_{\{l\}}$ to denote the tree
having a single noise edge with decoration $(\Xi,l)\in\mcE$ (in particular $\Xi_0\neq\Xi_\emptyset=\bone$ under this notation).
\begin{remark}
Given a smooth function $\xi\in C^\infty(\R^d,\R)$,
there is a canonical way to map every tree in $\mfT$ to a smooth function (see $\bfPi$ in Section~\ref{sec:pos_renorm}).
The tree $X^k$ corresponds to a polynomial $x^k$ (appearing, e.g. in the jet of $Gu_0$).
The tree $\Xi_{L}$ for $L\in\mfN$ corresponds to the function
\begin{equ}\label{eq:noise_description}
\prod_{l \in L} D^l \xi
\end{equ}
(we are identifying here $L$ with a finite multiset of $\N^d$).
The tree $\mcI_{m_j}[\tau_j]$ corresponds to the convolution of $D^{m_j} G$ with the function corresponding to $\tau_j$.
Finally, the tree product corresponds to the product of functions.
See, however, Section~\ref{sec:neg_renorm} where we reinterpret (renormalise) products.
\end{remark}
Most trees in $\mfT$ will not be useful in solving~\eqref{eq:SPDE}.
For example, if $\xi$ appears in $F$ in an affine way, then products of the form~\eqref{eq:noise_description} corresponding to $\Xi_L$ with $\#L>1$ never appear in the Picard iteration.

The way we select the relevant trees is through the notion of \textit{rules}.

\subsection{Rules}

\begin{definition}\label{def:rule}
A \textit{rule} $R$ is a subset $R\subset \N^{\mcE}$, i.e. a subset of the set of multisets of edge types.
We say that $R$ is \textit{normal} if $A\subset B$ and $B\in R$ implies $A\in R$.
\end{definition}

\begin{definition}\label{def:conform}
For a tree $\tau = T^{\mfn}_{\mfe}\in\mfT$ and $x\in N_\tau$,
define the multiset
\begin{equ}
\CN(x) \eqdef \{\mfe(x,y) \,:\, (x,y)\in E_\tau\,,\, x\leq y\} \in \N^{\CE}\;.
\end{equ}
We also denote
\begin{equ}\label{eq:mcN}
\mcN(\tau) \eqdef \mcN(\rho_\tau) = \{\Xi_{l}\,:\,l\in L\}
\sqcup \{\mcI_{m_j} \,:\, j \in J\}\;,
\end{equ}
where the final equality assumes that $\tau$ is of the form~\eqref{eq:general_tree}.

We say that a tree $\tau\in\mfT$ \textit{strongly conforms} to a rule $R$
if $\mcN(x)\in R$ for all $x\in N_\tau$.
We denote by $\mfT_\circ = \mfT_\circ(R)$ the set of trees that strongly conform to $R$.
\end{definition}
Equivalently, $\tau$ strongly conforms to $R$ if $\CN(\tau)\in R$ and, when written as~\eqref{eq:general_tree}, every $\tau_j$ for $j\in J$ strongly conforms to $R$.
\begin{remark}
Our definition of a rule is a simplified version of~\cite[Def.~5.7]{BHZ19} -- therein a rule $R$ is a map from the set of types $\mfL$ into subsets of $\N^\mcE$, and our definition of a rule is simply $R(\mcI)$ in the notation of~\cite{{BHZ19}}.
We can make this simplification because we only consider two types, namely $\Xi$ and $\mcI$,
and every noise node of a tree $\tau\in\mfT$ is a leaf by assumption.
\end{remark}

\begin{example}[Trivial cases]
\begin{itemize}
\item If $R$ is empty, then $\mfT_\circ$ is also empty since $\CN(x)$ is never an element of $R$ for any $x\in N_\tau$ and $\tau\in\mfT$.

\item If $R$ does not contain $0\in \N^{\mcE}$, then $\mfT_\circ$ is again empty because every $\tau\in\mfT$ contains at least one $x\in N_\tau$ for which $\CN(x) = 0$ (i.e. $x$ is a leaf).

\item If $R$ is the singleton multiset $R = \{0\}$, then no trees in $\mfT$ that contain an edge strongly conform to $R$, and thus $\mfT_\circ=\{X^k\,:\, k\in \N^d\}$.
\end{itemize}
\end{example}

The idea behind a rule $R$ is that the trees appearing in the Picard iterations outlined above should all strongly conform to $R$.
The following definition makes this precise.

\begin{definition}\label{def:obey}
A \textit{non-linearity} $F$ is a smooth map $F\colon \R^{\mcE}\to \R$
which depends on only finite many components in $\mcE$.
Given a non-linearity $F$, we say that $F$ \textit{obeys} $R$ if
\[
\{o_1,\ldots,o_n\}\notin R \quad \Rightarrow \quad D_{o_1}\ldots D_{o_n} F = 0\;.
\]
\end{definition}
Note that $\R^\mcE$ is the space of all possible `jets' of pairs of functions $(u,\xi)$ at a given point.
\begin{remark}\label{rem:BCCH21_diff_1}
In~\cite{BCCH21}, non-linearities are treated only as functions of the solution $u$ to~\eqref{eq:SPDE} and were explicitly polynomial in the noise; see also Remark~\ref{rem:BCCH21_diff_2}.
We take a more general approach here and allow (at the algebraic level) an arbitrary dependence on $\xi$ as in~\eqref{eq:SPDE} -- this approach was taken in~\cite{CCHS20} and treats the solution and the noise on a more equal footing.
However, conditions that we later impose (subcriticality) do effectively enforce polynomial dependence on the noise.
\end{remark}

Given a non-linearity $F$, we can always find a `minimal' normal rule such that $F$ obeys $R$.
This is most easily illustrated through examples
(a general construction can be found in~\cite[p.~942]{BCCH21}).

\begin{example}\label{ex:Phi43_rule}
Consider again the dynamical $\Phi^4_{d-1}$ equation~\eqref{eq:Phi4d} posed on $\R^{d}$
\[
(\partial_t -\Delta)u = -u^3 + \xi\;.
\]
Define the rule
\[
R = \{ ([\mcI]_\ell)\,,\, (\Xi)\,:\, \ell\in\{0,1,2,3\} \}\;,
\]
where $[\mcI]_\ell$ denotes the multiset with $\ell$ copies of $\mcI$.
Then $R$ is normal and the non-linearity
$F \eqdef \CY_\mcI^3 + \CY_\Xi$ obeys $R$,
where $\CY_o\colon \R^{\CE}\to \R$ for $o\in\CE$ denotes evaluation of the $o$-component.
\end{example}

\begin{example}\label{ex:KPZ_rule}
Consider the generalised KPZ equation posed on $\R^{1+1}$
\[
(\partial_t -\Delta)u = f(u) + g(u)(\partial_x u)^2 + h(u)\partial_x u + k(u)\xi\;,
\]
where $f,g,h,k\colon \R\to \R$ are smooth functions.
Define the rule
\[
R = \{ ([\mcI]_\ell, [\mcI_{1}]_m)\,,\, ([\mcI]_\ell,\Xi) \,:\, \ell\geq 0\,,\,m\in\{0,1,2\}\}
\]
where $\mcI_i$ is shorthand for $\mcI_{e_i}$, where $e_i\in\N^d$ is the function $e_i(i)=1$ and $e_i(j)=0$ for $j\neq i$.
Then $R$ is normal and the non-linearity
$F \eqdef f(\CY_\mcI) + g(\CY_\mcI)(\CY_{\mcI_{1}})^2+ h(\CY_\mcI)\CY_{\mcI_{1}} + k(\CY_\mcI)\CY_\Xi$
obeys $R$.
\end{example}
We will study non-linearities in more detail in Section~\ref{sec:renorm_SPDE}.

\subsection{Subcriticality}\label{subsec:subcrit}

We now come to the notion of subcriticality,
which is central in understanding the kind of equations~\eqref{eq:SPDE} that the theory of regularity structures is able to solve.

Fix henceforth a \textit{scaling} $\s=(\s_1,\ldots,\s_d)\in [1,\infty)^{d}$\label{scaling_page_ref}
and a \textit{degree map} $|\cdot|_\s\colon\mfL\to \R$
for which $|\Xi|_\s<0$ and $|\mcI|_\s>0$.
We think of $|\Xi|_\s$ as the regularity of the noise under the scaling $\s$, i.e. $\xi\in C^{|\Xi|_\s}$ where $C^\alpha= C^\alpha_\s$ is suitably defined, see~\cite[Def.~3.7]{Hairer14},
and $|\mcI|_\s$ is the amount by which $(\partial_t-\mcL)^{-1}$ regularises under this scaling, e.g. $|\mcI|_\s=2$ for $\partial_t-\Delta$ with the natural scaling $(1,\ldots, 1,2)$.
We extend the degree map to each tree $\tau\in\mfT$ of the form~\eqref{eq:general_tree} inductively by\label{deg_page_ref}
\[
|\tau|_\s = |k|_\s + \sum_{l\in L} |\Xi_{l}|_\s  + \sum_{j\in J} (|\tau_j|_{\s} + |\mcI_{m_j}|_\s) \;,
\]
where for $k\in\N^{d}$ we define $|k|_s=\sum_{i=1}^d k_i\s_i$
and for $\mft_k\in\mcE$ we define $|\mft_k|_\s = |\mft|_\s-|k|_\s$.
Remark that $|\tau\sigma|_\s = |\tau|_\s+|\sigma|_\s$ for all $\tau,\sigma\in\mfT$.

\begin{remark}\label{rem:motivation_degree}
At this stage, the motivation behind $|\tau|_\s$ is not completely clear.
It is certainly \textit{not} the regularity of the `canonical' function associated to $\tau$ (see the function $\bfPi$ in Section~\ref{sec:pos_renorm}).
It rather captures how this function behaves locally around a point $x\in\R^{d}$ once a suitable Taylor jet is subtracted, e.g. $|\tau|_\s=-1/2-$ for the tree $\tau$ corresponding to $v^2w$ in
 Remark~\ref{rem:motivation}.
\end{remark}
The kind of equations we are able to address are \textit{subcritical} in the following sense.
\begin{definition}
A rule $R$ is called \textit{subcritical} if there exists a map $\reg\colon \mfL \to \R$ such that
\[
\reg(\Xi)<|\Xi|_\s\;,\quad \reg(\mcI) < |\mcI|_\s +\inf_{N\in R} \reg(N)\;,
\]
where $\reg(N) = \sum_{(\mft,k)\in N} \reg(\mft) - |k|_\s$.
\end{definition}
We should think of $\reg(\Xi)$ and $\reg(\mcI)$ as a lower bound on the regularity of $\xi$ and $u$ respectively.
We now analyse the cases in which the rules defined in Examples~\ref{ex:Phi43_rule} and~\ref{ex:KPZ_rule} are subcritical, where $|\Xi|_\s=-\frac d2 - \frac12-\kappa$ is the white noise regularity, and $|\mcI|_\s=2$ is the gain in regularity of the heat kernel.
We recall here that $\kappa>0$ is a small parameter.
\begin{example}\label{ex:Phi43_subcritical}
Following Example~\ref{ex:Phi43_rule} with $d< 5$,
define $\reg(\Xi)=-\frac{d}{2} - \frac12 -2\kappa$ and $\reg(\mcI)=-\frac12-3\kappa$.
Then evidently $\reg(\Xi)<|\Xi|_\s
=-\frac{d}{2} - \frac12 -\kappa$.
Moreover, the $\inf$ in the definition of $\reg$ for $\mcI$ is achieved for $N=(\Xi)\in R$ (provided $\kappa$ is sufficiently small).
Hence $R$ is subcritical since $\reg(\mcI) < 2+ \reg(\Xi)$.

If the noise had scaling $|\Xi|_\s < -3$
(i.e. $(4+1)$-dimensional white noise)
then the rule is no longer subcritical.
Indeed, on the one hand, subcriticality would imply $\reg(\mcI)<2+\reg(\Xi)<-1$,
while on the other hand $\reg(\mcI) < 2+3\reg(\mcI) \Leftrightarrow \reg(\mcI)>-1$.
\end{example}

\begin{example}
Following Example~\ref{ex:KPZ_rule},
define $\reg(\Xi)=-\frac32-2\kappa$ and $\reg(\mcI)=\frac12-3\kappa$.
Then evidently $\reg(\Xi)<|\Xi|_\s=-\frac32-\kappa$.
Moreover, the $\inf$ in the condition for $\reg(\mcI)$ is achieved again for $N=(\Xi)\in R$ (provided $\kappa$ is sufficiently small).
Hence $R$ is subcritical since $\reg(\mcI) < 2+ \reg(\Xi)=\frac12-2\kappa$.
\end{example}

\begin{exercise}
Show that, if $|\Xi|_\s \leq -2$ in Example~\ref{ex:KPZ_rule},
then the rule is no longer subcritical.
(Note that $|\Xi|_\s < -2$ corresponds to the regularity of $(2+1)$-dimensional white noise.)
\end{exercise}
The main consequence of subcriticality is the following.
\begin{lemma}\label{lem:finite_trees}
Suppose $R$ is subcritical.
Then, for every $\gamma\in\R$, the set $\{\tau\in\mfT_\circ\,:\, |\tau|_\s<\gamma\}$ is finite.
\end{lemma}

\begin{proof}
See~\cite[Proposition~5.15]{BHZ19}.
\end{proof}

\begin{definition}\label{def:mcT}
Given a subcritical rule $R$,
the \textit{model space} generated by $R$
is the linear space
\[
\mcT = \Span_\R (\mfT_\circ) \subset \mcV\;.
\]
For $\gamma\in\R$, we denote
\[
\mcT_{\gamma} = \Span_\R \{\tau\in\mfT_\circ\,:\,|\tau|_\s = \gamma\}\;,
\quad
\mcT_{<\gamma} = \Span_\R \{\tau\in\mfT_\circ\,:\,|\tau|_\s < \gamma\}\;.
\]
The set $A \eqdef \{|\tau|_\s \,:\, \tau\in\mfT_\circ\}$ is called the \textit{index set}.\label{A_page_ref}
\end{definition}
By Lemma~\ref{lem:finite_trees}, $\mcT_{<\gamma}$ is a finite dimensional subspace of $\mcT$ for any $\gamma\in\R$.
\section{Positive renormalisation}
\label{sec:pos_renorm}

Fix henceforth a normal subcritical rule $R$ and recall that we have fixed a scaling $\s\in [1,\infty)^d$ and degree map $|\cdot|_\s\colon \mfL\to\R$ in Section~\ref{subsec:subcrit}.

Consider a smooth function $K\colon \R^d\setminus\{0\}\to\R$\label{K_page_ref}
with support in $\{0<|z|<1\}$
and
which agrees with the Green's function $G$ of $\partial_t-\mcL$ on $\{0<|z|\leq \frac12\}$.
We should think of $K*f$ as an approximation to $G*f$ modulo a smooth function such that $K*f(x)$ depends on $f$ only in the unit ball around $x$.
In particular, we should think of convolution by $K$ as increasing regularity by $|\mcI|_\s$.

Given a smooth function $\xi\colon\R^{d}\to \R$,
a straightforward way to associate a function to every tree in $\mfT_\circ$ (even to every tree in $\mfT$) is to define a linear map $\bfPi\colon\mcT \to C^\infty(\R^d,\R)$
recursively by
\begin{equs}\label{eq:canonical1}
\bfPi X^k &= \cdot^k\;, &\qquad\bfPi \Xi_l &= D^l \xi\;,
\\
\bfPi\tau\sigma
&=(\bfPi\tau) (\bfPi\sigma)\;,
&\qquad
\bfPi\mcI_m[\tau] &= D^mK*\bfPi\tau\;,\label{eq:canonical2}
\end{equs}
where $k,l,m\in\N^d$.
Recalling the notation~\eqref{eq:general_tree}, one readily sees that~\eqref{eq:canonical1}-\eqref{eq:canonical2} determine $\bfPi(\tau)$ for every $\tau\in\mfT$ uniquely.
The map $\bfPi$ is called the \textit{canonical lift} of $\xi$.

We now recentre $\bfPi$ around every point $x\in\R^{d}$ and get a new linear map $\Pi_x\colon\mcT\to C^\infty$
such that $|\Pi_x\tau (y)| \lesssim |y-x|_\s^{|\tau|_\s}$.
Following Remarks~\ref{rem:motivation} and~\ref{rem:motivation_degree}, this property interplays nicely with the additive nature of degree since then $(\Pi_x\tau\sigma) (y) \lesssim |y-x|_\s^{|\tau|_\s + |\sigma|_\s}$.
In particular the number of functions vanishing slower than rate $\gamma$
is now finite thanks to subcriticality (Lemma~\ref{lem:finite_trees}).

The definition of $\Pi_x$ is also recursive and is given by
\begin{align*}
\Pi_x X^k&=(\cdot - x)^k\;, \quad \Pi_x \Xi_l = D^l \xi\;,\quad 
\\
\Pi_x\tau\sigma
&=(\Pi_x\tau) (\Pi_x\sigma)\;,
\end{align*}
together with
\begin{align}
\Pi_x\mcI_m[\tau]
&= D^mK*\Pi_x\tau - \sum_{|k|_\s <|\mcI_m\tau|_\s}\frac{(\cdot-x)^k}{k!}  D^{k+m} K*\Pi_x\tau(x)\label{eq:Pi_x_Taylor_def}
\\
&= D^m K*\Pi_x\tau + \sum_{|k|_\s<|\mcI_m\tau|_\s} \frac{\cdot^k}{k!} f_x(\mcI_{k+m}\tau)
\;,\notag
\end{align}
where we use the shorthand $k!=(k_1)!\ldots (k_d)!$ for $k\in\N^d$ and
\begin{equation}\label{eq:pos_char_def}
f_x (\mcI_n\tau)
\eqdef -\sum_{|\ell|_\s<|\mcI_n\tau|_\s}
\frac{(-x)^\ell}{\ell!} D^{n+\ell} K*\Pi_x\tau(x)\;.
\end{equation}
We see that the definition of $\Pi_x$ is the same as that of $\bfPi$ except that we recentred the polynomials around $x$ and subtracted off the Taylor jet of $D^m K*\Pi_x\tau$
to see a decay of order $|\mcI_m \tau|_{\s}$ around $x$.

The point of this section is to describe algebraically the map taking $\bfPi$ to $\Pi_x$ and vice versa.
The following is our first result in this direction.
%, which will be particularly helpful when we need to redefine what we mean by `products' (`negative' renormalisation).

\begin{proposition}
There exists a linear map $F_x\colon\mcT\to\mcT$ such that $\Pi_x = \bfPi F_x$.
Furthermore, for every tree $\tau\in\mfT_\circ$, $F_x\tau$ is a linear combination of trees $\tau_i$
with $\mcN(\tau_i) \subset \mcN(\tau)$,
where $\mcN(\tau)$ is defined in~\eqref{eq:mcN}.
\end{proposition}

\begin{proof}
Define for $l\in\N^d$ and $i=1,\ldots, d$
\begin{equation}\label{eq:init}
F_x\Xi_l=\Xi_l\;,\quad \text{ and }
\quad
F_xX^i = (X-x)^i\;.
\end{equation}
We now proceed by induction.
Assume $F_x$ is defined on trees $\tau,\sigma$
and satisfies the claimed properties.
Suppose that $\tau\sigma\in\mcT$.
Then
\begin{equation}\label{eq:multiplicative_identities}
\Pi_x \tau\sigma =(\Pi_x\tau) (\Pi_x\sigma) = (\bfPi F_x\tau) (\bfPi F_x\sigma) = \bfPi (F_x\tau) (F_x\sigma)\;.
\end{equation}
So if we set
\begin{equation}\label{eq:mult}
F_x\tau\sigma\eqdef (F_x\tau) (F_x\sigma)\;,
\end{equation}
then we indeed obtain $\Pi_x\tau\sigma = \bfPi F_x(\tau\sigma)$.
Note that $(F_x\tau) (F_x\sigma)\in\mcT$ due to normality of the rule $R$ and the induction hypothesis on $\tau,\sigma$. 
Suppose moreover $\mcI_m[\tau]\in\mcT$.
Then
\begin{align*}
\Pi_x \mcI_m [\tau]
&= D^m K*\Pi_x\tau + \sum_{|k|_\s<|\mcI_m\tau|_\s}\frac{\cdot^k}{k!} f_x(\mcI_{k+m}\tau)
\\
&=
D^m K*\bfPi F_x \tau
+ \sum_{|k|_\s<|\mcI_m\tau|_\s} \bfPi\frac{X^k}{k!} f_x(\mcI_{k+m}\tau)
\\
&=\bfPi \Big(\mcI_m[F_x \tau]
+ \sum_{|k|_\s<|\mcI_{m}\tau|_\s} \frac{X^k}{k!} f_x(\mcI_{k+m}\tau)
\Big)\;.
\end{align*}
So if we set
\begin{equation}\label{eq:integ}
F_x\mcI_m [\tau] = \mcI_m[F_x \tau] +
\sum_{|k|_\s<|\mcI_m\tau|_\s} \frac{X^k}{k!} f_x[\mcI_{k+m}\tau]\;,
\end{equation}
then we indeed obtain $\Pi_x\mcI_m[\tau] = \bfPi F_x\mcI_m[\tau]$.
Observe again that $\mcI_m[F_x\tau]\in\mcT$ due to the induction hypothesis.
Therefore, we can take~\eqref{eq:mult} and~\eqref{eq:integ} as an inductive definition of $F_x$, so that we have $\Pi_x = \bfPi F_x$.
\end{proof}
\begin{remark}
In Section~\ref{sec:neg_renorm}, we will consider functions $\Pi_x$ and $\bfPi$ for which $\bfPi\tau\sigma\neq(\bfPi\tau)(\bfPi\sigma)$ and $\Pi_x\tau\sigma \neq (\Pi_x\tau)(\Pi_x\sigma)$ due to negative renormalisation.
In this case, everything in the above construction will remain intact \textit{except} the middle identities in~\eqref{eq:multiplicative_identities}.
In particular, we will still have a map $F_x$ such that $\bfPi F_x= \Pi_x$
and~\eqref{eq:init},~\eqref{eq:mult}, and~\eqref{eq:integ} will still hold.
\end{remark}
The above proof furthermore has the following important consequences.
\begin{itemize}
\item $F_x$ has the triangular form
\[
F_x\tau = \tau + h\;,\qquad h\in \mcT_{<|\tau|_\s}\;.
\] 
Consequently, $F_x$ is invertible.

\item The coefficients appearing in front of the trees in $h$
are products of numbers of the form $f_x[\mcI_{m}\tau]$ where $\tau\in\mfT_\circ$ and $|\mcI_{m}\tau|_\s>0$.
\end{itemize}
With these remarks, we can rewrite the above construction to make it algebraically more appealing.
Doing so will also reveal a further important property of $F_x$,
namely that the set of all linear maps defined recursively by~\eqref{eq:init},~\eqref{eq:mult} and~\eqref{eq:integ}
forms a group.
\begin{definition}\label{def:mfT_plus}
Let $\mfT_+\subset \mfT$ denote the set of all trees $\tau\in\mfT$ of the form
\begin{equ}
\tau = X^k\prod_{j\in J} \mcI_{m_j}[\tau_j]\;,
\end{equ} 
where $k\in \N^d$, $J$ is a finite index set (possibly empty), and $\tau_j\in\mfT_\circ$ and $|\mcI_{m_j}\tau_j|_\s>0$ for all $j\in J$.
Define $\mcT_+ =\Span_\R\mfT_+$.

Let $G_+$ denote the set of \textit{characters} on $\mcT_+$, i.e. the collection of linear maps $f\colon\mcT_+\to \R$ such that $f(\bone)=1$ and \begin{equation}\label{eq:pos_char_mult}
f(\tau\sigma)=f(\tau)f(\sigma)\;.
\end{equation}
\end{definition}
\begin{remark}\label{rem:pos_char_extension}
Observe that $\mcT_+$ is an (associative) algebra  with a commutative product (in contrast to $\mcT$ which is \textit{not} an algebra).
Furthermore $\mcT_+$ is freely generated by
the set
\[
\mfK_+ = \{X^i\,:\,i=1,\ldots, d\} \sqcup \{\mcI_{m}[\tau]\,:\, m\in\N^{d}\,,\,\tau\in\mfT_\circ\,,\,
|\mcI_{m}[\tau]|_\s>0\}\;.
\]
In particular, given any map $f\colon\mfK_+ \to \R$,
there exists a unique extension of $f$ to an element of $G_+$ by enforcing the multiplicative property~\eqref{eq:pos_char_mult} together with linearity.
\end{remark}
A natural way to achieve both~\eqref{eq:mult} and~\eqref{eq:integ} is to make the ansatz
\[
F_x = (\id \otimes f_x) \Deltap\;,
\]
where $\Deltap \colon \mcT\to \mcT\otimes \mcT_+$\label{Deltap_page_ref}
and $f_x\in G_+$ are suitable linear maps which we now define recursively.

To achieve~\eqref{eq:init}, we set for $l\in\N^d$ and $i=1,\ldots, d$
\begin{equ}\label{eq:Deltap_base_case}
\Deltap \bone = \bone\otimes\bone \;,\qquad \Deltap X^i = X^i\otimes \bone + \bone\otimes X^i\;,\quad \Deltap \Xi_l = \Xi_l\otimes \bone\;,
\end{equ}
together with $f_x(\bone)=1$ and $f_x(X^i)=-x_i$ for $i=1,\ldots, d$.

To achieve~\eqref{eq:mult},
we set
\begin{equ}\label{eq:Deltap_multiplicative}
\Deltap(\tau\sigma)=(\Deltap\tau)(\Deltap\sigma)\;,
\end{equ}
define $f_x[\mcI_m\tau]$ by~\eqref{eq:pos_char_def},
and then extend $f_x$ uniquely to an element of $G_+$ as in Remark~\ref{rem:pos_char_extension}.
Then
\begin{align*}
F_x(\tau\sigma)
&= (\id\otimes f_x)\Deltap(\tau\sigma) = (\id\otimes f_x)\Deltap(\tau)\Deltap(\sigma)
\\
&= (\id\otimes f_x)\Deltap(\tau)(\id\otimes f_x)\Deltap(\sigma) = (F_x\tau) (F_x\sigma)\;,
\end{align*}
where in the third equality we used the multiplicative property of $f_x$.

To achieve~\eqref{eq:integ}, we set
\begin{equ}\label{eq:Deltap_def}
\Deltap \mcI_m\tau = (\mcI_m\otimes \id) \Deltap\tau + \sum_{|k|_\s<|\mcI_m\tau|_\s}\frac{X^k}{k!}\otimes \mcI_{k+m}\tau\;,
\end{equ}
where we interpret $\mcI_m$ as a linear operator $\mcI_m \colon \mcT\to \mcT$ given by $\mcI_m\colon\tau\mapsto \mcI_m\tau$ for all $\tau\in\mfT$.
Then
\begin{align*}
F_x(\mcI_m\tau)
&= (\id\otimes f_x)\Deltap(\mcI_m\tau)
\\
&=
(\id\otimes f_x)(\mcI_m\otimes \id) \Deltap\tau + \sum_{|k|_\s<|\mcI_m\tau|_\s}\frac{X^k}{k!}\otimes f_x\mcI_{k+m}\tau
\\
&=
(\mcI_m\otimes \id)(\id \otimes f_x)\Deltap\tau
+ \sum_{|k|_\s<|\mcI_m\tau|_\s}\frac{X^k}{k!}\otimes f_x\mcI_{k+m}\tau
\\
&=\mcI_m F_x\tau + \sum_{|k|_\s<|\mcI_m\tau|_\s}
\frac{X^k}{k!}\otimes f_x\mcI_{k+m}\tau
\end{align*}
as desired.
\begin{exercise}
Verify that the map $\Deltap$ is well-defined in the sense that indeed $\Deltap$ maps $\mcT$ into $\mcT\otimes \mcT_+$.
\end{exercise}
In Section~\ref{subsec:global_pos}, we will see a global (non-recursive) characterisation of $\Deltap$.

\subsection{Positive Hopf algebra}

We next show that the space of \textit{all} maps appearing in the form $F = (\id\otimes f)\Deltap$,
where $f\in G_+$, is a group.
This will show that $F_x^{-1}$ is also of the form $(\id\otimes f)\Deltap$, as is $F_x^{-1}F_y$, which is an important property for moving from $\Pi_x$ to $\Pi_y$ due to the identity
$\Pi_y =
%\bfPi F_y = \bfPi F_x F_x^{-1}F_y=
\Pi_x F_x^{-1} F_y$.

To achieve this, we equip $\mcT_+$ with a \textit{Hopf} algebra structure.
Then the space of characters on $\mcT_+$ naturally becomes a group.

The form of the coproduct $\Deltap \colon \mcT_+\to\mcT_+\otimes\mcT_+$\label{Deltap_page_ref_2} is analogous to before: we set for $i=1,\ldots, d$ and $\mcI_m\tau\in\mfT_+$
\begin{equs}
\Deltap\bone &= \bone\otimes \bone\;,\quad \Deltap X^i = X^i\otimes \bone + \bone\otimes X^i\;,
\\
\Deltap \mcI_m\tau &= (\pi_+\mcI_m\otimes \id)\Deltap \tau + \sum_{|k|_\s<|\mcI_m\tau|_\s}\frac{X^k}{k!}\otimes \mcI_{m+k} \tau\;,\label{eq:Deltap_Tp}
\end{equs}
where $\pi_+\colon \CV\to\mcT_+$ is the projection which sends to $0$ every tree $\tau\in\mfT\setminus\mfT_+$,
and extend $\Deltap$ multiplicatively to $\mcT_+$, i.e. $\Deltap(\tau\sigma)=(\Deltap\tau)(\Deltap\sigma)$.

\begin{remark}
On the right-hand side of~\eqref{eq:Deltap_Tp},
$\tau$ is in $\mcT$ but not necessarily in $\mcT_+$ and so $\Deltap\tau$ is really what was defined for $\tau\in\mcT$ earlier.
\end{remark}

\begin{remark}\label{rem:abuse_positive}
Although we use the same symbol $\Deltap$ for the two maps $\Deltap\colon\mcT\to\mcT\otimes\mcT_+$
and $\Deltap\colon\mcT_+\to\mcT_+\otimes\mcT_+$,
these maps are actually different, and we should think of $\mcT_+$ and $\mcT$ as \textit{disjoint} for this purpose (rather than both embedded, say, into $\CV$).
This is because $\Deltap\tau$ can have different results depending on whether we interpret $\tau\in\mfT_\circ$ or $\tau\in\mfT_+$, even if $\tau\in\mfT_\circ\cap\mfT_+$
according to Definition~\ref{def:conform} and~\ref{def:mfT_plus}.

For example, suppose $\tau\eqdef\mcI_0[X^k\Xi_0] \in \mfT_+\cap\mfT_\circ$.
Then, if we interpret  $\tau\in\mfT_\circ$,
\begin{equ}
\Deltap\tau= \sum_{n\leq k}\binom kn \mcI_0[{X^n\Xi_0}]\otimes X^{k-n}
\end{equ}
while if we interpret $\tau\in\mfT_+$, then
\begin{equ}
\Deltap\tau= \sum_{n\leq k}\binom kn \pi_+\mcI_0[{X^n\Xi_0}]\otimes X^{k-n}\;,
\end{equ}
and the two expressions will differ if $|\mcI_0[\Xi_0]|_\s\leq0$.
We will always emphasise if $\Deltap$ has domain $\mcT$ or $\mcT_+$, so this abuse of notation will not cause confusion.
\end{remark}

\begin{definition}
For $\tau\in\mfT$, let $\fancynorm{\tau}$\label{kernel_count_page_ref} denote the number kernel edges in $\tau$.
\end{definition}

\begin{exercise}\label{ex:triangular}
\begin{enumerate}[label=(\alph*)]
\item\label{pt:Deltap_decompose_mfT_circ} Show that for all $\tau\in\mfT_\circ$
\begin{equ}
\Deltap\tau = \tau\otimes \bone+\sum_i c_i\tau_i^{(1)}\otimes \tau^{(2)}_i\;,
\end{equ}
where $c_i\in\R$ and $\tau^{(1)}_i\in\mfT_\circ$, $\tau^{(2)}_i\in\mfT_+$ with
\begin{equ}
|\tau_i^{(1)}|_\s<|\tau|_\s\;,\quad
|\tau_i^{(1)}|_\s+|\tau_i^{(2)}|_\s=|\tau|_\s\;,
\quad \text{and} \quad
\fancynorm{\tau^{(1)}_i}+ \fancynorm{\tau^{(2)}_i} = \fancynorm{\tau}\;.
\end{equ}

\item\label{pt:Deltap_decompose_mfT_+} Similarly, show that for all $\tau\in\mfT_+\setminus\{\bone\}$
\begin{equ}\label{eq:Deltap_decompose}
\Deltap\tau = \tau\otimes \bone + \bone\otimes\tau +\sum_i c_i\tau_i^{(1)}\otimes \tau^{(2)}_i\;,
\end{equ}
where $c_i\in\R$, and $\tau^{(1)}_i,\tau^{(2)}_i\in\mfT_+$ satisfy the conditions in~\ref{pt:Deltap_decompose_mfT_circ} and additionally
\begin{equ}
0<|\tau_i^{(1)}|_\s,|\tau_i^{(2)}|_\s<|\tau|_\s\;.
\end{equ}
\end{enumerate}
\end{exercise}

\begin{theorem}\label{thm:pos_Hop_alg}
There exists an involutory algebra morphism $\mcA_+\colon\mcT_+\to\mcT_+$ such that $(\mcT_+,\mcM_+, \Deltap,\mcA_+)$ is a Hopf algebra with antipode $\mcA_+$, unit $\bone$, and counit $\bone^*$, 
where $\mcM_+\colon\mcT_+\otimes\mcT_+\to\mcT_+$ is the product map $\mcM_+\colon\tau\otimes\sigma\mapsto \tau\sigma$
and $\bone^*\colon\CT_+\to\R$ is the linear map for which $\bone^*(\bone)=1$
and $\bone^*(\tau)=0$ for all $\tau\in\mfT_+\setminus\{\bone\}$.
Furthermore, $\mcT$ is a comodule over $\mcT_+$.

More explicitly, we have the identities
\begin{equ}\label{eq:Deltap_coasso}
(\id\otimes\Deltap)\Deltap = (\Deltap\otimes \id)\Deltap \;,
\end{equ}
as maps $\mcH\to \mcH\otimes \mcT_+\otimes\mcT_+$ for $\mcH\in\{\mcT,\mcT_+\}$,
and
\begin{equ}\label{eq:antipode_pos}
\mcM_+(\id\otimes \mcA_+)\Deltap = \mcM_+(\mcA_+\otimes \id)\Deltap = \bone\bone^*
\end{equ}
as maps $\mcT_+\to \mcT_+$.
\end{theorem}

\begin{proof}
We first prove~\eqref{eq:Deltap_coasso} for $\CH=\CT$.
The left- and right-hand sides of~\eqref{eq:Deltap_coasso} are clearly equal when applied to $\bone$, $X^i$ for $i=1,\ldots, d$,
and $\Xi_l$ for $l\in\N^d$.
Furthermore, if they are equal for $\tau, \sigma\in\mfT_\circ$ such that $\tau\sigma\in\mfT_\circ$,
then they are also equal on $\tau\sigma$ since
\begin{equs}
(\id\otimes\Deltap)\Deltap(\tau\sigma) &= (\id\otimes\Deltap)(\Deltap\tau)(\id\otimes\Deltap)(\Deltap\sigma)
\\
&=(\Deltap\otimes \id)(\Deltap\tau)(\Deltap\otimes \id)(\Deltap\sigma)
=(\Deltap\otimes \id)\Deltap(\tau\sigma)\;,
\end{equs}
where we used 
multiplicativity of $\Deltap$ in the first and third equalities.
By induction, it remains only to show that if the left- and right-hand sides of~\eqref{eq:Deltap_coasso} agree on $\tau\in\mfT_\circ$ with $\mcI_m\tau\in\mfT_\circ$, then they agree on $\mcI_m\tau$.
To this end,
\begin{equ}\label{eq:1st_coasso_Deltap}
(\id\otimes\Deltap)\Deltap\mcI_m\tau
= (\id \otimes \Deltap)(\mcI_m\otimes \id)\Deltap \tau + \sum_{|k|_\s<|\mcI_m\tau|_\s} \frac{X^k}{k!}\otimes\Deltap\mcI_{m+k}\tau\;.
\end{equ}
Observe that
\begin{equ}
(\id \otimes \Deltap)(\mcI_m\otimes \id)\Deltap \tau
=
(\mcI_m\otimes \id\otimes \id)(\id \otimes \Deltap)\Deltap \tau\;,
\end{equ}
while the final sum in~\eqref{eq:1st_coasso_Deltap} is equal to
\begin{equ}
\sum_{|k|_\s<|\mcI_m\tau|_\s} \frac{X^k}{k!}\otimes (\pi_+\mcI_{k+m}\otimes \id)\Deltap \tau
+
\sum_{|k|_\s+|\ell|_\s<|\mcI_m\tau|_\s}\frac{X^k}{k!}\otimes \frac{X^\ell}{\ell!}\otimes  \mcI_{m+k+\ell}\tau\;.
\end{equ}
%where we stop writing restrictions in the sums like $|k|_\s<|\mcI_m\tau|_\s$ but keep in mind that $\mcI_p\tau=0$ in the right-most factor whenever $\mcI_p\tau=0$  
On the other hand,
\begin{equ}\label{eq:2nd_coasso_Deltap}
(\Deltap\otimes\id)\Deltap\mcI_m\tau
= (\Deltap\otimes\id)(\mcI_m\otimes \id)\Deltap \tau + \sum_{|p|_\s<|\mcI_m\tau|_\s} \frac{\Deltap X^p}{p!}\otimes\mcI_{m+p}\tau\;.
\end{equ}
Using the identity $\Deltap X^p = \sum_{\ell+k=p}\binom{p}{k}X^k\otimes X^{\ell}$, we see that the final sum in~\eqref{eq:2nd_coasso_Deltap}
is equal to
\begin{equ}
\sum_{|k|_\s+|\ell|_\s<|\mcI_m\tau|_\s}\frac{X^k}{k!}\otimes \frac{X^\ell}{\ell!}\otimes  \mcI_{m+k+\ell}\tau\;.
\end{equ}
Furthermore, since $\Deltap\tau = \tau\otimes \bone+\sum_i c_i\tau_i^{(1)}\otimes \tau^{(2)}_i$ with $|\tau_i|_\s<|\tau|_\s$ (Exercise~\ref{ex:triangular}\ref{pt:Deltap_decompose_mfT_circ}),
it follows that
\begin{equs}\label{eq:tau_i_Deltap}
\Deltap \mcI_m \tau^{(1)}_i
&= (\mcI_m \otimes \id)\Deltap \tau^{(1)}_i + \sum_{|k|_\s<|\mcI_m\tau^{(1)}_i|_\s} \frac{X^k}{k!}\otimes \mcI_{m+k}\tau_i^{(1)}
\\
&= (\mcI_m \otimes \id)\Deltap \tau^{(1)}_i + \sum_{|k|_\s<|\mcI_m\tau|_\s} \frac{X^k}{k!}\otimes \pi_+\mcI_{m+k}\tau_i^{(1)}\;.
\end{equs}
Therefore
\begin{equs}
(\Deltap\otimes\id)(\mcI_m\otimes \id)\Deltap \tau
&=
(\mcI_m\otimes \id\otimes \id)(\Deltap\otimes\id)\Deltap \tau
\\
&\qquad+
\sum_{|k|_\s<|\mcI_m\tau|_\s} \frac{X^k}{k!}\otimes (\pi_+\mcI_{m+k}\otimes \id)\Deltap \tau\;.
\end{equs}
Combining the identities following the expressions~\eqref{eq:1st_coasso_Deltap} and~\eqref{eq:2nd_coasso_Deltap},
and the fact that $(\id\otimes \Deltap)\Deltap\tau=(\Deltap\otimes\id)\Deltap\tau$ by the inductive hypothesis, we obtain
\begin{equ}
(\id\otimes \Deltap)\Deltap\mcI_m\tau=(\Deltap\otimes\id)\Deltap\mcI_m\tau\;,
\end{equ}
which completes the proof of~\eqref{eq:Deltap_coasso} for $\CH=\CT$.
 
To conclude the proof of the theorem, it remains to show that
$\mcT_+$ is a Hopf algebra, i.e.~\eqref{eq:Deltap_coasso} for $\CH=\CT_+$ and the existence of an algebra morphism $\mcA_+$ satisfying~\eqref{eq:antipode_pos} and $\mcA_+\circ\mcA_+=\id$.
The proof of~\eqref{eq:Deltap_coasso} for $\CH=\CT_+$ is similar to (and uses)
the result for $\CH=\CT$, so we do not repeat the details
(the only difference is that, in the step~\eqref{eq:tau_i_Deltap},
one should use the identity \begin{equ}
\Deltap \pi_+ \mcI_m \tau^{(1)}_i
= (\pi_+\mcI_m \otimes \id)\Deltap \tau^{(1)}_i + \sum_{|k|_\s<|\mcI_m\tau^{(1)}_i|_\s} \frac{X^k}{k!}\otimes \mcI_{m+k}\tau^{(1)}_i\;,
\end{equ}
which is again implied by the triangular structure of $\Deltap$ from Exercise~\ref{ex:triangular}\ref{pt:Deltap_decompose_mfT_circ}).

The existence of the antipode $\mcA_+$ follows from Exercise~\ref{ex:antipode_plus} below.
%See~\cite[Theorem 8.16]{Hairer14} for a proof with a slightly different (but equivalent after a chance of basis) definition of $\Deltap$.
\end{proof}

\begin{exercise}\label{ex:antipode_plus}
\begin{enumerate}[label=(\alph*)]
\item\label{pt:antipode_def} Show that there exists an algebra morphism $\mcA_+\colon \mcT_+\to\mcT_+$ such that
\begin{equ}\label{eq:antipode_first_half}
\mcM_+(\mcA_+\otimes \id)\Deltap = \bone\bone^*\;.
\end{equ}
[Hint: define $\mcA_+$ inductively by $\mcA_+\bone=\bone$, $\mcA_+ X^i=-X^i$ for $i=1,\ldots, d$, and by enforcing $\mcA_+(\tau\bar\tau) = \mcA_+(\tau)\mcA_+(\bar\tau)$
together with~\eqref{eq:antipode_first_half}.
The fact that these conditions uniquely determine $\mcA_+$ follows the triangular structure of $\Deltap$ from Exercise~\ref{ex:triangular}\ref{pt:Deltap_decompose_mfT_+}
combined with double
induction on $\fancynorm{\tau}$ and $|\tau|_\s$.]

\item Let $\End(\CT_+)$ denote the set of algebra endomorphisms of $\CT_+$
\begin{equ}
\End(\CT_+) \eqdef \{\psi \in L(\CT_+,\CT_+) \,:\,
\psi(\bone)=\bone\,,\, \psi(\tau\bar\tau)=\psi(\tau)\psi(\bar\tau)\;\; \forall\,\tau,\bar\tau\in\CT_+
\}\;.
\end{equ}
Show that $\End(\CT_+)$ is a monoid with product
\begin{equ}
\psi\star\phi \eqdef \mcM_+(\psi\otimes\phi)\Deltap
\end{equ}
and identity element $\bone\bone^*$.
[Hint: use the fact that $\CT_+$ is a commutative algebra.]
%see~\cite[Chapter~4]{Sweedler}.]

\item Show that, for every $\psi\in \End(\CT_+)$, the map $\psi^{-1}\eqdef \psi\mcA_+\in\End(\CT_+)$ is a left inverse of $\psi$, i.e. that $\psi^{-1}\star \psi = \bone\bone^*$.
In particular, $\mcA_+$ is the left inverse of $\id$.

\item Suppose that $S$ is a monoid in which every element has a left inverse. Show that $S$ is a group.

\item Conclude that $\mcA_+$ is also a right inverse of $\id$ in $\End(\CT_+)$, i.e. that~\eqref{eq:antipode_pos} holds.

\item\label{pt:involutory} Show that $\CM_+[(\mcA_+\circ\mcA_+)\otimes \mcA_+]\Deltap = \bone\bone^*$
and conclude that $\mcA_+\circ\mcA_+=\id$, i.e. that $\mcA_+$ is an involution. 
\end{enumerate}
\end{exercise}
\begin{remark}
In Exercise~\ref{ex:antipode_plus}\ref{pt:antipode_def}, one could equally use the identity $\mcM_+(\id\otimes\mcA_+)\Deltap = \bone\bone^*$ to define $\mcA_+$,
with obvious modifications for the remainder of the exercise.
\end{remark}
\begin{remark}
Exercise~\ref{ex:antipode_plus}\ref{pt:involutory} also follows from the general fact that the antipode of a commutative Hopf algebra is an involution, see~\cite[Prop.~4.0.1]{Sweedler}.
It would be desirable to conclude the existence of $\mcA_+$ likewise from general principles, such as the fact that every
connected graded bialgebra admits a unique antipode turning it into a Hopf algebra.
(The strategy in Exercise~\ref{ex:antipode_plus} is based on the proof of this fact.)
However, it does not appear possible to equip $\CT_+$ with a graded structure under which it is connected, e.g. it is graded but not connected by the edge counting function.
\end{remark}
\begin{exercise}\label{ex:prod_G_+}
Using Theorem~\ref{thm:pos_Hop_alg}, show that $G_+$
is a group with product $f\circ \bar f \eqdef (f\otimes \bar f)\Deltap$
for which the inverse of every $f\in G_+$ is given by $f^{-1}=f\mcA_+$.
Furthermore, show that the map
\begin{equ}
G_+\ni f\mapsto \Gamma_f \eqdef (\id\otimes f)\Deltap\in \GL(\mcT)
\end{equ}
is a representation (a group homomorphism).
\end{exercise} 
\begin{definition}\label{def:reduced_RS}
We call $G_+$, or more precisely its image in $\GL(\mcT)$ under the mapping $f\mapsto \Gamma_f$, the \textit{structure group}.
Recalling the index set $A\subset\R$ from Definition~\ref{def:mcT},
the triple $(A,\mcT,G_+)$ is called the \textit{regularity structure} generated by the rule $R$.
\end{definition}
We now come the analytic statement
which is the
motivation behind the definition of $\Pi_x$.
Since we are primarily concerned with the algebraic theory, we do not provide a complete proof.
\begin{theorem}\label{thm:canonical_model}
For $\xi\in C^\infty(\R^d,\R)$, define $\Pi_x$ as above
and let
\begin{equ}
\Gamma_{xy}\eqdef F_x^{-1}F_y = \Gamma_{f_x^{-1}\circ f_y}\;.
\end{equ}
Fix any norm $|\cdot|_\gamma$ on $\mcT_{\gamma}$ for $\gamma\in A$.
Then, for all $\tau\in\mfT_\circ$, $|\Pi_x\tau(y)|\lesssim |y-x|_\s^{|\tau|_\s}$
and $|\Gamma_{xy}\tau|_{\gamma} \lesssim |x-y|_\s^{|\tau|_\s-\gamma}$ uniformly in $\gamma\leq |\tau|_\s$ and $x,y$ in any compact set of $\R^d$.
\end{theorem}
\begin{proof}
The fact that $|\Pi_x\tau (y)|\lesssim |x-y|_\s^{|\tau|_\s}$ follows readily from the multiplicative definition $\Pi_x\tau\sigma=(\Pi_x\tau)(\Pi_x\sigma)$
and the definition~\eqref{eq:Pi_x_Taylor_def}.
The bounds on $\Gamma_{xy}$ follow from the argument in~\cite[p.~436]{Hairer14}.
\end{proof}

\subsection{Admissible maps, pre-models, and models}
\label{subsec:models}

\begin{definition}\label{def:admissible}
A linear map $\bfPi\colon\mcT\to C^\infty$ is called \textit{admissible} if there exists $\xi\in C^\infty$ such that, for all $\tau\in\mfT_\circ$ and $k,l,m\in\N^d$
\begin{equs}
\bfPi X^k \tau &=\cdot^k\bfPi\tau\;,&\qquad
\bfPi\bone &= 1\;,\label{eq:admiss_prop_1}
\\
\bfPi \Xi_l &= D^l \xi\;,&\qquad \bfPi\mcI_m[\tau] &= D^mK*\bfPi\tau\;. \label{eq:admiss_prop_2}
\end{equs}
\end{definition}
The canonical lift of any $\xi\in C^\infty$ defined by~\eqref{eq:canonical1}-\eqref{eq:canonical2} is admissible.
The key difference between canonical lifts and admissible maps is that the former are multiplicative in the sense that $\bfPi\tau\sigma=(\bfPi\tau)(\bfPi\sigma)$, while no such condition is imposed on admissible maps.

Given any linear map $\bfPi\colon\CT\to C^\infty$ and $x\in \R^d$, we can define the linear map $\Pi_x\colon\CT\to C^\infty$ and $f_x\in G_+$ exactly as before by setting\label{Pi_x_page_ref}
\[
\Pi_x \eqdef \bfPi F_x\;,\qquad F_x \eqdef \Gamma_{f_x}\eqdef (\id\otimes f_x)\Deltap\;,
\]
where $f_x(X^i)=-x_i$
and $f_x(\mcI_m\tau)$ is given by~\eqref{eq:pos_char_def}.

\begin{definition}
We call $\{\Pi_x\}_{x\in \R^d}$ the family of \textit{recentred maps} and $\{f_x\}_{x\in \R^d}$ the family of \textit{positive characters} associated with $\bfPi$.
We denote by $\mcZ(\bfPi) = (\Pi,\Gamma)$ the families $\{\Pi_x\}_{x\in \R^d}$ and $\{\Gamma_{xy}\}_{x,y\in\R^d}$ where $\Gamma_{xy}\eqdef  F_x^{-1}F_y = \Gamma_{f_x^{-1}\circ f_y}$.
\end{definition}

\begin{remark}\label{rem:models_well_defined}
It may now seem like we have a circular definition: to define $\Pi_x$
we need $f_x$, but we define $f_x$ in terms of $\Pi_x$.
Luckily this is not a problem since
to define $f_x(\mcI_m\tau)$, we need to know $\Pi_x\tau$, which in turn depends on $f_x$ only through trees $\sigma\in\mfT_+$ for which $\fancynorm{\sigma}\leq\fancynorm{\tau}$.
Hence $\bfPi$ uniquely determines $f_x$ and $\Pi_x$ for every $x\in\R^d$.

Building on this observation, if $\bfPi$ is admissible, it is furthermore possible to write $f_x$
as an explicit function of $\bfPi$ using a `twisted antipode', which is the approach taken in~\cite[Sec.~6.1]{BHZ19},
but we will not use this construction here.
\end{remark}
We emphasise that $\bfPi$ and $\{f_x\}_{x\in \R^d}$ can be recovered from the recentred map $\{\Pi_x\}_{x\in \R^d}$ (and thus from $\mcZ (\bfPi)$).
This is clear for $f_x$,
while for $\bfPi$ we have
$\bfPi = \Pi_x F_x^{-1}$, where $F_x^{-1}=(1\otimes f_x^{-1})\Deltap$, which
depends only $\Pi_x$.
Therefore, knowledge of the recentred maps $\{\Pi_x\}_{x\in \R^d}$ is equivalent to that of $\bfPi$.
\begin{definition}\label{def:pre_model}
An admissible map $\bfPi$ is called a \textit{pre-model}
if $\mcZ(\bfPi)$ is a model in the sense of~\cite[Def.~2.17]{Hairer14},
i.e. for all $\tau\in\mfT_\circ$
\begin{equ}
\Big|\int_{\R^d}(\Pi_x\tau)(z)\phi^\lambda_x(z)\mrd z\Big| \lesssim \lambda^{|\tau|_\s}
\end{equ}
and
\begin{equ}
|\Gamma_{x,y} \tau|_{\gamma} \lesssim |x-y|_\s^{|\tau|_\s-\gamma}
\end{equ}
uniformly in $\gamma\leq |\tau|_\s$, $\lambda\in (0,1]$, $x,y$ in any compact set of $\R^d$,
and
\begin{equ}
\phi \in \mcB^r \eqdef \{\phi\in C^\infty(\R^d)\,:\,\supp(\phi)\subset B_0(1)\,,\, |\phi|_{C^r}\leq 1\}\;,
\end{equ}
where $r > 1-\min\{A\}$
and
\begin{equ}
\phi^\lambda_x(z) \eqdef \lambda^{-\s_1-\ldots-\s_d}(\lambda^{-\s_1}(z_1-x_1),\ldots, \lambda^{-\s_d}(z_d-x_d))\;.
\end{equ}
\end{definition}
We saw in Theorem~\ref{thm:canonical_model} that the admissible model which is multiplicative (the canonical lift) is a pre-model.

\subsection{Global description of \texorpdfstring{$\Deltap$}{Delta+}}\label{subsec:global_pos}

We conclude this section with a `global' (non-recursive) characterisation of $\Deltap$ that
complements its recursive definition.
This global definition is the one primarily used in~\cite{BHZ19}.

\begin{proposition}\label{prop:Deltap_global}
Consider $T^{\mfn}_{\mfe}\in\mfT_\circ$ and $\Deltap\colon \mcT\to\mcT\otimes\mcT_+$.
Then 
\begin{equs}[eq:Deltap_global]
 \Deltap  T^{\mfn}_{\mfe} = 
 \sum_{A \subset T} \sum_{\mfe_A,\mfn_A}  \frac1{\mfe_A!}
\binom{\mfn}{\mfn_A}
 &(A,\mfn_A+\pi\mfe_A, \mfe \restr E_A) 
\\
&\otimes \pi_+(T/A, [\mfn - \mfn_A]_A, \mfe + \mfe_A)\;, 
\end{equs}
with notation as follows.
\begin{itemize}
\item $\pi_+\colon\mfT\to\mcT_+$ is the projection which sends to $0$ every tree in $\mfT\setminus\mfT_+$.
\item For $ C \subset D $ and $ f \colon D \rightarrow \N^d $, we denote by $ f \restr C$ the restriction of $ f $ to $ C $.
We denote $f! = \prod_{x\in D}f(x)!$ where $n!=(n_1!)\ldots(n_d!)$ for $n\in\N^d$ as usual.
Furthermore, for $g\colon C\to \N^d$, we denote
\begin{equ}
\binom{f}{g} = \prod_{x\in C} \binom{f(x)}{g(x)}\;,
\end{equ}
where $\binom{n}{k} = \binom{n_1}{k_1}\ldots\binom{n_d}{k_d}$ for $k,n\in\N^d$.
\item The first sum in~\eqref{eq:Deltap_global}
runs over all non-empty subtrees $A$ of
$T$ containing the root. The second sum runs over all  $ \mfn_A \colon N_A \rightarrow \N^{d}  $ with $\mfn_A\leq \mfn$ and $ \mfe_{A} \colon \partial(A,T) \rightarrow \N^{d} $, where 
\begin{equ}\label{eq:boundary_def}
\partial(A,T) = \{(x,y)\in E_T\setminus E_A\,:\, x\leq y\,,\, x\in A\}
\end{equ}
is called the boundary of $A$ in $T$
and where we recall that $x\leq y$ means that $x$ is in the unique path connecting $y$ to the root.
The sum $\mfe+\mfe_A\colon E_T\to\mcE=\mfL\times \N^d$ is taken in the second component only.

\item  We write $T/A$ for the tree obtained by contracting $A$ to a new root $\bar \rho$. For $ f \colon N_T \rightarrow \N^{d}$, we define $[f]_A\colon N_{T/A}\to\N^d$ as $[f]_A(\bar\rho)=\sum_{x\in N_A} f(x)$
and $[f]_A(y)=f(y)$ for all $y\in N_{T/A}\setminus\{\bar\rho\}$.
\item For $ f \colon E_T  \rightarrow \N^{d}  $, we set for every $ x \in N_T $, $ (\pi f)(x) = \sum_{(x,y) \in E_T} f(x,y)$.
\end{itemize}
\end{proposition}

\begin{example}
For $\Deltap\colon \mcT\to\mcT\otimes\mcT_+$,
\begin{equs}
\Deltap
\begin{tikzpicture}  [scale=0.8,baseline=-0.3cm]
    \node [dot,label=90:{\scriptsize $X^b$}] (uu) at (-2, 0) {};
    \node [dot,label=-90:{\scriptsize $X^a$}] (m) at (-2, -1) {};
    \node [dot] (ml) at (-1, 0) {};
    \draw (m) to node[sloped, midway,below]{\scriptsize $\Xi_\beta$} (ml);
\draw (m) to node[midway,left]{\scriptsize $\mcI_p$} (uu);
\node [dot] (uur) at (-1,1) {};
\draw (uu) to node[sloped, midway,below]{\scriptsize $\Xi_\gamma$} (uur);
\node [dot,label=0:{\scriptsize $X^c$}] (uul) at (-3,1) {};
\draw (uu) to node[sloped, midway,below]{\scriptsize $\mcI_q$} (uul);
\end{tikzpicture}
\quad &= \sum_{\bar a \leq a} \sum_{k\in\N^d} \frac{1}{k!}\binom{a}{\bar a}
\begin{tikzpicture}  [scale=0.8,baseline=-0.3cm]
    \node [dot,label=-90:{\scriptsize $X^{\bar a+k}$}] (m) at (-2, -1) {};
\node [dot] (ml) at (-1, 0) {};
    \draw (m) to node[sloped, midway,below]{\scriptsize $\Xi_\beta$} (ml);
\end{tikzpicture}
\otimes
\pi_+ \Big(
\begin{tikzpicture}  [scale=0.8,baseline=-0.3cm]
    \node [dot,label=90:{\scriptsize $X^b$}] (uu) at (-2, 0) {};
    \node [dot,label=-90:{\scriptsize $X^{a-\bar a}$}] (m) at (-2, -1) {};
\draw (m) to node[midway,right]{\scriptsize $\mcI_{p+k}$} (uu);
\node [dot] (uur) at (-1,1) {};
\draw (uu) to node[sloped, midway,below]{\scriptsize $\Xi_\gamma$} (uur);
\node [dot,label=0:{\scriptsize $X^c$}] (uul) at (-3,1) {};
\draw (uu) to node[sloped, midway,below]{\scriptsize $\mcI_q$} (uul);
\end{tikzpicture}
\Big)
\\
&+\sum_{\bar a \leq a} \sum_{\bar b \leq b}
\sum_{k\in\N^d} \frac{1}{k!}
\binom{a}{\bar a}\binom{b}{\bar b}
\begin{tikzpicture}  [scale=0.8,baseline=-0.3cm]
    \node [dot,label=90:{\scriptsize $X^{\bar b}$}] (uu) at (-2, 0) {};
    \node [dot,label=-90:{\scriptsize $X^{\bar a}$}] (m) at (-2, -1) {};
    \node [dot] (ml) at (-1, 0) {};
    \draw (m) to node[sloped, midway,below]{\scriptsize $\Xi_\beta$} (ml);
\draw (m) to node[midway,left]{\scriptsize $\mcI_p$} (uu);
\node [dot] (uur) at (-1,1) {};
\draw (uu) to node[sloped, midway,below]{\scriptsize $\Xi_\gamma$} (uur);
\end{tikzpicture}
\otimes 
\pi_+\Big(
\begin{tikzpicture}  [scale=0.8,baseline=-0.3cm]
    \node [dot,label=-90:{\scriptsize $X^{a-\bar a + b-\bar b+k}$}] (m) at (-2, -1) {};
\node [dot,label=0:{\scriptsize $X^c$}] (ml) at (-3, 0) {};
    \draw (m) to node[sloped, midway,below]{\scriptsize $\mcI_{q+k}$} (ml);
\end{tikzpicture}
\Big)
\\
&+
\sum_{\bar a \leq a}\sum_{\bar b \leq b
} \sum_{\bar c\leq c}
\binom{a}{\bar a}\binom{b}{\bar b}\binom{c}{\bar c}
\begin{tikzpicture}  [scale=0.8,baseline=-0.3cm]
    \node [dot,label=90:{\scriptsize $X^{\bar b}$}] (uu) at (-2, 0) {};
    \node [dot,label=-90:{\scriptsize $X^{\bar a}$}] (m) at (-2, -1) {};
    \node [dot] (ml) at (-1, 0) {};
    \draw (m) to node[sloped, midway,below]{\scriptsize $\Xi_\beta$} (ml);
\draw (m) to node[midway,left]{\scriptsize $\mcI_p$} (uu);
\node [dot] (uur) at (-1,1) {};
\draw (uu) to node[sloped, midway,below]{\scriptsize $\Xi_\gamma$} (uur);
\node [dot,label=0:{\scriptsize $X^{\bar c}$}] (uul) at (-3,1) {};
\draw (uu) to node[sloped, midway,below]{\scriptsize $\mcI_q$} (uul);
\end{tikzpicture}
\otimes X^{a-\bar a + b- \bar b + c -\bar c}\;.
\end{equs}
\end{example}

\begin{proof}[of Proposition~\ref{prop:Deltap_global}]
Let us write $\bar\Delta \colon \mcT\to\mcT\otimes\mcT_+$ for the linear map that sends $T^\mfn_\mfe$ to the right-hand side of~\eqref{eq:Deltap_global}.
It is immediate that $\bar\Delta$ agrees with $\Deltap$ on $\bone$, $X^i$ for $i=1,\ldots, d$, and $\Xi_l$ for $l\in\N^d$.
Furthermore $\bar\Delta(\tau\bar\tau)=\bar\Delta(\tau)\bar\Delta(\bar\tau)$,
which follows from the fact that, for any finite $I\subset\mfT$, the product $\prod_{\tau\in I} \tau$ is in $\mfT_+$ if and only if $I\subset\mfT_+$, and from Vandermonde's identity $\binom{k+\ell}{c} = \sum_{a+b=c}\binom{k}{a}\binom{\ell}{b}$ to handle the sum of polynomial decorations at the root.

To conclude that $\bar \Delta=\Deltap$,
it remains to verify that for all $\mcI_m\tau\in\mfT_\circ$
\begin{equ}
\bar\Delta(\mcI_m\tau) = (\mcI_m\otimes \id)\bar\Delta\tau
+ \sum_{|k|_\s<|\mcI_m\tau|_\s} \frac{X^k}{k!}\otimes \mcI_{m+k}\tau\;.
\end{equ}
Let $\rho$ be the root of $T^\mfn_\mfe\eqdef\mcI_m\tau$.
Observe that $\mfn_A(\rho)=0$ in the sum defining $\bar\Delta$ since $\mfn(\rho)=0$.
Therefore, the terms in the sum in which $A$ contains at least two nodes (which necessary includes the root)
sum to $(\mcI_m\otimes \id)\bar\Delta \tau$.
For the remaining terms, where $A$ contains only the root $\rho$,
the boundary is a singleton $\partial(A,T)=\{e\}$ with only one edge $e=(\rho,x)$ (the single edge incident on $\rho$).
Therefore the corresponding right factor becomes $\mcI_{m+\mfe_A(e)}[\tau]$
and we take the sum over all $\mfe_A(e)=k\in\N^d$ such that $|\mcI_{m+k}\tau|_\s>0$, i.e. $|k|_\s < |\mcI_m\tau|_\s$.
The left factor thus becomes $\frac{1}{\mfe_A(e)!}X^{\mfe_A(e)}=\frac{1}{k!}X^k$.
Hence, the terms where $A$ contains only $\rho$ sum to
\begin{equ}
\sum_{|k|_\s<|\mcI_m\tau|_\s}\frac{X^k}{k!}\otimes \mcI_{m+k}\tau\;.
\end{equ}
This matches exactly the recursive definition of $\Deltap$, completing the proof.
\end{proof}

\begin{exercise}
Verify that, for $\Deltap\colon\mcT_+\to\mcT_+\otimes\mcT_+$ and $T^\mfn_\mfe\in\mfT_+$,
the same expression~\eqref{eq:Deltap_global} holds with $(A,\mfn_A+\pi\mfe_A, \mfe \restr E_A)$ replaced by $\pi_+(A,\mfn_A+\pi\mfe_A, \mfe \restr E_A)$.
\end{exercise}

\section{Negative renormalisation}
\label{sec:neg_renorm}

The reason (pre-)models (Definition~\ref{def:pre_model}) are of interest
is that models can be used as `input' to solve PDEs.
Furthermore, the solutions are stable under a metric on the space of models.
This metric (which we do not discuss here, but see~\cite[Eq.~(2.17)]{Hairer14}) is best understood using models (vs. pre-models) and is the reason we are interested in the recentred maps $(\Pi_x)_{x\in\R^d}$ and transformations $\{\Gamma_{x,y}\}_{x,y\in\R^d}$
rather than only admissible maps / pre-models $\bfPi$.
(See Section~\ref{sec:renorm_SPDE} and the discussion around~\eqref{eq:fixed_point}
for an indication of how one solves a PDE with a model $(\Pi,\Gamma)$ as input.)

Consider now a mollification $\xi_\eps = \xi * \delta_\eps$ at scale $\eps>0$ of a random distribution $\xi$.
Let $(\Pi^\eps,\Gamma^\eps)$ be the `canonical model' built from $\xi_\eps$ using the procedure in Section~\ref{sec:pos_renorm}.
Much like in Sections~\ref{subsubsec:DPD} and~\ref{subsubsec:d4},
we can try to define the `solution' to the singular SPDE
\begin{equ}\label{eq:SPDE2}
(\partial_t -\mcL) u = F(u,\nabla u,\ldots,\xi,\nabla\xi,\ldots)
\end{equ}
as the limit $u \eqdef \lim_{\eps\to 0} u_\eps$ of the solutions $u_\eps$ of the corresponding PDE with models $(\Pi^\eps,\Gamma^\eps)$ as input.
Showing that the SPDE~\eqref{eq:SPDE2} is well-posed therefore boils down to showing convergence of the models $(\Pi^\eps,\Gamma^\eps)$.

It turns out that for many random distributions $\xi$ of interest, this procedure works.
A notable case is when $\xi\colon\R\to\R^n$ is the derivative of an $n$-dimensional Brownian motion $W$ and $F(u,\xi)=F(u)\xi$, so that we are in the setting of stochastic \textit{ordinary} differential equations.
In this case the model converges to the so-called \textit{canonical rough path lift} of Brownian motion
and the limiting solution solves the Stratonovich SDE $\mrd u = F(u)\circ\mrd W$.
More generally, this holds for $W$ from a large class of Gaussian and Markovian processes~\cite{CQ02, FV08_Markov, FV10} (including fractional Brownian motion with Hurst parameter $H>\frac14$).
The interpretation is that, while products of the form
$(\Pi^\eps_x\tau)(\Pi^\eps_x\sigma)$ may be deterministically unstable, e.g. the integral $\int_0^T W_t^i \mrd W_t^j$ of one Wiener process against another is analytically ill-defined (see Remark~\ref{rem:integral_ill-defined}),
they can still be made sense of through smooth approximations due to probabilistic cancellations.

Unfortunately, for many interesting singular stochastic \textit{partial} differential equations,
the models $(\Pi^\eps,\Gamma^\eps)$ fail to converge as $\eps\to 0$.
The problem stems from the fact that products $(\Pi^\eps_x\tau)(\Pi^\eps_x\sigma)$ are not only analytically ill-defined, but they \textit{diverge} as distributions (cf. the Brownian motion case).
We already encountered such divergences in Proposition~\ref{prop:Wick}.

To cure these divergences, we introduce another type of renormalisation, called `negative renormalisation' and developed systematically in~\cite{BHZ19},
which redefines $\Pi_x\tau\sigma$ in order to restore convergence.

\begin{remark}
`Positive renormalisation' happened pathwise: given a realisation of $\xi_\eps$, we constructed $\Pi^\eps$ without caring which distribution $\xi_\eps$ came from.
In contrast, `negative renormalisation' happens \textit{in law}:
the way we choose how the renormalisation group acts on $\Pi^\eps$ will typically not depend on the specific realisation of $\xi_\eps$, but on its \textit{law}.
\end{remark}
Denote $\mfF \eqdef \N^\mfT$,\label{forests_page_ref} which we call the set of \textit{forests} (multisets of trees).
We will identify each forest $\tau\in\mfF$ with a decorated graph $F^\mfn_\mfe$,
where $\mfn\colon N_F\to\N^d$ and $\mfe\colon E_F\to\mcE$, which is an unordered collection of trees in $\mfT$ (hence the term forest).
In this way we think of $\mfT\subset\mfF$.

We denote by $\fprod$ the commutative and associative `forest product' on $\mfF$
given by $(F\fprod H) (\tau)= F(\tau)+H(\tau)$ for all $F,H\in\N^\mfT$ and $\tau\in\mfT$.
Note that this simply corresponds to combining the trees in $F$ and $H$ into a single forest.
The unit under $\fprod$ is the empty forest denoted by $\e$,\label{empty_forest_page_ref} for which $\e(\tau)=0$ for all $\tau\in\mfT$.
Note that $\e\neq\bone$, where $\bone\in\mfT$ is the unit under the tree product.

\begin{definition}\label{def:neg_trees}
We say a tree $\tau\in\mfT$ is \textit{planted} if either $\tau=\Xi_l$ for some $l\in \N^d$ or if there exists $m\in\N^d$ and $\sigma\in\mfT$
such that $\tau=\mcI_m\sigma$.
Define the set of trees $\mfT_-\subset \mfT$ by
%\footnote{There is a small discrepancy between our definition of $\mfT_-$ and that in~\cite{BCCH21}.}
\begin{equ}
\mfT_- \eqdef \{\tau=T^\mfn_\mfe \in\mfT_\circ\,:\, |\tau|_\s<0 \,,\,
% \mfn(\rho)=0\,,\,
\mfn(\rho_\tau)=0\,,\,
\tau \textnormal{ is not planted}
\}\;.
\end{equ}
In other words, $\mfT_-$ consists of all unplanted trees of negative degree with zero polynomial decoration at the root.
We furthermore define the set of forests $\mfF_-\subset\mfF$
\begin{equ}
\mfF_- \eqdef \{ \tau_1\fprod\ldots\fprod \tau_n \,:\, n\geq 0\,,\,
\tau_i\in\mfT_-\}\;.
\end{equ}
The case $n=0$ above corresponds to $\tau_1\fprod\ldots\fprod \tau_n=\e$.

Define $\mcT_- =\Span_\R\mfT_-$\label{mcT_minus_pageref}
and $\mcF_- = \Span_\R \mfF_-$.\label{mcF_minus_pageref}
Let $G_-$\label{G_minus_page_ref} denote the set of \textit{characters} on $\mcF_-$, i.e. the collection of linear maps $g\colon\mcF_-\to \R$ such that $g(\e)=1$ and \begin{equation*}%\label{eq:neg_char_mult}
g(\tau\fprod\sigma)=g(\tau)g(\sigma)\;.
\end{equation*}
\end{definition}
Note that we have a canonical inclusion $\mcT_-\hookrightarrow\mcF_-$ and that $\mcF_-$ is a commutative algebra freely generated by $\mfT_-$.

We next define a linear map $\Deltam\colon \mcT\to \mcF_-\otimes\mcT$\label{Delta_minus_page_ref}
with a similar form as the `global' definition of $\Deltap$
in Proposition~\ref{prop:Deltap_global}.

\begin{definition}\label{def:Deltam}
Writing $F^\mfn_\mfe=(F,\mfn,\mfe)\in\mfF$,
we define for $T^{\mfn}_{\mfe}\in\mfT_\circ$
\begin{equs}[eq:Deltam_def]
 \Deltam  T^{\mfn}_{\mfe} = 
 \sum_{A \subset T} \sum_{\mfe_A,\mfn_A}  \frac1{\mfe_A!}
\binom{\mfn}{\mfn_A}
 &\pi_-(A,\mfn_A+\pi\mfe_A, \mfe \restr E_A) 
\\
&\otimes( T/A, [\mfn - \mfn_A]_A, \mfe + \mfe_A)\;, 
\end{equs}
with notation the same as in Proposition~\ref{prop:Deltap_global}
except the following.
\begin{itemize}
\item $\pi_-\colon\mfF\to\mcF_-$ is the projection which sends to $0$ every forest $A\in\mfF\setminus\mfF_-$.
\item The first sum runs over all subgraphs $ A  $ of
$ T $ ($ A $ may be empty), which we identify with a forest of rooted trees. The second sum, as in Proposition~\ref{prop:Deltap_global},
runs over all  $\mfn_A \colon N_A \rightarrow \N^{d}$ with $\mfn_A\leq \mfn$ and $ \mfe_{A} \colon \partial(A,T) \rightarrow \N^{d} $ with $ \partial(A,T) $ defined by~\eqref{eq:boundary_def}.
% where 
%\[
%\partial(A,T) = \{(x,y)\in E_T\setminus E_A\,:\, x\leq y\,,\, x\in A\}
%\]
%is called the boundary of $A$ in $F$
%and where $x\leq y$ means that $x$ is in the unique path connecting $y$ to the root.
%The sum $\mfe+\mfe_A\colon E_T\to\mfL\times \N^d$ is taken in the second component only.

\item  We write $T/A$ for the tree obtained by contracting the connected components $A_1,\ldots, A_\ell$ of $A$ to nodes $x_1,\ldots,x_\ell$ respectively. For $f \colon N_T \to \N^{d}$, we define $[f]_A\colon N_{T/A} \to \N^d$ by $[f]_A(x_i)=\sum_{y\in N_{A_i}} f(y)$
and $[f]_A(x)=f(x)$ for $x\in N_{T/A}\setminus\{x_1,\ldots,x_\ell\}$.
%\item For $ f \colon E_T  \rightarrow \N^{d}  $, we set for every $ x \in N_T $, $ (\pi f)(x) = \sum_{e=(x,y) \in E_T} f(e)$.
\end{itemize}
\end{definition}

\begin{example}\label{ex:Deltam}
Consider the rule $R$ and scaling $\s=(1,1,1,2)$ associated to the $\Phi^4_3$ equation as in Examples~\ref{ex:Phi43_rule} and~\ref{ex:Phi43_subcritical}.
Denoting $\Xi=\Xi_0$ and $\mcI=\mcI_0$ as usual, consider
\begin{equ}
\Psi\eqdef\mcI[\Xi] = \she \in\mfT_\circ\;,
\end{equ}
and the tree
\begin{equ}
\tau\eqdef \mcI[\Psi^3] = \fork
\in\mfT_\circ
\end{equ}
($\tau$ is sometimes called the `pitchfork' or `chicken foot').
Above, straight lines denote edges with labels $\mcI$ and zigzag lines denote edges with decoration $\Xi$.
All nodes have zero polynomial decoration $X^0$.
Then
\begin{equs}
|\Psi|_\s
&=
|\she|_\s=-\frac12-\;,\quad |\Psi^2|_\s
=
|\shesq|_\s=-1-\;,\quad
|\Psi^3|_\s
=
|\shecube|_\s
=-\frac32-\;,
\\
|\tau|_\s
&=
\Big|
\fork
\Big|_\s=\frac12-\;.
\end{equs}
Therefore
\begin{equ}
\Deltam
\fork =\e\otimes 
\fork + 3\,
\shesq
\otimes
\intshe
 +
\shecube
\otimes
\intone\;.
\end{equ}
\end{example}
For another (even simpler) example of how $\Deltam$ operates, see the proof of Proposition~\ref{prop:direct_renorm_base_case} below.
\begin{remark}
We saw that $\Deltap$ admits a recursive definition and a global one (Proposition~\ref{prop:Deltap_global}).
It turns out that $\Deltam$ also admits a recursive definition, albeit not as simple as that of $\Deltap$, see~\cite{Bruned18}.
\end{remark}

\begin{remark}\label{rem:complete}
As stated, the fact that $\Deltam$ has right factor in $\mcT$ is not always true.
The technical condition on a rule which ensures $\Deltam$ maps $\CT\to\mcF_-\otimes\CT$
is called \textit{completeness} in~\cite[Sec.~5.3]{BHZ19};
it is somewhat lengthy to state, so we do not repeat it here.
Luckily, for any subcritical rule $R$, there exists another subcritical rule $\bar R$ which enlarges $R$ in the sense that $R\subset \bar R$ and under which $\Deltam$ indeed maps $\CT\to\mcF_-\otimes\CT$, see~\cite[Prop.~5.21]{BHZ19}.
The following exercise shows that, if a rule is built naively from an equation, it may happen that $\Deltam$ does not map $\CT\to\mcF_-\otimes\CT$, i.e. the rule may fail to be complete. 
\end{remark}

\begin{exercise}
Consider the rule $R$ defined in Example~\eqref{ex:Phi43_rule} associated to the $\Phi^4_{d-1}$ equation
with scalings $|\Xi|_\s=-\frac d2 - \frac12 -\kappa$ and $|\mcI|_\s=2$.
If $|\Xi|_\s < -\frac{11}{4}$, which corresponds to white noise in fractional dimension
$d\geq 9/2$, show that there exists $\tau\in\mcT$
such that $\Deltam \tau\notin\mfF_-\otimes\mcT$.

[Hint: following notation from Example~\ref{ex:Deltam}, consider the tree
\begin{equ}
\tau=\Psi^2\mcI[\Psi^3] = \bigtree\;.
\end{equ}
Find a negative tree subtree $A\subset \tau$ with non-empty boundary onto which we can place one derivative.]
\end{exercise}
We assume henceforth that $R$ is complete so that $\Deltam$ maps $\CT$ to $\mcF_-\otimes\mcT$.
Given any map $\bfPi\colon\mcT\to C^\infty$
and $g\in G_-$,
it is natural to define
\begin{equ}
\bfPi^g \eqdef (g\otimes \bfPi)\Deltam.
\end{equ}
The following is the main result on renormalisation of models.
\begin{theorem}\label{thm:pre-model_renorm}
Suppose that $g\in G_-$.
Then, if $\bfPi$ is a pre-model,
so is $\bfPi^g$.
\end{theorem}
We give the proof of Theorem~\ref{thm:pre-model_renorm} at the end of Section~\ref{subsec:reduced}.
In the rest of this section, we will present the algebraic framework and results necessary for its proof.
\begin{remark}\label{rem:renorm_cont}
Another important fact is that the map $\mcZ(\bfPi)\mapsto \mcZ(\bfPi^g)$ is continuous for every $g\in G_-$.
Since we have not discussed the topology on models, we will not elaborate on this point.
\end{remark}

\begin{remark}
There is a corresponding analytic result~\cite{CH16}
which states that, for a large class of random distributions $\xi$,
there exists a choice of (deterministic) elements $\{g_\eps\}_{\eps>0}\subset G_-$ such that the model $\mcZ((\bfPi^{\eps})^{g_\eps})$
converges in probability to a limiting random model
as $\eps\downarrow 0$,
where $\bfPi^{\eps}$ is the canonical lift of $\xi_\eps=\xi*\delta_\eps$ and $\delta_\eps$ is a mollifier.
\end{remark}
To try to prove Theorem~\ref{thm:pre-model_renorm}, it is natural to understand better the algebraic properties of the map $\Deltam$ and its interaction with $\Deltap$.
\begin{definition}\label{def:Deltam_CH}
We extend the domain and range of $\Deltam$ in the following two ways.
\begin{itemize}
\item We define a linear map $\Deltam \colon \mcT_-\to \mcF_-\otimes \mcT_-$.\label{Delta_minus_page_ref_2}
by the same expression as~\eqref{eq:Deltam_def} but with $(T/A, [\mfn - \mfn_A]_A, \mfe + \mfe_A)$ replaced by $\pi_-( T/A, [\mfn - \mfn_A]_A, \mfe + \mfe_A)$,
and then extend this to an algebra morphism $\Deltam \colon \mcF_-\to \mcF_-\otimes \mcF_-$.

\item We define an algebra morphism $\Deltam \colon\mcT_+ \to \mcF_-\otimes\mcT_+$ given by the same expression as~\eqref{eq:Deltam_def}
for every \textit{planted} tree $T^\mfn_\mfe = \mcI_m\tau\in\mfT_+$ (in particular the extracted subgraphs $A$ in~\eqref{eq:Deltam_def} cannot contain the root because $\mfT_-$ does not contain planted trees)
and then extended multiplicatively.
\end{itemize}
\end{definition}

\begin{remark}
Similar to Remark~\ref{rem:abuse_positive}, $\Deltam$ now corresponds to three different maps $\CH\to \CF_-\otimes \CH$ with domains $\CH\in \{\mcT,\mcT_+,\mcF_-\}$.
We will always specify which domain we are working with, thus avoiding any confusion.
\end{remark}

\begin{remark}
Readers familiar with~\cite{BHZ19} may notice that we make no use of `colours' here.
This is partly due to our abuse of notation which forces us to specify the domains of $\Deltam$ and $\Deltap$, and partly due to our definition of $\Deltam$ on $\tau\in\mcT_+$ which automatically does not extract the root of $\tau$.
\end{remark}
These maps have the following nice properties analogous to Theorem~\ref{thm:pos_Hop_alg}.
\begin{theorem}\label{thm:negative_Hopf}
There exists an involutory algebra morphism $\mcA_-\colon\mcF_-\to\mcF_-$
such that $(\mcF_-,\fprod, \Deltam, \mcA_-)$ is a Hopf algebra with antipode $\mcA_-$, unit $\e$, counit $\e^*\colon \mcF_-\to\R$, which is the linear map with $\e^*(\e)=1$ and $\e(\tau)=0$ for $\tau\in\mfF_-\setminus\{\e\}$.
Furthermore, $\mcT$ and $\mcT_+$ are comodules over $\mcF_-$.
In other words,
\begin{equ}
(\id\otimes \Deltam)\Deltam = (\Deltam\otimes \id)\Deltam
\end{equ}
as maps $\mcH\to \mcF_-\otimes\mcF_-\otimes\mcH$ where $\mcH\in\{\mcT,\mcT_+,\mcF_-\}$,
and
\begin{equ}%\label{eq:antipode_neg}
\mcM_-(\id\otimes \mcA_-)\Deltam = \mcM_-(\mcA_-\otimes \id)\Deltam = \e\e^*
\end{equ}
as maps $\mcF_-\to \mcF_-$, where $\mcM_-\colon\mcF_-\otimes\mcF_- \to \mcF_-$ maps $\tau\otimes\bar\tau\mapsto \tau\fprod \bar\tau$.
\end{theorem}
The proof of Theorem~\ref{thm:negative_Hopf} is somewhat more involved than that of Theorem~\ref{thm:pos_Hop_alg},
and we do not go over the details.
We mention however that in the absence of decorations, $\Deltam$ is the extraction-contraction coproduct of~\cite{CEFM11}, for which coassociativity is straightforward to show.

The following exercise follows in the same way as Exercise~\ref{ex:prod_G_+}.
\begin{exercise}\label{ex:prod_G_-}
Show that $G_-$ is a group under the product $g\bar g \eqdef (\bar g\otimes g)\Deltam$
and that, for $\CH\in \{\CT,\CT_+\}$, the map \begin{equ}
G_-\ni g\mapsto M_g\eqdef (g\otimes \id)\Deltam\in \GL(\CH)
\end{equ}
is a group homomorphism.
\end{exercise}
Theorem~\ref{thm:negative_Hopf} makes no link between $\Deltam$ and $\Deltap$.
The key to the proof of Theorem~\ref{thm:pre-model_renorm} turns out to be a \textit{cointeraction} between $\Deltam$ and $\Deltap$, which is similar to the one discovered in~\cite{CEFM11}.
The simplest cointeraction property one could hope for is
\begin{equation}\label{eq:cointeraction}
\mcM_{(13)(2)(4)} (\Deltam\otimes\Deltam)\Deltap = (\id\otimes\Deltap)\Deltam
\end{equation}
as maps $\mcH\to \mcF_-\otimes\mcH\otimes\mcT_+$,
where $\mcH\in\{\mcT,\mcT_+\}$ and
\begin{equ}
\mcM_{(13)(2)(4)}\colon \mcF_-\otimes\mcH\otimes \mcF_-\otimes\mcT_+\to \mcF_-\otimes\mcH\otimes\mcT_+
\end{equ}
is given by 
\begin{equ}\label{eq:mcM_def}
\mcM_{(13)(2)(4)}(g\otimes \tau\otimes h\otimes f)=g\fprod h\otimes \tau\otimes f\;.
\end{equ}
Unfortunately,~\eqref{eq:cointeraction} is \textit{false} as we will see in Example~\ref{ex:cointeract_fails}.
The importance of this type of cointeraction will be revealed in Section~\ref{subsec:extended_struct},
where we will see, amongst other things, that it allows us to define an action of $G_-$ on $G_+$ through group homomorphisms.
We will also see later how we can augment the above constructions in order to recover~\eqref{eq:cointeraction} in a larger space (see Theorem~\ref{thm:cointerct_ex}).

Before progressing, we give some hope for~\eqref{eq:cointeraction} by showing that it does hold for trees with no kernel edges.
This highlights that~\eqref{eq:cointeraction} fails because of kernel edges.
Although the proof is almost trivial, we show the computation as a basic example of how $\Deltam$ and $\Deltap$ operate.

\begin{proposition}\label{prop:direct_renorm_base_case}
Consider $k\in\N^d$ and $L\in\mfN$ such that $\tau \eqdef X^k\Xi_L \in\mfT_\circ$.
Then
\begin{equ}\label{eq:cointer_base_case}
\mcM_{(13)(2)(4)} (\Deltam\otimes\Deltam)\Deltap\tau = (\id\otimes\Deltap)\Deltam \tau\;.
\end{equ}
\end{proposition}

\begin{proof}
Observe that
\begin{equ}
\Deltap (X^k\Xi_L)
= \sum_{n \leq k}\binom kn X^n\Xi_L\otimes X^{k-n}\;.
\end{equ}
On the other hand,
\begin{equ}
\Deltam (X^k\Xi_L) = \sum_{\substack{A\subset L\\ \#A\neq 1}} \Xi_A\otimes X^{k}\Xi_{L\setminus A}\;,
\end{equ}
where the sum is over all multisubsets $A\subset L$ which are not singletons (since $\mfT_-$ does not contain planted trees).
In particular $\Deltam X^k = \e\otimes X^k$.
Consequently,
\begin{equs} (\Deltam\otimes\Deltam)\Deltap\tau = \sum_{n\leq k} \sum_{\substack{A\subset L\\ \#A\neq 1}}\binom{k}{n}\Xi_A\otimes X^{n}\Xi_{L\setminus A} \otimes \e \otimes X^{k-n}\;,
\end{equs}
while
\begin{equs}
(\id\otimes\Deltap)\Deltam \tau
= \sum_{\substack{A\subset L\\ \#A\neq 1}}\sum_{n \leq k}\binom{k}{n}
\Xi_A\otimes X^{n}\Xi_{L\setminus A} \otimes
X^{k-n}\;.
\end{equs}
The two expressions clearly match, concluding the proof.
\end{proof}
We now show that~\eqref{eq:cointeraction} fails in general, even for the algebraic structures arising from the $\Phi^4_3$ equation.
\begin{example}\label{ex:cointeract_fails}
Let us continue with Example~\ref{ex:Deltam}, recalling the pitchfork tree
\begin{equ}
\tau = \fork \in\mfT_\circ\;,
\end{equ}
and that
\begin{equ}\label{eq:tau_extract_Phi43}
\Deltam
\fork =\e\otimes 
\fork + 3\,
\shesq
\otimes
\intshe
 +
\shecube
\otimes
\intone\;.
\end{equ}
Note that $\Deltap
\she=\she\otimes\bone$, hence
$\Deltap\shecube = \shecube\otimes \bone$,
and therefore
\[
\Deltap
\fork
 =
\fork
\otimes \bone + \bone\otimes 
\fork
\;,
\]
while
\begin{equs}
\Deltap
\intshe
&= 
\intshe
\otimes \bone + \sum_{|k|_\s\leq 1} \frac{X^k}{k!} \otimes
\intshek\;,\label{eq:Deltap1}
\\
\Deltap \intone
&= \intone \otimes \bone +
\sum_{|k|_\s\leq 1}\frac{X^k}{k!}
\otimes
\intonek
\;,\label{eq:Deltap2}
\end{equs}
where `$k$' in the diagrams indicates an edge with decoration $\mcI_k$.
Recalling~\eqref{eq:tau_extract_Phi43}, we obtain
\begin{equs}[eq:Deltap_after_Deltam]
(\id\otimes\Deltap)\Deltam\fork
&= (\id\otimes\Deltap)(\e \otimes \fork + 3\,\shesq\otimes\intshe+\shecube\otimes \intone)
\\
&= \e
\otimes(\fork\otimes \bone + \bone\otimes \fork)
\\
&\qquad+ 3\,\shesq\otimes\Big(\intshe\otimes \bone + \sum_{|k|_\s\leq 1} \frac{X^k}{k!} \otimes\intshek\Big)
\\
&\qquad
+ \shecube\otimes
\Big(
\intone\otimes \bone +
\sum_{|k|_\s\leq 1}\frac{X^k}{k!}\otimes\intonek
\Big)\;.
\end{equs}
On the other hand
\begin{equs}[eq:Deltam_after_Deltap]
(\Deltam\otimes\Deltam)\Deltap\fork
&=
(\Deltam\otimes\Deltam)(\fork\otimes \bone + \bone\otimes \fork)
\\
&=(\e\otimes \fork + 3\,\shesq\otimes\intshe+ \shecube\otimes \intone)\otimes \e \otimes \bone
\\
&\qquad \qquad +
(\e \otimes \bone)\otimes (\e \otimes \fork + 3\,\shesq\otimes\intshe
+ \shecube\otimes \intone)\;.
\end{equs}
Therefore $(\id\otimes\Deltap)\Deltam\tau$ and $\mcM_{(13)(2)(4)} (\Deltam\otimes\Deltam)\Deltap\tau$ 
are almost the same except for the extra terms
\begin{equ}
3\, \shesq\otimes X^k\otimes \intshek
\quad\text{and}\quad
\shecube\otimes X^k\otimes \intonek
\end{equ}
with $|k|_\s=1$ in $(\id\otimes\Deltap)\Deltam\tau$.
\end{example}
What we see in Example~\ref{ex:cointeract_fails} is that,
when $\Deltam$ extracts the trees $\shesq$ and $\shecube$ in~\eqref{eq:tau_extract_Phi43},
it increases the degree of the right factor (the one in $\mcT$) from $|\tau|_\s=\frac12-$ to
\begin{equ}
|\mcI[\Psi]|_\s = \Big|\intshe\Big|_\s =\frac32- \quad \text{and} \quad |\mcI[\bone]|_\s=|\intone|_\s=2\;.
\end{equ}
Then applying $\Deltap$ in~\eqref{eq:Deltap_after_Deltam}, we obtain polynomial decorations $X^k$ with $|k|_\s=1$ because $|\mcI[\Psi]|_\s, |\mcI[\bone]|_\s\in(1,2]$.
On the other hand,
first applying $\Deltap$ in~\eqref{eq:Deltam_after_Deltap}, we never see $X^k$ with $|k|_\s=1$ since $|\tau|_\s\in(0,1]$, and subsequently they are not created by $\Deltam$ in~\eqref{eq:Deltam_after_Deltap}.

Note that the situation would have been different if $\mcI[\Psi]$ and $\mcI[\bone]$ in~\eqref{eq:tau_extract_Phi43} had the same degree as $\tau$,
that is, if $|\mcI[\Psi]|_s,|\mcI[\bone]|_\s\in(0,1]$.
This is because we would then have no $X^k$ with $|k|_\s=1$ in~\eqref{eq:Deltap1}-\eqref{eq:Deltap2}-\eqref{eq:Deltap_after_Deltam},
leading to an equality between $(\id\otimes\Deltap)\Deltam\tau$ and $\mcM_{(13)(2)(4)} (\Deltam\otimes\Deltam)\Deltap\tau$.
In essence, to see the cointeraction~\eqref{eq:cointeraction},
we would like $\mcI[\Psi]$ and $\mcI[\bone]$ in~\eqref{eq:tau_extract_Phi43}
to `remember' that we extracted the trees $\shesq$ and $\shecube$ respectively to ensure that $\Deltap$ in~\eqref{eq:Deltap1}-\eqref{eq:Deltap2}-\eqref{eq:Deltap_after_Deltam}
expands these trees to the correct level, i.e. to $|k|_\s=0$ instead of $|k|_\s\leq 1$.

It turns out that one can implement this idea of `remembering' extracted subtrees to good effect.
% is an elegant way to resolve this issue by keeping track of the degrees of all negative subtrees which are extracted by $\Deltam$.
The idea is to introduce a new `extended decoration'
$\mfo\colon N_T\to \R$ on nodes which changes under applications of $\Deltam$
in order to keep track of the degree of negative subtrees that are extracted.
We will see that this allows us to recover the desired cointeraction~\eqref{eq:cointeraction} in a larger space.

\subsection{Extended structures}\label{subsec:extended_struct}

We now explain the construction outlined above in detail, albeit omitting some proofs.
We define a new set of rooted decorated trees $\mfT^\ex$\label{mfT^ex page ref}
exactly like $\mfT$ in Definition~\ref{def:mfT}, except every tree $\tau = T^{\mfn,\mfo}_\mfe = (T,\mfn,\mfo,\mfe)\in\mfT^\ex$
has an extra decoration $\mfo\colon N_T\to \R$ that vanishes on noise nodes.
Every $\tau\in\mfT^\ex$ therefore has a unique representation as
\begin{equ}\label{eq:general_tree_ext}
\tau = \bone^a X^k \Xi_L \prod_{j\in J} \mcI_{m_j} [\tau_j]\;,
\end{equ}
where $a\eqdef \mfo(\rho_\tau)\in\R$ with $\rho_\tau$ the root of $\tau$, and
where $\tau_j\in\mfT^\ex$.
The tree product extends to $\mfT^\ex$ by defining
\begin{equ}\label{eq:mfT^ex_prod}
\Big(\bone^a X^k \Xi_L \prod_{j\in J} \mcI_{m_j} [\tau_j]\Big)
\Big(\bone^{\bar a} X^{\bar k} \Xi_{\bar L} \prod_{j\in \bar J} \mcI_{m_j} [\tau_j]\Big)
\eqdef \bone^{a+\bar a}X^{k+\bar k} \Xi_{L+\bar L}\prod_{j\in J\sqcup\bar J} \mcI_{m_j} [\tau_j]
\end{equ}
where $J,\bar J$ are disjoint.
That is, $\tau\bar\tau$ is obtained by joining the roots of $\tau$ and $\bar\tau$ and adding the polynomial and extended decorations at the roots.
We in particular identify $\bone^0 =\bone$
and will interpret $\mfT$ as the subset of $\mfT^\ex$ consisting of all trees with $\mfo=0$; in Section~\ref{subsec:reduced} we will give a complete description of how our previous construction fits into this extended one.

We define a projection map $\mcQ\colon \mfT^\ex\to\mfT$ by $\mcQ\colon T^{\mfn,\mfo}_\mfe \mapsto T^{\mfn,0}_\mfe = T^{\mfn}_\mfe$
that sets every extended decoration to $0$.
We extend the degree map $|\cdot|_\s$ to $\mfT^\ex$ by setting $|\tau|_\s\eqdef |\mcQ\tau|_\s$, i.e. by ignoring the extended decoration $\mfo$.

We next define $\mfT^\ex_\circ$, which will determine our model space $\mcT^\ex = \Span_\R(\mfT^\ex_\circ)$.
In terms of the polynomial and edge decorations, there will be no difference between $\mfT^\ex_\circ$ and $\mfT_\circ$, that is, $\mfT_\circ\subset \mfT^\ex_\circ$ and $\mcQ\mfT^\ex_\circ = \mfT_\circ$.
Therefore, the main part in the definition of $\mfT^\ex_\circ$ is determining which extended decorations are allowed and which nodes.

For $\tau = T^{\mfn,\mfo}_\mfe\in \mfT^\ex$,
a subforest $A\subset T$ with connected components $A_1,\ldots, A_\ell$, and $\mfn_A \colon N_A \rightarrow \N^{d}$ with $\mfn_A\leq \mfn$ and $ \mfe_{A} \colon \partial(A,T) \rightarrow \N^{d}$,
we define the \textit{contracted tree} $\mcC(\tau;A,\mfn_A,\mfe_A) \in \mfT^\ex$ by
\begin{equ}\label{eq:contract_def}
\mcC(\tau;A,\mfn_A,\mfe_A) \eqdef (T/A, [\mfn - \mfn_A]_A, \mfo(A)+[\mfn_A + \pi \mfe_A ]_A  , \mfe + \mfe_A)\;,
\end{equ}
where notations are exactly as in Definition~\ref{def:Deltam} and Proposition~\ref{prop:Deltap_global}
and $\mfo(A) \colon N_{T/A}\to\R$ is defined by
\begin{equation*}
 \mfo(A)(x_i) = \sum_{y\in N_{A_i}}\mfo(y) + \sum_{e \in E_{A_i}} |\mfe(e)|_\s\;,
\end{equation*}
and $\mfo(A)(x)=\mfo(x)$ for $x\in N_{T/A}\setminus\{x_1,\ldots, x_\ell\}$.
We further define the \textit{extracted forest} $\mcE(\tau;A,\mfn_A,\mfe_A)\in\mfF^\ex \eqdef \N^{\mfT^\ex}$, interpreted as a multiset of trees in $\mfT^\ex$ as usual, by
\begin{equ}
\mcE(\tau;A,\mfn_A,\mfe_A)\eqdef (A,\mfn_A+\pi\mfe_A, \mfo \restr N_A,\mfe \restr E_A)\;.
\end{equ}

\begin{remark}
If $\mfo=0$ and $A$ has only one connected component $A_1$,
then
\begin{equ}
\mcE(\tau;A,\mfn_A,\mfe_A)\in\mfT_\circ
\end{equ}
and
\begin{equ}\label{eq:remember}
(\mfo(A)+[\mfn_A+\pi\mfe_A]_A)(x_1)=|\mcE(\tau;A,\mfn_A,\mfe_A)|_\s\;.
\end{equ}
A similar statement holds if $A$ has multiple connected components, provided we restrict the right-hand side of~\eqref{eq:remember} to the correct component.
This is what we mean above that $\mfo$ remembers the degree of the extracted subtrees.
\end{remark}

\begin{example}
In the notation of Example~\ref{ex:Deltam},
consider the tree
\begin{equ}
\sigma = T^\mfn_\mfe = X^a\bar \sigma = \bigtreecolour \in \mfT_\circ\;,
\end{equ}
and a subforest $A \eqdef \{\shesqred,\shesqgreen\}$ with two (isomorphic) connected components indicated in red and green.
All extended decorations $\mfo$ in $\sigma$ are zero.
The boundary $\partial(A,T)=\{e\}$ consists of a single edge (coloured black) joining $\noise$ to the root of $\shesqgreen$.
Suppose $\mfe_A(e)=m$ and $\mfn_A(\rho)=b \leq a$, where $\rho$ is the root of $\sigma$, and $\mfn_A(x)=0$ for all $x\in N_A\setminus\{\rho\}$.
Then
\begin{equ}
\mcE(\sigma;A,\mfn_A,\mfe_A) = \shesqredb \;\fprod \shesqgreenm \in\mfF \subset \mfF^\ex\;.
\end{equ}
Furthermore
\begin{equ}
\mcC(\sigma;A,\mfn_A,\mfe_A) = \intshepa \in \mfT^\ex\;,
\end{equ}
where `$m$' indicates a kernel edge with decoration $\mcI_m$, and $\bone^p$ and $\bone^q$ indicate extended decorations at the corresponding nodes with
\begin{equ}
p = \Big|\shesqredb\Big|_\s\;\quad \text{and} \quad
 q= \Big|\shesqgreenm\Big|_\s\;.
\end{equ}
\end{example}

%Observe that $\mcQ$ extends to a map $\mfF^\ex\to\mfF$ again by setting all extended decorations to $0$.

\begin{definition}\label{def:mfT^ex_circ}
Define $\mfT^\ex_\circ\subset \mfT^\ex$ as the set of all trees of the form
\begin{equ}
\mcC(\tau;A,\mfn_A,\mfe_A)
\end{equ}
where $\tau=T^{\mfn,0}_\mfe\in\mfT_\circ\subset \mfT^\ex$,
$A\subset T$ is a subforest,
$\mfn_A \colon N_A \rightarrow \N^{d}$ with $\mfn_A\leq \mfn$ and $ \mfe_{A} \colon \partial(A,T) \rightarrow \N^{d}$,
such that
\begin{equ}
\mcE(\tau;A,\mfn_A,\mfe) \in \mfF_-\;.
\end{equ}
We denote $\mcT^\ex = \Span_\R(\mfT^\ex_\circ)$.
\end{definition}
In other words, $\mfT^\ex_\circ$ is the set of all trees obtained by extracting negative subforests from trees in $\mfT_\circ$ and adding extended decorations according to the procedure in~\eqref{eq:contract_def}.
In particular, $\mfo\leq0$ for all $T^{\mfn,\mfo}_\mfe\in\mfT_\circ^\ex$.
\begin{remark}
Since $R$ is complete (see Remark~\ref{rem:complete}), we indeed have
$\mfT_\circ\subset \mfT^\ex_\circ$ and $\mcQ\mfT^\ex_\circ = \mfT_\circ$ as mentioned before.
\end{remark}
While Definition~\ref{def:mfT^ex_circ} is relatively straightforward, it is not immediately obvious how to verify if a given tree $\tau\in\mfT^\ex$ is in $\mfT^\ex_\circ$.
The following exercise makes this task easier by showing that the property $\tau\in\mfT^\ex_\circ$ is essentially `local' and can be formulated in a way similar to the definition of strongly conforming to $R$ (Definition~\ref{def:conform}).
This characterisation of $\mfT^\ex_\circ$ is similar to the definition in~\cite[Def.~5.26]{BHZ19} (see the set $B_\circ$ therein).

\begin{exercise}\label{ex:trees_in_mfT^ex_circ}
For $\tau\in\mfT^\ex$ and $x\in N_\tau$, we define $\CN(x)$ exactly as in Definition~\ref{def:conform}.
For $\tau$ of the form~\eqref{eq:general_tree_ext} we further define $
\mcO(\tau) = a
$
and recall the notation
\begin{equ}
\CN(\tau) \eqdef \CN(\rho_\tau) = \{\Xi_l \,:\, l\in L\} \sqcup \{\CI_{m_j}\,:\, j\in J\} \in \N^\CE\;.
\end{equ}
Finally, define the function
\begin{equ}
\tilde R\colon \R\to \CP(\N^\CE)\;,\quad
\tilde R(a) = \{\CN(\tau)\,:\, \tau \in \mfT^\ex_\circ\,,\, \CO(\tau)=a\}\;,
\end{equ}
where $\CP(\Omega)$ denotes the power set of $\Omega$.
That is, $\tilde R(a)$ is the collection of all edge decorations that leave the root of any tree in $\mfT^\ex_\circ$ with extended decoration $a$ at the root.

Show that, for $\tau=T^{\mfn,\mfe}_\mfo\in\mfT^\ex$, the following statements are equivalent.

\begin{itemize}
\item $\tau \in \mfT^\ex_\circ$.

\item For any $x\in N_T$, one has $\CN(x) \in \tilde R (\mfo(x))$.
\end{itemize}
[Hint: the direction $\Rightarrow$ is easy.
For the converse, proceed by induction on the number of edges and use the fact that strongly conforming to $R$ is a local property.]
\end{exercise}
%
%\begin{proof}
%The direction $\Rightarrow$ is clear.
%To prove the converse, we proceed by induction on the number of edge in $\tau$.
%Suppose $\CN(x) \in \tilde R (\mfo(x))$ for all $x\in N_T$, where $\tau=T^{\mfn,\mfo}_\mfe$ is of the form~\eqref{eq:general_tree_ext}.
%Then, by induction, $\tau_j\in\mfT^\ex_\circ$ for all $j\in J$, and therefore
%there exist $\sigma_j \in \mfT_\circ$ and subforests $A_j\subset \sigma_j$ with functions $\mfn_{A_j},\mfe_{A_j}$
%such that
%$\tau_j = \mcC(\sigma_j;A_j,\mfn_{A_j},\mfe_{A_j})$.
%Furthermore, since $\CN(\tau)\in \tilde R(\CO(\tau))$,
%there exists $\tilde \tau \in\mfT^\ex_\circ$ of the form
%\begin{equ}
%\tilde \tau = \bone^a \Xi_L \prod_{j\in J} \mcI_{m_j}[\tilde\tau_j]
%\end{equ}
%and therefore there exists $\tilde \sigma\in\mfT_\circ$ with $\tilde A\subset \tilde \sigma$ such that $\tilde \tau = \CC(\tilde\sigma;\tilde A,\ldots)$.
%Let $\tilde A_1 \subset \tilde A$ be the connected component of $\tilde A$ that contains the root.
%Observe that the edge leaving $\tilde A_1$ have decorations precisely $\CN(\tilde\tau) = \CN(\tau)$.
%Since conforming to $R$ is a local property,
%we can substitute all the trees with roots at the end of these edges by $\sigma_j$ and the resulting tree is still in $\mfT_\circ$.
%Finally, we see that $\tau\in\mfT^\ex_\circ$ because we can extract $\tilde A_1$ together with $\{A_j\}_{j\in J}$ and the extracted forest is in $\mfF_-$.
%\end{proof}
%
By analogy to Definition~\ref{def:neg_trees},
we define\footnote{Our definition of $\mfT^\ex_-$ differs from the analogous space of~\cite[Eq.~(5.23)]{BHZ19} in three ways:
(i) we do not allow polynomial decorations at the root,
(ii) we do not allow trees of the form $\Xi_l$, i.e. trees made of a single noise edge, and
(iii) we do not allow trees of the form $\bone^p\mcI_m\tau$ with $p\neq 0$.
The fact that these possibilities were allowed in~\cite{BHZ19} is a typo and leads to minor issues in some proofs,
namely (i)-(ii) are needed for Lemma~\ref{lem:admissible} and (iii) is needed for Exercise~\ref{ex:reduced_identity}.
These two results correspond to the first statement of~\cite[Thm.~6.16]{BHZ19} and to~\cite[Eq.~(6.34)]{BHZ19} respectively.}
\begin{equ}
\mfT_-^\ex \eqdef \{\tau=F^{\mfn,\mfo}_\mfe \in\mfT_\circ^\ex\,:\, |\tau|_\s<0 \,,\,
\mfn(\rho_\tau)=0\,,\,
%\mfo(\rho_\tau)=0\,,\,
\tau \textnormal{ is not planted}
\}\;,
\end{equ}
where we say that $\tau\in\mfT^\ex$ is planted if there exists $p\in\R$ such that either $\tau=\bone^p\Xi_l$ for some $l\in\N^d$ or $\tau=\bone^p\mcI_m\sigma$ for some $m\in\N^d$ and $\sigma\in\mfT^\ex$.
That is, $\tau$ is planted if and only if $\mcQ\tau$ is planted.

We then define $\mcT_-^\ex\eqdef\Span_\R(\mcT_-^\ex)$ and $\mfF_-^\ex$, $\mcF_-^\ex$, and $G_-^\ex$ exactly as in Definition~\ref{def:neg_trees}
upon replacing $\mfT_-$ by $\mfT^\ex_-$.
In particular, $\mcF_-^\ex$ is the free commutative algebra generated by $\mfT_-^\ex$
and $G_-^\ex$ is the set of characters on $\mcF^\ex_-$.
\begin{definition}
For $\tau=T^{\mfn,\mfo}_{\mfe}\in\mfT_\circ^\ex$, we define $\Deltam_\ex  \tau \in\mcF_-^\ex\otimes \mcT^\ex$
\begin{equ}\label{eq:Deltam_ex}
 \Deltam_\ex \tau = 
 \sum_{A \subset T} \sum_{\mfe_A,\mfn_A}  \frac1{\mfe_A!}
\binom{\mfn}{\mfn_A}
 \pi_-^\ex\mcE(\tau;A,\mfn_A,\mfe_A)\otimes\mcC(\tau;A,\mfn_A,\mfe_A)\;,
\end{equ}
where $\pi_-^\ex\colon\mfF^\ex\to \mcF^\ex_-$ annihilates every forest in $\mfF^\ex\setminus \mfF^\ex_-$.
We extend $\Deltam_\ex$ to a linear map $\Deltam_\ex\colon\mcT^\ex\to \mcF_-^\ex\otimes \mcT^\ex$.
\end{definition}

\begin{example}\label{ex:Deltam_ex}
Following Example~\ref{ex:cointeract_fails},
consider again the pitchfork tree
\begin{equ}
\tau = \fork \in \mfT_\circ \subset \mfT^\ex_\circ\;.
\end{equ}
Then
\begin{equ}%\label{eq:tau_extract_extended}
\Deltam_\ex
\fork =\e\otimes 
\fork + 3\,
\shesq
\otimes
\intsheext
 +
\shecube
\otimes
\intoneext\;,
\end{equ}
where
\begin{equ}
p = \Big| \shesq \Big|_\s = -1-\;,\qquad q = \Big| \shecube \Big|_\s = -\frac32-\;.
\end{equ}
\end{example}

\begin{remark}
It is not obvious that $\Deltam_\ex$ indeed maps $\mcT^\ex$ to $\mcF_-^\ex\otimes \mcT^\ex$. Luckily,
this turns out to be the case because our rule $R$ is complete (see Remark~\ref{rem:complete}).
\end{remark}
The following exercise shows that $\Deltam$ and $\Deltam_\ex$ differ only in that $\Deltam_\ex$ adds extended decorations. 
\begin{exercise}\label{ex:reduced_identity}
Show that, for all $\tau\in\mfT^\ex_\circ$,
\begin{equ}
\Deltam\mcQ =  (\mcQ\otimes\mcQ)\Deltam_\ex\;.
\end{equ}
\end{exercise}
We now proceed to define the structure group $G_+^\ex$.
We define a new notion
of degree $|\cdot|_+$ on $\mfT^\ex$ by\label{deg_ex_page_ref}
\begin{equ}
|\tau|_+ \eqdef |\tau|_\s+\sum_{x\in N_\tau} \mfo(x)\;. 
\end{equ}
Note that $|\cdot|_+=|\cdot|_\s$ on $\mfT$, but in general $|\cdot|_+\leq |\cdot|_\s$ on $\mfT_\circ^\ex$.
Furthermore, and of crucial importance for the sequel,
$\Deltam_\ex$ preserves the $|\cdot|_+$-degree on the right factor, i.e. for all $A,\mfn_A,\mfe_A$ appearing in the sum~\eqref{eq:Deltam_ex},
\begin{equ}\label{eq:preserve_+}
|\mcC(\tau;A,\mfn_A,\mfe_A)|_+ = |\tau|_+\;.
\end{equ}
Similar to Definition~\ref{def:mfT_plus},
we define $\mfT_+^\ex\subset\mfT^\ex$ as the set of all trees $\tau\in\mfT^\ex$ of the form
\begin{equ}
\tau = X^k\prod_{j\in J}\mcI_{m_j}[\tau_j]\;,
\end{equ}
where now $\tau_j\in\mfT^\ex_\circ$ and $|\mcI_{m_j}\tau_j|_+>0$ for all $j\in J$.
We define the commutative algebra $\mcT_+^\ex=\Span_\R(\mfT^\ex_+)$ and characters $G_+^\ex$ exactly as in Definition~\ref{def:mfT_plus}.
\begin{definition}\label{def:Deltap_ex}
Define $\Deltap_\ex\colon \CH \to \CH\otimes \CT^\ex_+$ for $\CH\in\{\CT^\ex,\CT_+^\ex\}$ precisely as before using~\eqref{eq:Deltap_base_case}+\eqref{eq:Deltap_multiplicative}+\eqref{eq:Deltap_def}, replacing $\CT,\CT_+$ by $\CT^\ex,\CT_+^\ex$ and $|\cdot|_\s$ by $|\cdot|_+$,
and making the additional definition
\begin{equ}
\Deltap_\ex\bone^p = \bone^p\otimes \bone\;.
\end{equ}
\end{definition}
Effectively the only change these definitions entail is that the number of terms in the sums in~\eqref{eq:Deltap_def} and~\eqref{eq:Deltap_Tp} will now depend on $|\mcI_m\tau|_+$ instead of $|\mcI_m\tau|_\s$.
As far as $\Deltap_\ex$ is concerned, the extended decoration $\bone^p\in\mfT^\ex_\circ$ is a static variable treated like a noise (vs. $\Deltam_\ex$ for which $\mfo$ is a dynamic variable, see~\eqref{eq:contract_def} and~\eqref{eq:Deltam_ex}).

\begin{example}\label{ex:Deltap_ex}
Following Example~\ref{ex:Deltam_ex},
\begin{equs}
(\id\otimes\Deltap_\ex)\Deltam_\ex\fork
&= (\id\otimes\Deltap_\ex)(\e\otimes 
\fork + 3\,
\shesq
\otimes
\intsheext
 +
\shecube
\otimes
\intoneext)
\\
&= \e
\otimes(\fork\otimes \bone + \bone\otimes \fork)
\\
&\qquad+ 3\,\shesq\otimes\Big(\intsheext\otimes \bone + \bone \otimes\intsheext\Big)
\\
&\qquad
+ \shecube\otimes
\Big(
\intoneext\otimes \bone +
\bone\otimes\intoneext
\Big)\;.
\end{equs}
Nodes without labels, as usual, have extended and polynomial decoration $\bone^0X^0$.
Comparing this to~\eqref{eq:Deltap_after_Deltam}, we see that no polynomials $X^k$ with $|k|_\s=1$ appear because
\begin{equ}
\Big|\intsheext\Big|_+
=
\Big|\intoneext\Big|_+
=
\Big|\fork\Big|_+ \in (0,1]\;.
\end{equ}
Remark how we achieved our goal highlighted after Example~\ref{ex:cointeract_fails}.
See also Exercise~\ref{ex:cointeract_example}.
\end{example}
With these definitions, the statement (and proof) of Theorem~\ref{thm:pos_Hop_alg} remains true upon replacing $\{\CT,\CT_+,\Deltap\}$ by $\{\CT^\ex,\CT_+^\ex,\Deltap_\ex\}$.
In particular $\mcT_+^\ex$ is a commutative Hopf algebra with coproduct $\Deltap_\ex$ and $\mcT^\ex$ is a comodule over $\mcT^\ex_+$.
The following is the `extended' version of Definition~\ref{def:reduced_RS}.
\begin{definition}\label{def:extended_structure}
The triple $(A^\ex,\mcT^\ex,G^\ex_+)$ is called the \textit{extended regularity structure},
where $A^\ex = \{|\tau|_+\,:\,\tau\in\mfT_\circ^\ex\}$ is the index set,
$\mcT^\ex$ is the model space, and $G^\ex_+$ is the structure group, which acts on $\mcT^\ex$ by \begin{equ}
\mcT^\ex\times G^\ex_+ \ni (\tau,f)\mapsto (\id\otimes f)\Deltap_\ex\tau \in \mcT^\ex\;.
\end{equ}
\end{definition}
\begin{remark}
The regularity structure $(A,\mcT,G_+)$ (see Definition~\ref{def:reduced_RS}) is contained in the extended structure $(A^\ex,\mcT^\ex,G^\ex_+)$ in the sense of~\cite[Sec.~2.1]{Hairer14}.
In particular, the index set $A$ from Definition~\ref{def:mcT} is contained in $A^\ex$, and $A^\ex$ is locally finite and bounded from below by subcriticality of $R$ and the definition of $\mfT^\ex_\circ$.
We will describe the relation between the two structures further in Section~\ref{subsubsec:pos_ext_reduced}.
\end{remark}
Finally, we extend the domain and range of $\Deltam_\ex$ to $\CH\to \mcF_-^\ex\otimes\CH$ for $\CH\in \{\mcF_-^\ex,\mcT_+^\ex\}$ exactly like in Definition~\ref{def:Deltam_CH}
except replacing expression~\eqref{eq:Deltam_def}
by~\eqref{eq:Deltam_ex}.
In particular, $\Deltam_\ex$ is multiplicative on both $\mcT^\ex_+$ and $\mcF_-^\ex$.

With these definitions (and at this point unsurprisingly) the `extended' version of Theorem~\ref{thm:negative_Hopf} holds where we replace every $\CH\in\{\CT,\CT_+,\mcF_-\}$ by $\CH^\ex$ and $\Deltam$ by $\Deltam_\ex$.

With all of these notations and preliminaries in place, we can state one of the main algebraic results of~\cite{BHZ19}, which is the desired cointeraction.
\begin{theorem}\label{thm:cointerct_ex}[Theorem~5.37 of~\cite{BHZ19}]
Consider $\mcH\in\{\mcT^\ex,\mcT_+^\ex\}$ and define
\begin{equ}
\mcM_{(13)(2)(4)}\colon \mcF_-^\ex\otimes\mcH\otimes \mcF_-^\ex\otimes\mcT_+^\ex\to \mcF_-^\ex\otimes\mcH\otimes\mcT_+^\ex
\end{equ}
by the expression~\eqref{eq:mcM_def}.
Then
\begin{equation}\label{eq:cointeraction_ex}
\mcM_{(13)(2)(4)} (\Deltam_\ex\otimes\Deltam_\ex)\Deltap_\ex = (\id\otimes\Deltap_\ex)\Deltam_\ex
\end{equation}
as maps $\mcH\to \mcF_-^\ex\otimes\mcH\otimes\mcT_+^\ex$.
\end{theorem}
\begin{exercise}\label{ex:cointeract_example}
Verify directly that, for $\tau\in\mfT_\circ$ as in Examples~\ref{ex:Deltam_ex},
\begin{equ}
\mcM_{(13)(2)(4)} (\Deltam_\ex\otimes\Deltam_\ex)\Deltap_\ex \tau = (\id\otimes\Deltap_\ex)\Deltam_\ex \tau\;.
\end{equ}
[Hint: follow the computation in Example~\ref{ex:cointeract_fails}, making the relevant changes as in Example~\ref{ex:Deltap_ex}. Note in particular that $\Deltap_\ex \tau = \tau\otimes \bone + \bone\otimes \tau$.]
\end{exercise}
We do not prove Theorem~\ref{thm:cointerct_ex} (see, however,~\cite[Thm.~8]{CEFM11} for a proof of essentially the same property in a simpler setting with no decorations or projections),
but we will explain how to derive Theorem~\ref{thm:pre-model_renorm} from it.
First, we will state and prove the `extended' version of Theorem~\ref{thm:pre-model_renorm}.

\subsection{Renormalisation of extended pre-models}
For $f\in G_+^\ex$, we let $\Gamma^\ex_f\in\GL(\mcT^\ex)$ be defined by
\begin{equ}
\Gamma^\ex_f = (\id\otimes f) \Deltap_\ex\;.
\end{equ}
By a simple extension of Exercise~\ref{ex:prod_G_+},
$G_+^\ex$ is a group with product $f\circ \bar f = (f\otimes\bar f)\Deltap_\ex$
and $f\mapsto \Gamma^\ex_f \in\GL(\mcT^\ex)$ is a group homomorphism.
For any linear map $\bfPi\colon \mcT^\ex\to C^\infty$ and $x\in \R^d$, we define
$\Pi_x \eqdef \bfPi \Gamma^\ex_{f_x}$,
where $f_x$ is defined as in Section~\ref{sec:pos_renorm} upon replacing $\Deltap$ by $\Deltap_\ex$
(which is well-defined by the same reasoning as in Remark~\ref{rem:models_well_defined}).
As before, we call $\{\Pi_x\}_{x\in\R^d}$ the family of recentred maps and $\{f_x\}_{x\in\R^d}\subset G_+^\ex$ the family of positive characters of $\bfPi$.

We denote by $\mcZ^\ex(\bfPi)=(\Pi,\Gamma)$ the families $\{\Pi_x\}_{x\in\R^d}$ and $\{\Gamma_{xy}\}_{x,y\in\R^d}$ where $\Gamma_{xy}\eqdef \Gamma^\ex_{f_x^{-1}\circ f_y}$ for $x,y\in\R^d$.
We make the following analogue of Definitions~\ref{def:admissible} and~\ref{def:pre_model}.
\begin{definition}\label{def:admissible_ex}
An \textit{extended admissible map} is a linear map $\bfPi\colon\mcT^\ex\to C^\infty$ that
satisfies~\eqref{eq:admiss_prop_1}-\eqref{eq:admiss_prop_2} for all $\tau\in\mfT_\circ^\ex$ and $k,l,m\in\N^d$.
An extended admissible map $\bfPi$ is called an \textit{extended pre-model} if $\mcZ^\ex(\bfPi)$ is a model on the extended regularity structure $(A^\ex,\mcT^\ex,G_+^\ex)$.
\end{definition}
Note that we do not yet make any assumption on the value of $\bfPi$ on trees of the form $\bone^a\tau$ for $a\neq 0$.
However, because $\bone^a$ is a `bookkeeping' decoration only seen by $\Deltap_\ex$, it is natural for it to not impact the value of $\bfPi$.
We will later call admissible maps for which $\bfPi\bone^a\tau=\bfPi\tau$ \textit{reduced}.

\begin{example}\label{ex:reduced_lift}
Suppose $\bfPi\colon\mcT\to C^\infty$ is an admissible map.
Then $\bfPi^\ex\tau \eqdef \bfPi\mcQ\tau$  defines an extended admissible map $\bfPi^\ex\colon\CT^\ex\to C^\infty$. 
Furthermore, if $\bfPi^\ex$ is an extended pre-model, then clearly $\bfPi$ is a pre-model because $\mcZ(\bfPi)$ is the restriction of $\mcZ^\ex(\bfPi^\ex)$ to the reduced regularity structure.
\end{example}
\begin{remark}\label{rem:pre-model}
The converse of the final statement of Example~\ref{ex:reduced_lift}, namely that if $\bfPi$ is a pre-model then $\bfPi^\ex$ is an extended pre-model, is also true, although this is less obvious, see~\cite[Thm.~6.32]{BHZ19}.
\end{remark}
The following is an analogue of Theorem~\ref{thm:pre-model_renorm}.
\begin{theorem}\label{thm:pre-model_ex_renorm}[Theorem~6.16 of~\cite{BHZ19}]
Suppose $g\in G_-^\ex$ and that $\bfPi$ is an extended pre-model.
Then $\bfPi^g \eqdef (g\otimes \bfPi)\Deltam_\ex$ is also an extended pre-model.
\end{theorem}
As in Remark~\ref{rem:renorm_cont}, the map $\mcZ^\ex(\bfPi) \mapsto \mcZ^\ex(\bfPi^g)$ is furthermore continuous.
In the rest of this subsection, we prove Theorem~\ref{thm:pre-model_ex_renorm}.
The main ingredient in the proof is Lemma~\ref{lem:recentred_simple} which provides a simple relationship between the recentred maps and positive characters of $\bfPi^g$ and those of $\bfPi$,
which in turn relies on the cointeraction~\eqref{eq:cointeraction_ex}.
We will emphasise throughout the proof which statements do not rely on the extended structure,
i.e. the statements which hold with all spaces $\CH^\ex$ and maps $\Deltapm_\ex$ replaced by $\CH$ and $\Deltapm$ respectively.

For the rest of the subsection, let us fix $g\in G_-^\ex$ and a linear map $\bfPi\colon\CT^\ex\to C^\infty$.
For $\CH\in\{\CT^\ex,\CT^\ex_+\}$, let $M_g\colon\CH\to \CH$ denote the map\label{M_g_page_ref}
\begin{equ}
M_g = (g\otimes \id) \Deltam_\ex\;,
\end{equ}
so that $\bfPi^g = \bfPi M_g$.
One of the key properties implied by~\eqref{eq:cointeraction_ex} is that $M_g$ defines a group homomorphism on $G_+^\ex$ as seen in the following lemma.

\begin{lemma}\label{lem:G_m_action}
Let $f,\bar f \in G_+^\ex$. Then $f M_g\in G_+^\ex$ and
\begin{equ}\label{eq:G_m_action}
(f\circ \bar f)M_g = (f M_g)\circ (\bar f M_g)\;.
\end{equ}
\end{lemma}

\begin{proof}
The fact that $f M_g\in G_+^\ex$ (which does not rely on the extended structure), follows from writing $fM_g=(g\otimes f)\Deltam_\ex$ and using the facts that $\Deltam_\ex$ and $f$ are multiplicative on $\mcT_+^\ex$ and that $g$ is multiplicative on $\CF_-^\ex$.
To prove~\eqref{eq:G_m_action},
\begin{equs}
(f\circ\bar f) M_g &= (f\otimes \bar f)\Deltap_\ex(g\otimes \id)\Deltam_\ex = (g\otimes f\otimes\bar f)(\id \otimes\Deltap_\ex)\Deltam_\ex
\\
&= (g\otimes f\otimes g\otimes \bar f)(\Deltam_\ex\otimes\Deltam_\ex)\Deltap_\ex
= (f M_g\otimes \bar f M_g)\Deltap_\ex
\\
&= (f M_g)\circ (\bar f M_g) \;,
\end{equs}
where we used~\eqref{eq:cointeraction_ex} with $\CH=\mcT^\ex_+$ in the third equality together with multiplicativity of $g$ on $\CF_-^\ex$.
\end{proof}
The following exercise reveals another useful consequence of~\eqref{eq:cointeraction_ex}.
\begin{exercise}
Show that
\begin{equ}\label{eq:M_g_commute}
\Gamma^\ex_f M_g = M_g \Gamma^\ex_{f M_g}\;.
\end{equ}
\end{exercise}
The following lemma does not rely on the extended structure.
\begin{lemma}\label{lem:admissible}
Suppose $\bfPi$ is an (extended) admissible map.
Then
$\bfPi^g$ is also an (extended) admissible map.
\end{lemma}

\begin{proof}
Clearly $\bfPi^g\bone=1$ since $\bfPi\bone=1$ and $\Deltam_\ex\bone = \e\otimes \bone$.
Furthermore, $\Deltam_\ex X^k\tau = (\id\otimes X^k)\Deltam_\ex\tau$ for any $X^k\tau\in\mfT_\circ^\ex$
because $\mfT_-^\ex$ does not contain trees with non-zero polynomial decorations at the root.
Hence $\bfPi^gX^k\tau = \cdot^k \bfPi^g\tau$, which proves~\eqref{eq:admiss_prop_1} for $\bfPi^g$.

Next, for any $l\in\N^d$, we have $\Deltam_\ex \Xi_l = \e\otimes\Xi_l$ because $\mfT_-^\ex$ does not contain planted trees or trees with only one node.
Therefore $\bfPi^g\Xi_l = \bfPi \Xi_l = D^l \xi$ by admissibility of $\bfPi$.

It remains to prove the second identity in~\eqref{eq:admiss_prop_2} for $\bfPi^g$.
Consider $T^{\mfn,\mfo}_\mfe=\mcI_m\tau\in\mfT^\ex_\circ$.
Since $\mfT_-^\ex$ does not contain planted trees or trees with only one node, $\pi_-$ in the sum~\eqref{eq:Deltam_def}
will annihilate every term in which $A$ contains the root, and thus
\begin{equ}\label{eq:Deltam_planted}
\Deltam_\ex\mcI_m\tau = (\id\otimes\mcI_m)\Deltam_\ex\tau\;.
\end{equ}
The second identity in~\eqref{eq:admiss_prop_2} for $\bfPi^g$ then follows from
\begin{equ}
\bfPi^g \mcI_m\tau = (g\otimes \bfPi)\Deltam_\ex \mcI_m\tau = (g\otimes\bfPi\mcI_m)\Deltam_\ex \tau = D^m K* (g\otimes\bfPi)\Deltam_\ex \tau = D^m K* \bfPi^g\;,
\end{equ}
where we used admissibility of $\bfPi$ in the third equality.
Hence $\bfPi^g$ is admissible.
\end{proof}
We next show that there is a simple relationship between $\mcZ^\ex(\bfPi)$ and $\mcZ^\ex(\bfPi^g)$.
We henceforth denote $\mcZ^\ex(\bfPi) = (\Pi,\Gamma)$
and $\mcZ^\ex(\bfPi^g) = (\Pi^g,\Gamma^g)$ and let $f_x, f_x^g\in G_+^\ex$ be the positive characters of $\bfPi$ and $\bfPi^g$ for $x\in \R^d$.
The following lemma is essentially a consequence of the cointeraction property of~\eqref{eq:cointeraction_ex} together with the fact that $\Deltam_\ex$ preserves the $|\cdot|_+$-degree on the right factor due to~\eqref{eq:preserve_+}.

\begin{lemma}\label{lem:recentred_simple}
For all $\tau\in\mfT_\circ^\ex$
\begin{equ}\label{eq:direct_form_Pi_x}
\Pi^g_x \tau = \Pi_x M_g \tau\;,
\end{equ}
and for all $\tau\in\mfT_+^\ex$
\begin{equ}\label{eq:direct_form_f_x}
f^g_x \tau = f_x M_g \tau\;.
\end{equ}
\end{lemma}

\begin{proof}
We first claim that if $\tau\in\mfT_+^\ex$ with $\fancynorm{\tau}=0$, then $M_g\tau=\tau$ and $f^g_x(\tau)=f_x(\tau)$.
Indeed, in this case $\tau=X^k$ for some $k\in\N^d$, and clearly $M_gX^k = X^k$ and $f_x(X^k)=f^g_x(X^k)$ by definition.

We next claim that, if $\tau\in\mfT_\circ^\ex$ with $\fancynorm{\tau}=0$, then
$\Pi^g_x \tau = \Pi_x M_g \tau$.
(This claim, as the previous one, does not rely on the extended structure -- see Proposition~\ref{prop:direct_renorm_base_case}.)
Indeed, in this case there exist $a\leq 0$, $k\in\N^d$, and $L\in\mfN$ such that $\tau = \bone^aX^k\Xi_L$.
Hence
\begin{equ}\label{eq:Deltap_ex_base_case}
\Deltap_\ex \tau
= \sum_{n \leq k}\binom kn \bone^a X^n\Xi_L\otimes X^{k-n}\;,
\end{equ}
and therefore
\begin{equs}
\Pi_xM_g \tau
&= (g\otimes \bfPi \otimes f_x)(\id\otimes \Deltap_\ex)\Deltam_\ex\tau = (g\otimes \bfPi \otimes g\otimes f_x)(\Deltam_\ex\otimes \Deltam_\ex)\Deltap_\ex\tau
\\
&= (g\otimes \bfPi \otimes f^g_x)(\Deltam_\ex\otimes \id)\Deltap_\ex\tau = \bfPi^g  F^g_x\tau = \Pi_x^g\tau\;,
\end{equs}
where we used~\eqref{eq:cointeraction_ex} in the second equality and~\eqref{eq:Deltap_ex_base_case} together with
$(g\otimes f_x)\Deltam_\ex X^n = f_x(X^n)=f^g_x(X^n)$ for all $n\in\N^d$ in the third equality.

We now prove~\eqref{eq:direct_form_Pi_x}-\eqref{eq:direct_form_f_x} by induction.
Suppose $n\geq 0$ and that~\eqref{eq:direct_form_Pi_x} and~\eqref{eq:direct_form_f_x} hold for all $\tau\in\mfT_\circ^\ex$ and $\tau\in\mfT^\ex_+$ respectively with $\fancynorm{\tau}\leq n$.

Consider any $\tau\in\mfT_\circ^\ex$ with $\fancynorm{\tau}\leq n$.
We first show that
\begin{equ}
f_x^g(\mcI_m\tau) = f_xM_g\mcI_m\tau\;.
\end{equ}
Indeed, by definition of $f_x^g$ (i.e.~\eqref{eq:pos_char_def} with $|\cdot|_\s$ replaced by $|\cdot|_+$)
we have
\[
f_x^g(\mcI_m\tau) = -\sum_{|\ell|_\s<|\mcI_m\tau|_+}
\frac{(-x)^\ell}{\ell!} D^{m+\ell} K*\Pi^g_x\tau(x)\;.
\]
By the induction hypothesis, $\Pi^g_x\tau=(g\otimes \Pi_x)\Deltam_\ex\tau$,
and therefore
\begin{equ}\label{eq:f^g_x_1}
f^g_x(\mcI_m\tau) =  -\sum_{|\ell|_\s<|\mcI_m\tau|_+}
\frac{(-x)^\ell}{\ell!} (g\otimes D^{m+\ell} K*\Pi_x)(\Deltam_\ex\tau)(x)\;.
\end{equ}
On the other hand, using~\eqref{eq:Deltam_planted}, we have $M_g\mcI_m\tau = \mcI_m M_g\tau$,
and thus
\begin{equ}\label{eq:M_g_I_commute}
f_xM_g \mcI_m\tau = f_x \mcI_m M_g\tau\;.
\end{equ}
Writing $\Deltam_\ex \tau = \sum\tau_1\otimes \tau_2$ in Sweedler notation,
it follows from the fact that
$\Deltam_\ex$ preserves the $|\cdot|_+$ degree of the right factor, i.e. $|\tau_2|_+=|\tau|_+$ due to~\eqref{eq:preserve_+}, that
\begin{equ}\label{eq:f_x_tau_2}
f_x(\mcI_m\tau_2) = -\sum_{|\ell|_\s<|\mcI_m\tau|_+}
\frac{(-x)^\ell}{\ell!} (D^{m+\ell} K*\Pi_x\tau_2)(x)\;.
\end{equ}
Comparing~\eqref{eq:M_g_I_commute}-\eqref{eq:f_x_tau_2} with~\eqref{eq:f^g_x_1}, we obtain $f_x^g(\mcI_m\tau)= f_x M_g\mcI_m\tau$ as claimed.

Next, since $f^g_x$ and $f_xM_g$ are both characters (Lemma~\ref{lem:G_m_action}) and agree on planted trees with at most $n+1$ kernel edges,
it follows that
\begin{equ}\label{eq:f^g_x_strong_induction}
f^g_x\tau = f_xM_g\tau \quad \forall \tau\in\mfT^\ex_+\,,\, \fancynorm{\tau}\leq n+1\;.
\end{equ}
Finally, consider $\tau\in\mfT^\ex_\circ$ with $\fancynorm{\tau}\leq n+1$.
We show that $\Pi^g_x \tau= \Pi_xM_g\tau$.
Indeed,
\begin{equ}
\Pi^g_x \tau = (\bfPi^g\otimes f^g_x)\Deltap_\ex\tau
=(g\otimes \bfPi \otimes g \otimes f_x)(\Deltam_\ex\otimes\Deltam_\ex)\Deltap_\ex\tau\;,
\end{equ}
where the first equality follows from definitions and, in the second equality, we used~\eqref{eq:f^g_x_strong_induction} and the fact that, if we write $\Deltap_\ex\mcI_m\tau = \sum_i c_i \tau^{(1)}_i\otimes\tau^{(2)}_i$,
then $\fancynorm{\tau_i^{(2)}}\leq n+1$.
Then, by the cointeraction property~\eqref{eq:cointeraction_ex}
\begin{equ}
\Pi^g_x \tau=(g\otimes \bfPi \otimes f_x)
(\id\otimes \Deltap_\ex)\Deltam_\ex\tau=(g\otimes \Pi_x)\Deltam_\ex\tau = \Pi_x M_g\tau
\end{equ}
as desired. By induction, we obtain~\eqref{eq:direct_form_Pi_x}-\eqref{eq:direct_form_f_x} for all $\tau$.
\end{proof}
The proof of Theorem~\ref{thm:pre-model_ex_renorm} is now straightforward.
\begin{proof}[of Theorem~\ref{thm:pre-model_ex_renorm}]
Recall the notation
$\mcZ^\ex(\bfPi) = (\Pi,\Gamma)$
and $\mcZ^\ex(\bfPi^g) = (\Pi^g,\Gamma^g)$.
The fact that $\bfPi^g$ is admissible follows from Lemma~\ref{lem:admissible}.
To show that $\mcZ^\ex(\bfPi^g)$ is a model, we use~\eqref{eq:direct_form_Pi_x}-\eqref{eq:direct_form_f_x}.
Indeed, writing $\Deltam_\ex\tau = \sum_i \tau^{(1)}_i\otimes \tau^{(2)}_i$ in Sweedler notation, we have
\begin{equ}
\Pi^g_x \tau = \sum_ig(\tau^{(1)}_i)\Pi_x\tau^{(2)}_i\;.
\end{equ}
Since $|\tau^{(2)}_i|_+ = |\tau|_+$ due to~\eqref{eq:preserve_+},
the required bounds on $\Pi^g_x$ follow at once.
Next, recall that $\Gamma^g_{xy} = \Gamma^\ex_{(f^g_x)^{-1}\circ f_y^g}$ by definition.
Therefore
\begin{equ}
M_g\Gamma^g_{xy} = M_g\Gamma^\ex_{(f^g_x)^{-1}\circ f_y^g} = M_g \Gamma^\ex_{(f_x^{-1}\circ f_y)M_g}
= \Gamma^\ex_{f_x^{-1}\circ f_y} M_g = \Gamma_{xy}M_g\;,
\end{equ}
where the first and last equalities follow by definition, the second equality follows from~\eqref{eq:direct_form_f_x} together with~\eqref{eq:G_m_action}, and the third equality follows from~\eqref{eq:M_g_commute}.
Therefore
\begin{equ}
\Gamma^g_{xy} = M_{g^{-1}}\Gamma_{xy}M_g\;,
\end{equ}
where $g^{-1}$ is the inverse of $g$ for the product on $G^\ex_-$ defined by $g\bar g \eqdef (\bar g\otimes g)\Deltam_\ex$ (see Exercise~\ref{ex:prod_G_-}, which carries over to the extended spaces).
The necessary bounds on $\Gamma_{xy}^g$ therefore again follow from $|\tau^{(2)}_i|_+ = |\tau|_+$.
\end{proof}

\subsection{Reduced regularity structure}
\label{subsec:reduced}

In this subsection, we describe the link between the extended regularity structure $(A^\ex,\CT^\ex,G_+^\ex)$ of Definition~\ref{def:extended_structure} with renormalisation group $G_-^\ex$
and the regularity structure $(A,\CT,G_+)$ with renormalisation group $G_-$ described in Section~\ref{sec:pos_renorm} and the beginning of this section.
In particular, we will prove Theorem~\ref{thm:pre-model_renorm}.
To clarify which structure we speak about, we call $(A,\CT,G_+)$ the \textit{reduced} regularity structure.
It is important to understand the link between the two structures because, although the extended structure has nicer algebraic properties, it is the reduced structure which is often used in practice.

\subsubsection{Positive Hopf algebra}
\label{subsubsec:pos_ext_reduced}

For the `positive' Hopf algebra, the natural relationship between the reduced and extended structures is given by
the injection $\iota \colon\mfT\hookrightarrow \mfT^\ex,
\iota\colon T^\mfn_\mfe\mapsto T^{\mfn,0}_\mfe$.
This injection is the way we implicitly treated $\mfT$ as a subset of $\mfT^\ex$.
We further obtain injections
\begin{equ}
\iota\colon\CT\hookrightarrow\CT^\ex\;,\qquad\iota\colon\CT_+\hookrightarrow\CT_+^\ex\;.
\end{equ}
These embeddings furthermore preserve degrees, i.e. $|\iota\tau|_+=|\tau|_\s$.
Because $\mfo$ is a static variable under $\Deltap_\ex$, $\iota$ is a Hopf algebra (resp. comodule) morphism on $\CT_+$ (resp. $\CT$), i.e.
\begin{equ}
\Deltap_\ex\iota = (\iota\otimes \iota)\Deltap
\end{equ}
as maps $\CH\to \CH^\ex\otimes\CT^\ex_+$ for $\CH\in\{\CT,\CT_+\}$.
%Therefore, the structure $(\CT,\CT_+)$ can be thought of as a substructure of $(\CT^\ex,\CT^\ex_+)$.
The embedding $\iota\colon\CT_+\hookrightarrow\CT^\ex_+$ induces a surjective group homomorphism
\begin{equ}
\iota^*\colon G^\ex_+\twoheadrightarrow G_+\;,\qquad
\iota^* f \eqdef f\iota\;,
\end{equ}
and one has for all $f\in G_+^\ex$
\begin{equ}
\Gamma^\ex_{f} \iota = \iota\Gamma_{\iota^*f}\;,
\end{equ}
where we recall that
\begin{equ}
\Gamma^\ex_f=(\id\otimes f)\Deltap_\ex \in \GL(\CT^\ex)\;,\qquad
\Gamma_f=(\id\otimes f)\Deltap \in \GL(\CT)\;.
\end{equ}
In particular, the action of $G_+^\ex$ on $\CT^\ex$ given by $\Gamma^\ex$ leaves $\iota\CT$ invariant and agrees with the action of $G_+$ on $\CT$ we saw in Section~\ref{sec:pos_renorm}.
In summary, the reduced regularity structure is contained in the extended structure.

Furthermore, consider any linear map  $\bfPi\colon\CT^\ex\to C^\infty$ with recentred maps $\Pi_x\colon\CT^\ex\to C^\infty$ and positive characters $f_x\in G_+^\ex$.
Then, because $|\cdot|_\s=|\cdot|_+$ on $\iota\mfT$,
the recentred maps and positive characters of $\bfPi\iota\colon \CT\to C^\infty$
are $\Pi_x\iota\colon \CT\to C^\infty$ and $\iota^*f_x\in G_+$ respectively.
In particular,
\begin{equ}\label{eq:bfPi_restriction}
\mcZ(\bfPi\iota) = \mcZ^\ex(\bfPi)\restr_{\iota\CT}
\end{equ}
(cf. Example~\ref{ex:reduced_lift}).
We can partially summarise these relationships in the following commutative diagram:
 \[
  \begin{tikzcd}[row sep=1.5cm]
\arrow[loop above]{l}{\Gamma_{\iota^* f}}
\mcT \arrow[r,hook,"\iota"]
\arrow[dr,"\bfPi\iota",dashed,swap] &
\arrow[loop above]{l}{\Gamma^\ex_{f}}\mcT^\ex \arrow[d,"\bfPi"]  \\
& C^\infty
 \end{tikzcd}
\qquad\qquad
\begin{tikzcd}[row sep=0.7cm]
G_+ \ni \iota^*f & \arrow[l,"\iota^*",two heads] f\in G_+^\ex \\
\CT_+ \arrow[r,"\iota",hook] &  \CT_+^\ex
 \end{tikzcd}
\]

\subsubsection{Negative Hopf algebra}
\label{subsubsec:neg_Hopf_reduced}

For the `negative' Hopf algebra, while there is an injection $\mcF_-\hookrightarrow\mcF^\ex_-$ given by the same mapping $T^\mfn_\mfe\mapsto T^{\mfn,0}_\mfe$,
this injection is \textit{not} a coalgebra morphism (and thus not a Hopf algebra morphism) because $\mfo$ is no longer static under $\Deltam_\ex$.
The natural relationship between $\CF_-$ and $\CF^\ex_-$ is given instead by the projections \begin{equ}
\mcQ\colon\CT^\ex \twoheadrightarrow\CT\;,\qquad \mcQ\colon\mcF_-^\ex \twoheadrightarrow \mcF_-\;.
\end{equ}
By Exercise~\ref{ex:reduced_identity}, the first projection is a comodule morphism and the second is a Hopf algebra morphism, i.e.
\begin{equ}\label{eq:Q_morphism}
\Deltam\mcQ = (\mcQ\otimes\mcQ)\Deltam_\ex
\end{equ}
as maps $\CH^\ex \to\CF_- \otimes\CH$ for $\CH\in\{\CT,\CF_-\}$.
(The same holds for $\mcQ\colon\mcT^\ex_+\to\mcT_+$,
but this is not important, see Remark~\ref{rem:useless_comodule}.)
The projection $\mcQ\colon\CT^\ex\twoheadrightarrow\CT$ is a left inverse of $\iota\colon\CT\hookrightarrow\CT^\ex$, i.e. $\CQ\iota=\id$.
Furthermore, the projection $\mcQ\colon\mcF_-^\ex \twoheadrightarrow \mcF_-$ induces an embedding of groups \begin{equ}
\mcQ^*\colon G_-\hookrightarrow G_-^\ex\;,\qquad \mcQ^*g \eqdef g\mcQ\;.
\end{equ}
To understand how the maps $\mcQ$ and $\mcQ^*$ interact with linear maps $\bfPi\colon\CT^\ex\to C^\infty$, we require the notion of reduced maps.
\begin{definition}
A linear map $\bfPi\colon \CT^\ex \to C^\infty$ is called \textit{reduced} if $\bfPi\iota\mcQ=\bfPi$, i.e. $\bfPi$ does not depend on the extended decoration $\mfo$.
\end{definition}
Clearly admissible maps on $\CT$ are in bijection with reduced admissible maps on $\CT^\ex$ through the mapping $\bfPi\mapsto \bfPi^\ex\eqdef \bfPi\mcQ$ of Example~\ref{ex:reduced_lift}.
Furthermore, and less trivially, this bijection restricts to a bijection between pre-models on $\CT$ and reduced pre-models on $\CT^\ex$ (see Remark~\ref{rem:pre-model}).

An important consequence of the morphism property~\eqref{eq:Q_morphism} together with $\CQ\iota=\id$
is that, for all $g\in G_-$,
\begin{equ}
\CQ M_{\CQ^* g}\iota = M_g
\end{equ}
as maps in $\GL(\CT)$.
Consequently, if $\bfPi^\ex\colon\CT^\ex\to C^\infty$ is a reduced linear map, then for all $g\in G_-$
\begin{equ}\label{eq:reduced_M_g}
\bfPi^\ex M_{\CQ^* g}\iota = \bfPi^\ex\iota M_g\;.
\end{equ}
Although we do not use it later, the following exercise shows that reduced maps are closed under the action of $\CQ^*G_-$.
\begin{exercise}\label{ex:reduced_preserve}
Show that $M_g\CQ = \CQ M_{\CQ^* g}$ for all $g\in G_-$, and therefore, if $\bfPi^\ex \colon \CT^\ex\to C^\infty$ is a reduced map, then so is $\bfPi^\ex M_{\CQ^* g}$.
\end{exercise}
These relationships are again summarised in the following (partially) commutative diagram:

\[
  \begin{tikzcd}[row sep=1.5cm]
\arrow[loop above]{l}{M_{g}}
\mcT \arrow[r,hook,"\iota", shift right=4pt,swap]
\arrow[dr,"\bfPi",swap] &
\arrow[loop above]{l}{M_{\mcQ^* g}}\mcT^\ex \arrow[d,"\bfPi^\ex=\bfPi\CQ",dashed] \arrow[l,two heads,"\mcQ",shift right=3pt,swap] \\
& C^\infty
 \end{tikzcd}
\qquad\qquad
\begin{tikzcd}[row sep=0.7cm]
G_- \ni g \arrow[r,"\CQ^*",hook] &  \CQ^* g\in G_-^\ex \\
\CF_-  &  \arrow[l,"\CQ",two heads] \CF_-^\ex
 \end{tikzcd}
\]

\begin{proof}[of Theorem~\ref{thm:pre-model_renorm}]
Let $g\in G_-$ and $\bfPi\colon\CT\to C^\infty$ a pre-model.
Define the reduced admissible map $\bfPi^\ex \eqdef \bfPi\mcQ$.
By Remark~\ref{rem:pre-model}, $\bfPi^\ex$ is a reduced pre-model.
Therefore, by Theorem~\ref{thm:pre-model_ex_renorm}, $\bfPi^\ex M_{\CQ^*g}$ is an extended pre-model,
i.e. $\mcZ^\ex(\bfPi^\ex M_{\CQ^*g})$ is a model on the extended regularity structure.
% ($\bfPi^\ex M_{\CQ^*g}$ is furthermore reduced by Exercise~\ref{ex:reduced_preserve}, but we do not use this).
Finally
\begin{equ}
\mcZ(\bfPi^g) = \mcZ(\bfPi^\ex \iota M_g)
= \mcZ (\bfPi^\ex M_{\CQ^* g} \iota)
= \mcZ^\ex(\bfPi^\ex M_{\CQ^*g})\restr_{\iota\CT}\;,
\end{equ}
where we used $\CQ\iota=\id$ in the first equality,~\eqref{eq:reduced_M_g} in the second equality, and~\eqref{eq:bfPi_restriction} in the third equality.
Since the restriction $\mcZ^\ex(\bfPi^\ex M_{\CQ^*g})\restr_{\iota\CT}$ is again a model, $\mcZ(\bfPi^g)$ is a model on the reduced regularity structure as desired.
\end{proof}

\begin{remark}\label{rem:useless_comodule}
All the results for the reduced spaces in Section~\ref{sec:pos_renorm} and the start of Section~\ref{sec:neg_renorm}
are used in the statement and proof of Theorem~\ref{thm:pre-model_renorm} \textit{except} the fact that $\mcT_+$ is a comodule over $\mcF_-$ in Theorem~\ref{thm:negative_Hopf}
-- this property turned out useless because the naive cointeraction~\eqref{eq:cointeraction} failed.
On the other hand, the fact that $\CT_+^\ex$ is a comodule over $\CF_-^\ex$ is crucial.
\end{remark}

\begin{remark}
To summarise, in our proof of Theorem~\ref{thm:pre-model_renorm}, the statements that we did not prove here are Theorem~\ref{thm:negative_Hopf}, the `extended' versions of Theorems~\ref{thm:pos_Hop_alg} and~\ref{thm:negative_Hopf} (see the discussion before Theorem~\ref{thm:cointerct_ex}),
the result of Remark~\ref{rem:pre-model}, and the cointeraction property Theorem~\ref{thm:cointerct_ex}.
The interested reader can find proofs of these statements in~\cite{BHZ19}.
\end{remark}

\section{Renormalised SPDEs}
\label{sec:renorm_SPDE}

We saw in Section~\ref{sec:neg_renorm} that the renormalisation group $G_-^\ex$ acts on extended (pre-)models (Theorem~\ref{thm:pre-model_renorm}).
We now derive the final missing piece in order to understand renormalised SPDEs, which is the adjoint action of the renormalisation group on non-linearities identified in~\cite{BCCH21}.
This action is given in Theorem~\ref{thm:MU_coherent} below.

Recall that $\mfL=\{\Xi,\mcI\}$ is the set of type decorations, $\mcE = \mfL\times \N^d$
is the set of edge types,
and that a non-linearity is a smooth map $F\colon\R^\mcE\to\R$ which depends on only finitely many components.
As before, we fix a normal, complete, subcritical rule $R$ and recall the corresponding spaces $\mcT^\ex,\mcT_+^\ex,\mcF_-^\ex$ and groups $G_+^\ex$, $G_-^\ex$.

Throughout this section, as at the end of the previous one, we work on the extended regularity structure (see, however, Remark~\ref{rem:Upsilon_reduced_structure} about the reduced structure).

To state the main result of this section,
we recall several algebraic aspects of the fixed point maps considered in regularity structures.
A \textit{jet} is a formal sum of elements of the regularity structure
\begin{equ}\label{eq:jet}
U = U^P + U^R\;,
\qquad U^P = \sum_{k \in\N^d} u_k\frac{X^k}{k!}\;,
\qquad U^R = \sum_{\tau\in\mfT_*^\ex} u_\tau\frac{\tau}{\tau!}\;,
\end{equ}
where $\mfT_*^\ex \eqdef \mfT^\ex\setminus\{X^k\,:\,k\in\N^d\}$
is the set of non-polynomial trees.\label{mfT_star_pageref}
In~\eqref{eq:jet}, $u_k$ and $u_\tau$ are real numbers
and the series $U^P$ and $U^R$ are called the \textit{polynomial}
and \textit{non-polynomial part} of $U$ respectively.
The combinatorial factor $\tau!$ is defined as follows.
Every $\tau\in\mfT^\ex$ can be written uniquely in the form
\begin{equ}
\tau = X^k \bone^a \prod_{i=1}^I \Xi_{l_i}^{\alpha_i}\prod_{j=1}^J \mcI_{m_j}[\tau_j]^{\beta_j}\;,
\end{equ}
where $I,J\geq 0$, and $l_1,\ldots, l_I\in \N^d$ are \textit{distinct},
$(m_1,\tau_1),\ldots,(m_J,\tau_J)\in \N^d\times\mfT^\ex$
are \textit{distinct}, and $\alpha_i,\beta_j \geq 1$.
We then define $\tau!$ inductively by
\begin{equ}\label{eq:tau_factorial}
\tau! = k! \prod_{i=1}^I \alpha_i! \prod_{j=1}^J (\tau_j!)^{\beta_j}\beta_j!\;.
\end{equ}

A central object in regularity structures is a \textit{modelled distribution},\footnote{The
space of modelled distributions is always defined with respect a fixed underlying model $(\hat\Pi,\hat\Gamma)$ and involves important analytic considerations and an upper bound on the number of coefficients $u_k,u_\tau$ that are allowed to be non-zero, see~\cite[Def.~3.1]{Hairer14}.
As before, we do not discuss these analytic aspects here.} which is a space-time dependent jet $U$, i.e. we allow $u_k, u_\tau\colon\R^d\to \R$ in~\eqref{eq:jet}.
Consider a non-linearity $F$ that obeys $R$ (see Definition~\ref{def:obey}).
An important feature is that we can rewrite the SPDE~\eqref{eq:SPDE} as an equation for $U$ of the form
\begin{equ}\label{eq:fixed_point}
U = \mcP(\F(U)) + Gu_0\;.
\end{equ}
Here, $Gu_0$ is the lift of the harmonic extension of the initial condition $u_0$ to the space of polynomials $Gu_0 = \sum_{k\in\N^d} a_k X^k/k!$.
The operator $\mcP$ encodes convolution with $G$, the Green's function of $\partial_t -\mcL$, and $\F$ is a lift of $F$ to the space of jets.
Before presenting the precise form of $\mcP$ and $\F$, we explain how $U$ relates to a classical PDE.

The way one recovers a distribution $u$ on $\R^d$ from a modelled distribution $U$ is through a reconstruction operator $\mcR\colon U\to u$, which is continuous in both $U$ and the underlying model $(\hat\Pi,\hat\Gamma)$.
When the underlying model is smooth (as is always the case for us)
we simply have $\CR U(x)= [\hat\Pi_x U(x)](x)$.

A non-trivial fact is that, if the underlying model is of the form
$(\hat\Pi,\hat\Gamma)=\mcZ(\bfPi^g)$, where $g\in G_-^\ex$ and $\bfPi$ is the canonical lift of a smooth function $\xi$, and if $U$ solves~\eqref{eq:fixed_point}, then $\CR U$ solves a classical PDE.
This PDE will be of a form similar to~\eqref{eq:SPDE} but with additional `counterterms' on the right-hand side that depend explicitly on $g$ and $F$.

To understand this, we note that the operator $\CP$ satisfies $\CR\CP U = G*\CR U$.
To find the equation solved by $u\eqdef \CR U$, we therefore write
\begin{equ}
u = G * (\CR \F(U)) + G u_0 = G*[(\hat\Pi_\cdot \F(U)(\cdot))(\cdot)] + Gu_0\;.
\end{equ}
The goal is now to express the term inside the parentheses $[\ldots]$ in terms of $u$.
Denoting by $\Pi$ the recentred map associated to $\bfPi$, recall that
\begin{equ}\label{eq:nice_M_g}
\hat\Pi_x  = \Pi_x M^g \;,
\end{equ}
which is where we use that we are working on the extended structure.
Observe that $M^g \F(U)$ will in general contain trees with extended decorations (in contrast to $U$ which in practice does not -- see Remark~\ref{rem:U_reduced}).

It turns out, and this is the content of Theorem~\ref{thm:MU_coherent} below,
that there exists a new \textit{family} of non-linearities $\{F^g_a\}_{a\leq 0}$, with $F^g_0 = F$, of which only finitely many are non-zero,
and which satisfy the `covariance' property with $M_g$
\begin{equ}\label{eq:covariance}
M_g \F(U) = \sum_{a\leq 0}\F^g_a(M_g U) \bone^a\;.
\end{equ}
(The reason $a\leq 0$ appears here instead of the more natural index set $\mfT_-$ is that $M_g$ discards all information of the negative extracted subtrees except for their degrees.)
Combining~\eqref{eq:nice_M_g} and~\eqref{eq:covariance}, we obtain
\begin{equ}
\hat\Pi_x \F(U) = \Pi_x \sum_{a\leq 0} \F_a^g(M^g U)\;.
\end{equ}
The final ingredient, which is easy to show from the fact that $\Pi_x$ is multiplicative, is that
\begin{equ}
\Pi_x \F^g_a(U) = F^g_a(u,\nabla u,\ldots, \xi,\nabla \xi,\ldots)
\end{equ}
(see~\cite[Lem.~5.3]{BCCH21}).
In conclusion, we obtain
\begin{equ}
u = G*\Big(\sum_{a\leq 0} F_a^g(u,\nabla u,\ldots, \xi,\nabla \xi,\ldots)\Big) + G u_0\;,
\end{equ}
which, recalling that $F^g_0=F$,
is the integral form of the PDE
\begin{equ}\label{eq:renorm_PDE}
(\partial_t -\mcL )u = F(u,\nabla u,\ldots, \xi,\nabla \xi,\ldots) + \sum_{a<0} F^g_a(u,\nabla u,\ldots, \xi,\nabla \xi,\ldots)\;.
\end{equ}
The terms in the final sum over $a<0$ are the aforementioned counterterms in the renormalised PDE.

With this motivation in mind, we now describe in more detail the terms in~\eqref{eq:fixed_point}.
While the precise definition of the `integration' operator $\mcP$ takes some time to write down, see~\cite[Sections 5, 6.6]{Hairer14},
it has the very simple \textit{algebraic} form
\[
\mcP U = \sum_{k\in \N^d} v_k\frac{X^k}{k!}  + \mcI[U]\;.
\]
Here, $v_k$ are functions that are derived from analytic formulae and that depend non-locally on $U$.
The operator $\mcI$ on the other hand maps every tree $\tau\in\CT^\ex$ to $\mcI_0[\tau]$
and extends to an operator on jets.
The coefficient appearing in front of $\mcI[\tau]$ for every $\tau\in\mfT_\circ^\ex$
is thus $u_\tau/\tau!$,
and is thus entirely local in $U$.
Therefore, if $U$ satisfies~\eqref{eq:fixed_point},
then it also satisfies the purely \textit{algebraic} fixed point problem
\begin{equ}\label{eq:fixed_point_single_F}
U = U^P +
\mcI[\F(U)]
\end{equ}
(but generally not vice versa).
To define the function $\F$ from a non-linearity $F$,
we identify the polynomial part $U^P$
with the function $(U^P,0)\in\R^{\mcE}$ which maps $(\Xi,k)\mapsto 0$ and $(\mcI,k)\mapsto u_k$.
The function $\F$ is then the `lift' of $F$ to jets, given through a Taylor expansion
\begin{equ}\label{eq:F_def}
\F(U) = \sum_{\alpha\in \N^\mcE} \frac{D^\alpha F(U^P,0)}{\alpha!}(U-u,\Xi)^\alpha\;.
\end{equ}
Above we use the suggestive notation $(U-u,\Xi)^{(\Xi,k)}\eqdef\Xi_k$,
and $(U-u,\Xi)^{(\mcI,k)}\eqdef\mcD^k U-u_k\bone$,
where $\mcD^k$ is the natural lift of the derivative operator which, for a jet $U$
of the form
\begin{equ}\label{eq:planted_nonpoly}
U = U^P + U^R\;,
\qquad U^P = \sum_{k \in\N^d} u_k\frac{X^k}{k!}\;,
\qquad
U^R = \sum_{\tau\in \mfT^\ex} u_{\mcI [\tau]} \frac{\mcI[\tau]}{\tau!}\;,
\end{equ}
is given by
\begin{equ}
\mcD^k U = \sum_{m \geq k} u_{m} \frac{X^{m-k}}{(m-k)!} + \sum_{\tau\in\mfT^\ex} u_{\mcI[\tau]} \frac{\mcI_k[\tau]}{\tau!}\;.
\end{equ}
Then $(U-u,\Xi)^{\alpha}$ for general $\alpha\in\N^\mcE$ is defined by multiplicativity;
in particular, $(U-u,\Xi)^0=\bone$.
Remark that $\F(U)$ is well-defined as a jet.

Due to the appearance of a family of non-linearities $\{F_a\}_{a\leq 0}$  in~\eqref{eq:covariance},
it is natural to make the following definition.
\begin{definition}\label{def:ext_non-linear}
An \emph{extended non-linearity} is a family of non-linearities $\cbF=\{F_a\}_{a\leq 0}$, indexed by the non-positive real numbers,
of which at most finitely many are non-zero.
We say that $\cbF$ is a \textit{reduced non-linearity} if $F_a=0$ for all $a<0$.
\end{definition}
Note that every non-linearity $F$ gives rise to a reduced non-linearity $\cbF=\{F_a\}_{a\leq 0}$ by setting $F_0=F$ and $F_a=0$ for $a<0$,
and that this correspondence gives a bijection between non-linearities and reduced non-linearities.
\begin{remark}\label{rem:BCCH21_diff_2}
In~\cite{BCCH21}, `drivers' (instead of number $a\leq 0$) are used to index
non-linearities,
and the non-linearities therein do not depend on the noise component $\{\Xi_l\}_{l\in\N^d}$ in $\CE$ (see Remark~\ref{rem:BCCH21_diff_1}),
i.e. they are functions $\R^{\N^d} \to \R$ in the present context.
Each `driver' is then a pair $(a,L)\in (-\infty,0]\times \mfN$, where the component $L\in\mfN$ encodes a non-linearity of the form $F(u,\nabla u,\ldots)\prod_{l\in L}D^l\xi$.
The differences between our setting and that in~\cite{BCCH21} are mostly notational,
but our setting allows for several cleaner statements, e.g. that of Theorem~\ref{thm:MU_coherent}.
\end{remark}
For an extended non-linearity $\cbF= \{F_a\}_{a\leq 0}$ and a jet $U$, we denote
\begin{equ}\label{eq:def_cbF(U)}
\cbF(U) = \sum_{a\leq 0} \F_a(U)\bone^a\;.
\end{equ}
It is now natural to consider the algebraic fixed point problem
\begin{equ}\label{eq:fixed_point_simple}
U = U^P +
\mcI[\cbF(U)]\;,
\end{equ}
which generalises~\eqref{eq:fixed_point_single_F}.
Remark that, if $U$ satisfies~\eqref{eq:fixed_point_simple},
then $U$ necessarily takes the form~\eqref{eq:planted_nonpoly} and, in particular, the non-polynomial part of $U$ consists entirely of planted trees.
Remark also that $\F_a(U)$ is another jet but which is not in general of the form~\eqref{eq:planted_nonpoly}.

%\begin{remark}
%If $F$ obeys $R$, then all the trees appearing in $\F(U)$ are in $\mfT_\circ^\ex$, see~\cite[Prop.~3.13]{BCCH21}.
%\end{remark}

\begin{remark}\label{rem:triangular}
Suppose $\tau\in\mfT_\circ^\ex$ with $\fancynorm{\tau}=n$, i.e. $\tau$ has $n$ kernel edges.
It follows readily from the definition
that the coefficient of $\tau$ in $\F_a(U)$
is determined by $U^P$ and the coefficients $u_\sigma$ with $\fancynorm{\sigma} \leq n$.
\end{remark}
While the definition of $\CP$ is non-local,
the identity~\eqref{eq:fixed_point_simple}
\textit{is} local since $\mcI$ and $\F_a$ operate on jets (vs. modelled distributions).
Therefore, since we are interested in the algebraic aspects,
we will henceforth consider only jets rather than
modelled distributions.
Furthermore, the triangular structure of the fixed point problem~\eqref{eq:fixed_point_simple} (Remark~\ref{rem:triangular}) implies that the entire jet can be recovered from the polynomial part.

To summarise the above discussion, we state the following proposition.
\begin{proposition}\label{prop:polys_determine}
Suppose $\cbF=\{F_a\}_{a\leq 0}$ is an extended non-linearity
and let $U^P = \sum_{k\in\N^d} u_k X^k/k!$ be a jet taking values in the polynomials.
Then there exists a unique jet $U$ with polynomial part $U^P$ and which satisfies~\eqref{eq:fixed_point_simple}.
\end{proposition}

\begin{proof}
We will show that $U^R = \sum_{\tau\in\mfT^\ex}u_{\mcI[\tau]}\tau/\tau!$ can be defined recursively from $U^P$ and $\cbF$ to satisfy $U^R = \mcI[\cbF(U)]$.
Consider first any $\tau\in\mfT^\ex$ with $\fancynorm{\tau}=0$ (so only noise edges).
Then, for any $a\leq 0$, obviously the coefficient of $\tau$ in $\F_a(U)\bone^a$ is determined by $U^P$ and $F_a$.
By~\eqref{eq:planted_nonpoly}-\eqref{eq:fixed_point_simple}, we see that $u_\sigma$ with $\fancynorm{\sigma}\leq 1$ 
is determined by $U^P$ and $\cbF$.
Continuing, the coefficients in $\F_a(U)$ of every tree $\tau$ with $\fancynorm{\tau}\leq 1$ is determined by $U^P$ and $\cbF$;
after applying $\mcI$, we see that $u_\sigma$ with $\fancynorm{\sigma}\leq 2$
is determined by $U^P$ and $\cbF$.
The proof follows from induction.
\end{proof}
%
%\begin{remark}
%At this stage, we have not made any claims about whether we can solve the fixed point problem~\eqref{eq:fixed_point_simple}.
%We will do this soon in Lemma~\ref{...}.
%\end{remark}
%
We now give a completely explicit description of $u_\tau$ in terms of $U^P$ for every $U$ satisfying~\eqref{eq:fixed_point_simple}.
\begin{definition}\label{def:Upsilon}
Given an extended non-linearity $\cbF=\{F_a\}_{a\leq 0}$ and a tree $\tau\in\mfT^\ex$ of the form~\eqref{eq:general_tree_ext}, we define a new non-linearity $\Upsilon^\cbF[\tau]$ inductively by
\begin{equ}\label{eq:Upsilon_def}
\Upsilon^\cbF [\tau] = \Big(\prod_{j\in J}
\Upsilon^\cbF[\tau_j]\Big)
\cdot
\Big(
\partial^k \prod_{j\in J} D_{(\mcI,m_j)} \prod_{l\in L} D_{(\Xi,l)}
\Big) F_a\;,
\end{equ}
where $\partial^i F$ for $i \in [d]$ is defined by
\begin{equ}\label{eq:partial_def}
\partial^i F = \sum_{p\in\N^d} \mcY_{(\mcI,p+e_i)} D_{(\mcI,p)} F\;,
\end{equ}
where $\mcY_o\colon \CE\to \R$ for $o\in\mcE$ is the non-linearity given by evaluation of the $o$-component,\label{mcY_page_ref}
and $\partial^k$ for $k\in\N^d$ is given by composition.
We extend $\Upsilon^\cbF$ to a linear map on $\CV^\ex\eqdef \Span_\R(\mfT^\ex)$.
We call $\Upsilon^\cbF$ 
the \emph{coherence map} of $\cbF$.
\end{definition}
Note that $\partial^i,\partial^j$ commute, hence $\partial^k$ is well-defined for $k\in\N^d$.
On the other hand, $\partial^i$ and $D_o$ for $o\in\mcE$ do \textit{not} commute.
Remark also that $\Upsilon^\cbF[\tau]=0$ whenever $\tau$ has a strictly positive extended decoration -- this is due to our choice in Definition~\ref{def:ext_non-linear} to index extended non-linearities by non-positive numbers.
\begin{definition}\label{def:coherent}
A jet $U$
is called \textit{coherent} with an extended non-linearity $\cbF$
if it is of the form~\eqref{eq:planted_nonpoly}
and for all $\tau\in\mfT^\ex$
\begin{equ}\label{eq:coherent_def}
u_{\mcI[\tau]} = \Upsilon^{\cbF}[\tau](U^P,0)\;.
\end{equ}
\end{definition}
\begin{remark}\label{rem:U_reduced}
If $F=\cbF$ is a reduced non-linearity and $U$ is coherent with $F$, then the expansion of $U$ in~\eqref{eq:planted_nonpoly} contains no trees with non-zero extended decorations.
\end{remark}
The motivation behind Definitions~\ref{def:Upsilon} and~\ref{def:coherent} is the following explicit version of Proposition~\ref{prop:polys_determine},
the proof of which is outlined in Section~\ref{subsec:coherent_alt_def}.
\begin{lemma}\label{lem:coherent_alt_def}
Let $U$ be a jet and $\cbF$ an extended non-linearity.
Then the following two statement are equivalent.
\begin{itemize}
\item $U$ satisfies~\eqref{eq:fixed_point_simple}.
\item $U$ is coherent with $\cbF$.
\end{itemize}
\end{lemma}
%
%Note that, up to now, we have not used the rule $R$ or any conditions relating $F$ to $R$.
%Our next goal is to 
%We see that the unique $U$ which satisfies~\eqref{eq:fixed_point_simple} takes values in $\mfT^\ex_\circ$
%whenever $\Upsilon^{\cbF}[\tau]=0$ for all $\tau$.
%
We next show how $\Upsilon$ allows us to define the non-linearities $\{F^g_a\}_{a\leq 0}$ appearing in~\eqref{eq:covariance}.
Recall that every $g\in G_-^\ex$ defines a linear map $M_g\colon\CT^\ex \to\CT^\ex$, $M_g = (g\otimes \id)\Deltam_\ex$.
Moreover, equipping $G_-^\ex$ with the product
$gh \eqdef (h\otimes g)\Deltam_\ex$, we obtain a group representation
$G_-\ni g\mapsto M_g\in \GL(\CT^\ex)$.

We further define an adjoint map $M_g^*\in\GL(\CT^\ex)$ given by
\begin{equ}\label{eq:adjoint_M_g}
\scal{M_g^*\sigma,\tau} = \scal{\sigma,M_g\tau}
\end{equ}
where $\scal{\cdot,\cdot}$ is the inner product on $\CV^\ex$ for which the elements $\tau\in\mfT^\ex$ are orthogonal with $\scal{\tau,\tau}=\tau!$.
In particular,
\begin{equ}
M_g^\ast \sigma = \sum_{\tau\in\mfT_\circ^\ex} \frac{\scal{\sigma,M_g\tau}}{\tau!}\tau\;.
\end{equ}
We naturally extend $M_g^*\colon \CV^\ex\to\CT^\ex$ by $M_g^*(\sigma)=0$ for all $\sigma\in\mfT^\ex\setminus\mfT^\ex_\circ$.
(This extension becomes natural once we view $\CT^\ex\subset\CV^\ex$ and
$M_g$ as a map $\CT^\ex\to \CV^\ex$. Then $M_g^*$ is the unique linear map $\CV^\ex\to \CT^\ex$ satisfying~\eqref{eq:adjoint_M_g}.)

\begin{remark}
Since $M_g$ maps $\tau$ to a linear sum $\sum_i c_i \tau_i$ where $|\tau_i|_\s\geq |\tau|_\s$,
it follows that $M_g^*\sigma$ is well-defined for every $\sigma\in\mfT^\ex_\circ$ and 
takes the form $M_h^*\sigma=\sum_j b_j \sigma_j \in\mcT^\ex$ where $|\sigma_j|_\s\leq|\sigma|_\s$ and the sum is finite.
\end{remark}
\begin{remark}\label{rem:coh_inner_prod}
Using the inner product $\scal{\cdot,\cdot}$ on $\CT^\ex$, we can rewrite the condition~\eqref{eq:coherent_def}
as
\begin{equ}
\scal{U, \mcI[\tau]} = \Upsilon^\cbF[\tau](U^P,0)\;,
\quad \forall \tau\in\mfT^\ex_\circ\;.
\end{equ}
\end{remark}
For an extended non-linearity $\cbF=\{F_a\}_{a\leq 0}$ and a character $g\in G_-^\ex$, we now define
another extended non-linearity $\cbF^g = \{F^g_a\}_{a\leq 0}$ by
\begin{equ}\label{eq:F^g_def}
F^g_a \eqdef \Upsilon^\cbF[M_g^*\bone^a] = \sum_{\tau\in\mfT_\circ}\frac{\scal{M_g\tau,\bone^a}}{\tau!}\Upsilon^\cbF[\tau]\;.
\end{equ}
Note that $M_g^*\bone^0 = \bone^0$ and therefore $F^g_0 = F_0$.

\begin{definition}\label{def:F_orbit}
For a non-linearity $F$, we define the \textit{orbit} of $F$ as
the set of extended non-linearities
\begin{equ}
F^{G^\ex_-} \eqdef \{F^g \,:\, g\in G^\ex_-\}\;,
\end{equ}
where we treat $F$ as a reduced non-linearity on the right-hand side.
\end{definition}
Remark that $F_0=F$ for all $\cbF=\{F_a\}_{a\leq 0} \in F^{G^\ex_-}$.
The following lemma gives an explicit formula for the coherence map of every $\cbF\in F^{G^\ex_-}$.
Note that this lemma is the first time where we use that $F$ obeys $R$ (see Definition~\ref{def:obey}).
%
%\begin{remark}\label{rem:Upsilon_vanish_trees}
%It is obvious from the definition of $\mfT_\circ^\ex\subset\mfT^\ex$ and from the definition of non-linearities obeying $R$ (Definitions~\ref{def:conform} and~\ref{def:obey}) that if $\tau\in\mfT^\ex\setminus\mfT_\circ^\ex$,
%then $\Upsilon^\cbF[\tau]=0$ for all $\cbF\in F^{G^\ex_-}$.
%\end{remark}
%
\begin{lemma}\label{lem:Upsilon_MF}
Suppose $F$ obeys $R$.
Then for $g\in G_-^\ex$, $\cbF\in F^{G_-^\ex}$, and $\tau\in\mfT_\circ^\ex$,
\begin{equ}
\Upsilon^{\cbF^g}[\tau] = \Upsilon^\cbF[M^*_g\tau]\;.
\end{equ}
\end{lemma}

We will give the proof of Lemma~\ref{lem:Upsilon_MF} in Section~\ref{subsec:Upsilon_MF}.
We can now state and prove the main algebraic result of~\cite{BCCH21}.

\begin{theorem}\label{thm:MU_coherent}[Theorem~3.25 of~\cite{BCCH21}]
Suppose that $F$ obeys $R$, that $\cbF\in F^{G_-^\ex}$, and that $U$ solves the fixed point problem~\eqref{eq:fixed_point_simple}.
% satisfies~\eqref{eq:fixed_point_single_F}.
Then, for all $g\in G_-^\ex$,
$M_gU$ is coherent with $\cbF^g$
and
\begin{equ}\label{eq:MU_coherent}
M_g \cbF(U) = \cbF^g (M_g U)\;.
\end{equ}
In particular, $V\eqdef M_g U$ solves the fixed point problem
\begin{equ}
V = V^P + \mcI[\cbF^g(V)]\;.
\end{equ}
\end{theorem}

\begin{proof}
Since $M_g$ acts trivially on the polynomials, we have $M_gU^P = U^P$.
Furthermore $U$ is of the form~\eqref{eq:planted_nonpoly} and \begin{equ}\label{eq:M_I_commute}
M_g\mcI[\tau] = \mcI[M_g\tau]\;.
\end{equ}
Therefore, by Lemma~\ref{lem:coherent_alt_def},~\eqref{eq:MU_coherent} follows from the coherence of $M_gU$ with $\cbF^g$, which we now prove.
Since $M_g U$ is also of the form~\eqref{eq:planted_nonpoly} due to~\eqref{eq:M_I_commute},
it suffices by Remark~\ref{rem:coh_inner_prod} to show
\begin{equ}\label{eq:MU_coherent2}
\scal{M_g U,\mcI[\tau]} = \Upsilon^{\cbF^g}[\tau](M_g U^P,0)\;.
\end{equ}
To this end, observe that
\begin{equs}
\scal{M_g U,\mcI[\tau]}
&= \scal{U,M_g^*\mcI[\tau]}
=
\scal{U,\mcI[M_g^*\tau]}
\\
&=
\Upsilon^{\cbF}[M_g^*\tau](U^P,0)
=
\Upsilon^{\cbF^g}[\tau](U^P,0)\;.
\end{equs}
In the third equality above we used that $U$ is coherent with $\cbF$ by Lemma~\ref{lem:coherent_alt_def} together with Remark~\ref{rem:coh_inner_prod},
and in the fourth equality we used Lemma~\ref{lem:Upsilon_MF}.
Since $U^P = M_gU^P$,
we obtain~\eqref{eq:MU_coherent2}.
\end{proof}
As we see above, Theorem~\ref{thm:MU_coherent} is a consequence of Lemmas~\ref{lem:coherent_alt_def} and~\ref{lem:Upsilon_MF}.
These two results are stated and proved as~\cite[Lem.~4.6]{BCCH21}
and~\cite[Lem.~3.22 \&~3.23]{BCCH21}.
In the remainder of these lecture notes, we
sketch the proofs of these results.

\begin{remark}\label{rem:Upsilon_reduced_structure}
A similar and simpler result holds on the reduced structure.
That is, for each $g\in G_-$, we can define a single non-linearity $M_g F = \Upsilon^F[M_g^*\bone]$, where we now treat $M_g\colon\CT\to\CT$ as a map on the reduced structure.
In this case one obtains an action of $G_-$ on the space of non-linearities $F$ together with a simpler covariance property
\begin{equ}
M_g \F(U) = \F^g (M_gU)\;.
\end{equ}
Unfortunately the argument outlined
at the start of this section to derive the renormalised PDE~\eqref{eq:renorm_PDE} breaks down if working on the reduced structure (specifically,~\eqref{eq:nice_M_g} does not hold).
\end{remark}

\subsection{An enlarged space of trees}

The main idea behind both proofs of Lemmas~\ref{lem:coherent_alt_def} and~\ref{lem:Upsilon_MF} is to enlarge the space of trees in order to keep track of
which derivative of the solution every polynomial term comes from (in the case of a system of SPDEs where the solution has multiple components,
this would also keep track of polynomial terms coming from different components).

The new set of trees will be denoted $\mfB$ and we define $\mfB$ first using symbols.
For $n\geq0$, consider the set of trees $\mfB^n$ defined in symbols by $\mfB^0=\emptyset$ and for $n\geq 1$ by\label{CJ_page_ref}
\begin{multline*}
\mfB^n = \Big\{ \bone^a
\prod_{k\in K} \CJ_{s_k}[X^{t_k}]
\Xi_L
\prod_{j \in J} \mcI_{m_j}[\sigma_j] \,:\,
a\in \R\,,\,
\# J,\#K<\infty\,,\,
L \in\mfN\,,
\\
m_j,s_k,t_k \in \N^d\,,\,
s_k\leq t_k\,,\, s_k\neq t_k\,,\,
\sigma_j \in \mfB^{n-1}\Big\} \cup \mfB^{n-1}\;.
\end{multline*}
We then set\label{mfB_pag_ref}
\begin{equ}
\mfB =
%\Big[
\bigcup_{n\geq 0}\mfB^n
%\Big]
%\sqcup \{X^k\,:\,k \neq 0\}
\end{equ}
and let $\CB=\Span_\R(\mfB)$.\label{CB_pag_ref}
Every $\sigma\in\mfB$ is of the form
\begin{equ}\label{eq:tree_in_mfB}
\sigma = Y\prod_{j\in J} \mcI_{m_j}[\sigma_j]
\eqdef
\bone^a\prod_{k\in K} \CJ_{s_k}[X^{t_k}]
\Xi_L \prod_{j\in J} \mcI_{m_j}[\sigma_j] 
\end{equ}
where $J,K$ are finite index sets, $a\in \R$, $t_k\geq s_k$, $t_k\neq s_k$, $L\in\mfN$, and $\sigma_j \in \mfB$. This representation is unique up to reordering factors
and induces an associated and commutative product $\mfB\times\mfB\to\mfB$ exactly as in~\eqref{eq:mfT^ex_prod}.

The interpretation of $\sigma\in\mfB$ as a tree follows the same logic as the explanation around~\eqref{eq:general_tree}.
The term $\prod_{k\in K} \CJ_{s_k}[X^{t_k}]$ is treated as a decoration at each node (although it is helpful to think of $\CJ_{s_k}$ as an edge of the same nature as $\mcI_{s_k}$).
The only difference between trees in $\mfB$ and those in $\mfT^\ex$
is that the polynomial decorations $\prod_{k\in K} \CJ_{s_k}[X^{t_k}]$ come with extra information.
We further introduce a surjective `projection' 
map\label{Q_page_ref}
\begin{equ}
Q\colon \mfB\to \mfT^\ex\;,
\end{equ}
which forgets this extra information and is defined, for $\sigma\in\mfB$ of the form~\eqref{eq:tree_in_mfB}, recursively by
% $Q(X^k)=0$ for $k\neq 0$ and, for $\sigma\in\mfB$ of the form~\eqref{eq:tree_in_mfB_polys}, by
\begin{equ}
Q(\sigma) = \bone^a\prod_{k\in K} X^{t_k-s_k}
\Xi_L
\prod_{j\in J} \mcI_{m_j}[Q \sigma_j]\;.
\end{equ}
We extend $Q$ to a linear map $Q\colon\CB\to\CV^\ex$.

We let $\bone\in\mfB$ denote the tree with a single node and zero decorations, i.e. the case $L=\emptyset$, $a=0$, and $\#K=\#J=0$ in~\eqref{eq:tree_in_mfB}, which we note is the unit for the product on $\mfB$.\label{bone_CB_page_ref}

Finally, for $\sigma\in\mfB$ of the form~\eqref{eq:tree_in_mfB} and extended non-linearity $\cbF=\{F_a\}_{a\leq 0}$, we define a new non-linearity $\mathring\Upsilon^{\cbF}[\sigma]$ inductively by
\begin{equ}\label{eq:mathring_Upsilon_def}
\mathring\Upsilon^{\cbF} [\sigma] =
\prod_{k\in K} \CY_{(\mcI,t_k)}
\prod_{j\in J}
\mathring\Upsilon^{\cbF}[\sigma_j]
\cdot
\Big(
\prod_{k\in K} D_{(\mcI,s_k)}
\prod_{j\in J} D_{(\mcI,m_j)} \prod_{l \in L} D_{(\Xi,l)}
\Big) F_a\;.
\end{equ}
While operating on a seemingly more complicated space,
it turns out that the lifted coherence map $\mathring\Upsilon^\cbF$ on the set of trees $\mfB$ is significantly simpler to work with.

First, we give a relationship between $\mathring\Upsilon$ and $\Upsilon$.

\begin{lemma}\label{lem:Upsilon_Q}
For any $\tau\in\mfT^\ex$ and extended non-linearity $\cbF$,
\begin{equ}
\Upsilon^{\cbF}[\tau] = \mathring \Upsilon^{\cbF}[Q^*\tau]\;.
\end{equ}
\end{lemma}
The map $Q^*$ above is the adjoint of $Q$ under the inner product $\scal{\cdot,\cdot}$ on $\CV^\ex$ and the inner product $\scal{\cdot,\cdot}_{\CB}$ on $\CB$ for which the elements $\sigma\in \mfB$ are orthogonal with $\scal{\sigma,\sigma}_{\CB}=\sigma!$.
Here $\sigma!$ is defined analogously to~\eqref{eq:tau_factorial}: let us write $\sigma\in\mfB$ uniquely as
\begin{equ}
\sigma = \bone^a \prod_{k=1}^K\CJ_{s_k}[X^{t_k}]^{\gamma_k}
\prod_{i=1}^I \Xi_{l_i}^{\alpha_i}
\prod_{j=1}^J \mcI_{m_j}[\sigma_j]^{\beta_j}
\end{equ}
where $I,J,K\geq 0$, and $(s_1,t_1),\ldots, (s_K,t_K)\in \N^d\times\N^d$ are distinct, $l_1,\ldots,l_I\in \N^d$ are distinct, $(m_1,\sigma_1),\ldots,(m_J,\sigma_J)\in \N^d\times\mfB$
are distinct, and $\alpha_i,\beta_j,\gamma_k \geq 1$.
We then define inductively
\begin{equ}
\sigma! = \prod_{k=1}^K (t_k-s_k)!^{\gamma_k}\gamma_k! \prod_{i=1}^I \alpha_i!\prod_{j=1}^J (\sigma_j!)^{\beta_j}\beta_j!\;.
\end{equ}
Lemma~\ref{lem:Upsilon_Q} demonstrates that $\mathring \Upsilon^\cbF$ can be seen as a decomposition of $\Upsilon^\cbF$.
The proof of this lemma, which proceeds by induction, rests on a multi-variable Fa\`a Di Bruno formula (see~\cite[Appendix~A.2]{BCCH21}).

\subsection{Proof of Lemma~\ref{lem:coherent_alt_def}}
\label{subsec:coherent_alt_def}

We can now give a brief sketch the proof of Lemma~\ref{lem:coherent_alt_def}.
Suppose first that $U$ is coherent with $\cbF=\{F_a\}_{a\leq 0}$; we aim to prove~\eqref{eq:fixed_point_simple}.
By Lemma~\ref{lem:Upsilon_Q},
\begin{equ}
\CD^s U = u_s\bone+
\sum_{\substack{t\geq s\\t\neq s}}u_t \frac{X^{t-s}}{(t-s)!}
+
Q\sum_{\sigma\in\mfB}\mathring \Upsilon^\cbF[\sigma] \frac{\mcI_s[\sigma]}{\sigma!}\;.
\end{equ}
The above expression suggests we should look at a new `jet' $\CD^s\hat U$ which is a formal series of trees in $\mfB$ and is defined by
\begin{equ}
\CD^s \hat U = u_s\bone +
\sum_{\substack{t \geq s \\ t\neq s}}u_t\frac{\CJ_s [X^t]}{(t-s)!}
+
\sum_{\sigma\in\mfB}\mathring \Upsilon^F[\sigma] \frac{\mcI_s[\sigma]}{\sigma!}\;.
\end{equ}
We should think of this new jet as living above $\CD^s U$ since $Q\CD^s\hat U = \CD^s U$.
What is now easy to see is that $\hat U\eqdef \CD^0 \hat U$ solves the fixed point problem
\begin{equ}\label{eq:fixed_point_lifted}
\hat U = u_0\bone + \sum_{\substack{t\in\N^d\\ t\neq 0}} u_t \frac{\CJ_0[X^t]}{t!} + \mcI[\cbF(\hat U)]\;,
\end{equ}
where $\cbF(\hat U)$ is defined analogously to~\eqref{eq:F_def}-\eqref{eq:def_cbF(U)}.
This identity is not difficult to show
because, when we consider $\scal{\F_a(\hat U),\sigma}$ for $\sigma$ of the form~\eqref{eq:tree_in_mfB},
the only terms in the expansion~\eqref{eq:F_def} of $\F_a(\hat U)$ which do not vanish are those for the \textit{single} multi-index $\alpha\in\N^\CE$
given by
\begin{equ}
\alpha(\mcI,q) = \sum_{k\in K}\delta_{s_k,q} + \sum_{j\in J} \delta_{m_j,q}
\end{equ}
and
\begin{equ}
\alpha(\Xi,q) = \sum_{l\in L} \delta_{l,q}\;,
\end{equ}
where $\delta$ is the Kronecker delta.
That is, we simply count the instances of $\mcI_q$, $\CJ_q$, and $\Xi_q$ appearing at the root of $\sigma$ for each $q\in\N^d$ (here $\mcI_q$ and $\CJ_q$ combine towards the same total).
We see here the point of the extra decorations $\CJ_{s_k}[X^{t_k}]$:
an analogous statement for $\F_a(U)$ is more complicated
because polynomials of the different derivatives $\CD^s U$ are not distinguished
and different multi-indexes can contribute to the coefficient of one tree.
With this remark in hand, it is simple to verify~\eqref{eq:fixed_point_lifted}.
To pass from this `lifted' fixed point problem to the original one~\eqref{eq:fixed_point_simple}, one uses Lemma~\ref{lem:Upsilon_Q}
together with (somewhat tedious) manipulations of multi-indexes (see~\cite[Lem.~4.6]{BCCH21}).

The reverse implication, that~\eqref{eq:fixed_point_simple} implies that $U$ is coherent with $\cbF$, follows from the following exercise.
\begin{exercise}
Suppose $U$ solves~\eqref{eq:fixed_point_simple} and write $U$ as in~\eqref{eq:planted_nonpoly}.
Let $V$ be the unique jet such that $V^P=U^P$ and $V$ is coherent with $\cbF$.
Prove that $u_{\mcI[\tau]} = v_{\mcI[\tau]}$ for all $\tau\in\mfT^\ex$.
[Hint: proceed by induction on $\fancynorm{\tau}$
and use that $V$ also solves~\eqref{eq:fixed_point_simple} due to the first implication.]
\end{exercise}
This completes the sketch proof of Lemma~\ref{lem:coherent_alt_def}.

\subsection{Proof of Lemma~\ref{lem:Upsilon_MF}}
\label{subsec:Upsilon_MF}

We now turn to the proof of Lemma~\ref{lem:Upsilon_MF}.
It turns out that the proof of this lemma rests on a pre-Lie morphism property of the adjoint renormalisation map $M_g^*$.

Recall that a pre-Lie algebra is a vector space $V$ together with a pre-Lie product $\triangleleft$, i.e. a bilinear map $\triangleleft \colon V\times V\to V$
which satisfies the pre-Lie identity
\begin{equ}\label{eq:pre-Lie}
x\triangleleft (y\triangleleft z) - (x \triangleleft y)\triangleleft z
= y\triangleleft (x\triangleleft z) - (y\triangleleft x) \triangleleft z\;.
\end{equ}
Notice that if $\triangleleft$ is associative, then the pre-Lie identity is trivially satisfied with both the left- and right-hand sides of~\eqref{eq:pre-Lie} equal to $0$.
Pre-Lie algebras are therefore generalisations of associative algebras where the condition $x\triangleleft (y\triangleleft z) = (x \triangleleft y)\triangleleft z$
is relaxed to only imposing that the associator $x\triangleleft (y\triangleleft z) - (x \triangleleft y)\triangleleft z$ is symmetry in $x$ and $y$.
See~\cite{Manchon} for a short survey on pre-Lie algebras.

The name `pre-Lie' comes from the fact that $[x,y] \eqdef x\triangleleft y - y\triangleleft x$ is a Lie bracket (which is a simple exercise to verify).
We give two prototypical examples of pre-Lie algebras which do \textit{not} come from an associative product.

\begin{example}
The space $C^\infty(\R,\R)$ 
equipped with $\triangleleft \colon (f,g) \mapsto f D g$
is a pre-Lie algebra.
Indeed, the associator $f\triangleleft (g\triangleleft h) - (f \triangleleft g)\triangleleft h$ is equal to $fg D h$,
which is symmetric in $f,g$ as required. 
The associated Lie bracket is the familiar  Lie bracket of vector fields 
$[f,g]=f D g - g D f$.
\end{example}

\begin{example}\label{ex:graft_trees}
The space $(\CH,\curvearrowright)$ is a pre-Lie algebra, where $\CH$ is the span of combinatorial rooted trees with no decorations
and $\curvearrowright$ is the grafting operator,
i.e. $\tau \curvearrowright \sigma$ is the sum of trees obtained by grafting $\tau$ to every node of $\sigma$ using a new edge.
For example,
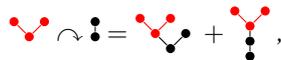
\begin{equ}
\intonesq \graft \intone = \graftone \;+ \grafttwo\;,
\end{equ}
where we colour the grafted tree in red to make the operation easier to follow.
Both the left- and right-hand sides of the pre-Lie identity~\eqref{eq:pre-Lie} are then equal to grafting $x$ and $y$ to nodes of $z$ using new edges (but \textit{not} grafting $x$ to $y$ or vice versa), which is clearly symmetric in $x$ and $y$.

It was shown in~\cite{Chapoton01} (see also~\cite{Dzhumadildaev02})
that $(\CH,\graft)$ is the \textit{free} pre-Lie algebra generated by $\bone$ (the tree with a single node).
That is, for any other pre-Lie algebra $(V,\triangleleft)$ and element $v\in V$,
there exists a unique linear map $\phi\colon\CH\to V$ such that $\phi(\bone)=v$ and such that $\phi$ is a pre-Lie morphism,
i.e. $\phi(\tau\graft\sigma)=\phi(\tau)\triangleleft \phi(\sigma)$.
\end{example}
Both examples above will be relevant to the proof of Lemma~\ref{lem:Upsilon_MF}.
In fact, we will generalise these examples as follows.
The space of non-linearities comes with a \textit{family} of natural pre-Lie products $\{\triangleleft_p\}_{p\in\N^d}$
defined by\label{pre_Lie_page_ref}
\begin{equ}
F \triangleleft_p G = F D_{(\mcI,p)} G\;.
\end{equ}
These operators satisfy, for all $p,q\in\N^d$,
\begin{equ}
F \triangleleft_p  (G \triangleleft_q H) -  (F \triangleleft_p G) \triangleleft_q H = G \triangleleft_q (F \triangleleft_p H) - (G \triangleleft_q F) \triangleleft_p H\;,
\end{equ}
which is a generalisation of the pre-Lie identity~\eqref{eq:pre-Lie}.

On the side of trees, it is again
easier to define analogous grafting operators on the space of trees $\CB$ with enriched polynomial decorations.
Before proceeding, we introduce a piece of notation that will simplify the definitions of several operators below.
\begin{definition}\label{def:mcL_lift}
Consider a linear operator $\Phi\colon\CB\to\CB$.
Let $\mcL(\Phi) \colon \CB\to\CB$ be the linear operator defined for $\sigma$ of the form~\eqref{eq:tree_in_mfB} by
\begin{equ}
\mcL(\Phi) \sigma = Y\sum_{i \in J} \mcI_{m_i}[\Phi\sigma_i]\prod_{\substack{j\in J\\ j\neq i}} \mcI_{m_j}[\sigma_j]\;.
\end{equ}
We make an analogous definition for linear operators $\Phi\colon\CV^\ex\to\CV^\ex$.
\end{definition}
Intuitively, the operator $\mcL(\Phi)$ is a `lift' of $\Phi$ which acts separately on the branches of $\sigma$ and leaves the root untouched.
This construction allows us to define certain operator through recursion using the heuristic formula
\begin{equ}
\Phi = \text{[action at the root] + [lift the operator to the branches],}
\end{equ}
where the first term does not involve $\Phi$ and the second term is $\mcL(\Phi)$.
The next definition illustrates how this works.

For $p\in \N^d$, we define the linear `grafting' operator
\begin{equ}
\graft_p\colon \CB \otimes \CB \to \CB\;, \quad \graft_p\colon \tilde\sigma\otimes \sigma \mapsto \tilde\sigma \graft_p\sigma
\end{equ}
as follows.
Consider $\tilde \sigma,\sigma \in \mfB$
with  $\sigma$ of the form~\eqref{eq:tree_in_mfB}.
Then we define inductively\label{graft_page_ref}
\begin{equs}
\tilde \sigma \graft_p \sigma
&\eqdef \tilde \sigma \graft_p^{\root} \sigma  + \tilde \sigma \graft_p^{\poly} \sigma
+ \tilde \sigma \graft_p^{\nonroot} \sigma
\\
&\eqdef \mcI_p[\tilde \sigma] \sigma
+ \bone^a\Big(
\sum_{i \in K}
\delta_{p,t_i} \CI_{s_{i}}[\tilde \sigma]
\prod_{\substack{k \in K\\ k\neq i}} \CJ_{s_k}[X^{t_k}]
\Big)
\Xi_L
\prod_{j\in J} \mcI_{m_j}[\sigma_j]
\\
&\quad + \mcL(\tilde\sigma\graft_p) \sigma\;,
\end{equs}
where $\delta$ is the Kronecker delta.
The two terms $\tilde \sigma \graft_p^{\root} \sigma  + \tilde \sigma \graft_p^{\poly} \sigma$
correspond to the action of $\tilde\sigma \graft_p$ at the root, while $\tilde \sigma \graft_p^{\nonroot} \sigma = \mcL(\tilde \sigma \graft_p )(\sigma)$
is the lift of $\tilde \sigma \graft_p$ to the branches of $\sigma$,
i.e. for $\sigma$ of the form~\eqref{eq:tree_in_mfB},
\begin{equ}
\mcL(\tilde\sigma\graft_p) \sigma = Y
\Big(
\sum_{i \in J} \mcI_{m_{i}}[\tilde \sigma \graft_p
\sigma_{i}]
\prod_{\substack{j \in J\\ j\neq i}} \mcI_{m_j}[\sigma_j]
\Big)\;.
\end{equ}
\begin{example}
Suppose $p,q\in\N^d$ are distinct. Then, for any $\tilde\sigma\in \mfB$,
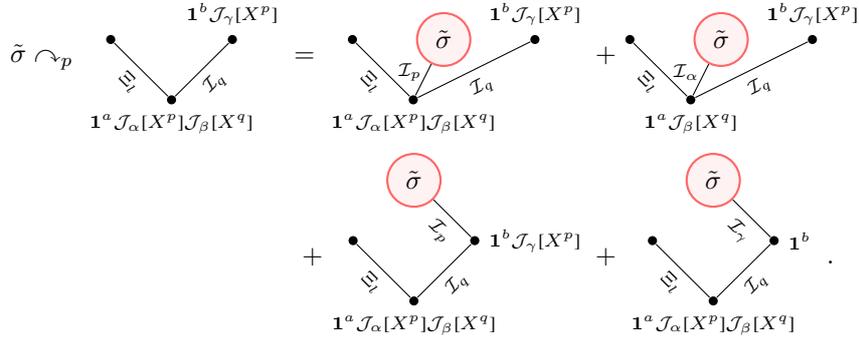
\begin{equs}
\tilde \sigma \graft_p
\begin{tikzpicture}  [scale=0.8,baseline=-0.3cm] 
    \node [dot,label=-90:{\scriptsize $\bone^a\CJ_{\alpha}[X^{p}]\CJ_{\beta}[X^{q}]$}] (m) at (-2, -1) {};
    \node [dot,label=90:{\scriptsize $\bone^b\CJ_{\gamma}[X^p]$}] (ml) at (-1, 0) {};
    \node [dot] (new) at (-3, 0) {};
    \draw (m) to node[sloped, midway,below]{\scriptsize $\CI_q$} (ml);
    \draw (m) to node[sloped, midway,below] {\scriptsize $\Xi_l$} (new);
\end{tikzpicture}
=
\begin{tikzpicture} [
roundnode/.style={circle, draw=red!60, fill=red!5, thick, minimum size=7mm},
scale=0.8,baseline=-0.3cm] 
    \node [dot,label=-90:{\scriptsize $\bone^a\CJ_{\alpha}[X^{p}]\CJ_{\beta}[X^{q}]$}] (m) at (-2, -1) {};
    \node [dot,label=90:{\scriptsize $\bone^b\CJ_{\gamma}[X^p]$}] (ml) at (0, 0) {};
    \node [dot] (new) at (-3, 0) {};
\node [roundnode] (new1) at (-1.5, 0) {$\tilde\sigma$};
\draw (m) to node[midway,left,yshift=1mm,xshift=1mm]{\scriptsize $\CI_p$} (new1);
    \draw (m) to node[sloped, midway,below]{\scriptsize $\CI_q$} (ml);
    \draw (m) to node[sloped, midway,below] {\scriptsize $\Xi_l$} (new);
\end{tikzpicture}
&+
\begin{tikzpicture} [
roundnode/.style={circle, draw=red!60, fill=red!5, thick, minimum size=7mm},
scale=0.8,baseline=-0.3cm] 
    \node [dot,label=-90:{\scriptsize $\bone^a\CJ_{\beta}[X^{q}]$}] (m) at (-2, -1) {};
    \node [dot,label=90:{\scriptsize $\bone^b\CJ_{\gamma}[X^p]$}] (ml) at (0, 0) {};
    \node [dot] (new) at (-3, 0) {};
\node [roundnode] (new1) at (-1.5, 0) {$\tilde\sigma$};
\draw (m) to node[midway,left,yshift=1mm,xshift=1mm]{\scriptsize $\CI_\alpha$} (new1);
    \draw (m) to node[sloped, midway,below]{\scriptsize $\CI_q$} (ml);
    \draw (m) to node[sloped, midway,below] {\scriptsize $\Xi_l$} (new);
\end{tikzpicture}
\\
+
\begin{tikzpicture} [
roundnode/.style={circle, draw=red!60, fill=red!5, thick, minimum size=7mm},
scale=0.8,baseline=-0.3cm] 
    \node [dot,label=-90:{\scriptsize $\bone^a\CJ_{\alpha}[X^{p}]\CJ_{\beta}[X^{q}]$}] (m) at (-2, -1) {};
    \node [dot,label=0:{\scriptsize $\bone^b\CJ_{\gamma}[X^p]$}] (ml) at (-1, 0) {};
    \node [dot] (new) at (-3, 0) {};
\node [roundnode] (new1) at (-2, 1) {$\tilde\sigma$};
\draw (ml) to node[midway,sloped,below]{\scriptsize $\CI_p$} (new1);
    \draw (m) to node[sloped, midway,below]{\scriptsize $\CI_q$} (ml);
    \draw (m) to node[sloped, midway,below] {\scriptsize $\Xi_l$} (new);
\end{tikzpicture}
&+
\begin{tikzpicture} [
roundnode/.style={circle, draw=red!60, fill=red!5, thick, minimum size=7mm},
scale=0.8,baseline=-0.3cm] 
    \node [dot,label=-90:{\scriptsize $\bone^a\CJ_{\alpha}[X^{p}]\CJ_{\beta}[X^{q}]$}] (m) at (-2, -1) {};
    \node [dot,label=0:{\scriptsize $\bone^b$}] (ml) at (-1, 0) {};
    \node [dot] (new) at (-3, 0) {};
\node [roundnode] (new1) at (-2, 1) {$\tilde\sigma$};
\draw (ml) to node[midway,sloped,below]{\scriptsize $\CI_{\gamma}$} (new1);
    \draw (m) to node[sloped, midway,below]{\scriptsize $\CI_q$} (ml);
    \draw (m) to node[sloped, midway,below] {\scriptsize $\Xi_l$} (new);
\end{tikzpicture}
\;.
\end{equs}
The first two terms on the right-hand side above correspond to $\tilde\sigma\graft_p^{\root} \sigma$ and $\tilde\sigma\graft_p^{\poly} \sigma$ respectively, while the final two terms correspond to $\tilde\sigma\graft_p^{\nonroot}\sigma$.
Red circles with $\tilde\sigma$ indicate that the root of $\tilde\sigma$ is attached to the corresponding incoming edge.
\end{example}
\begin{remark}
One should think of $\graft_p$ as related to the adjoint of the linearisation of $\Deltap_\ex$ (see~\eqref{eq:adjoint_bgraft} for a precise statement) and in particular does not alter the extended decorations.
\end{remark}
For an extended non-linearity $\cbF$,
it follows almost immediately from the definition that $\mathring\Upsilon^\cbF$
is a pre-Lie morphism in the sense that, for all $p\in\N^d$,
\begin{equ}\label{eq:pre-Lie_morph}
\mathring\Upsilon^\cbF[\tilde \sigma \graft_p
\sigma] =
\mathring\Upsilon^\cbF[\tilde \sigma]
\triangleleft_p \mathring\Upsilon^\cbF[\sigma]\;.
\end{equ}
To see this, we simply apply the Leibniz rule to the expression~\eqref{eq:mathring_Upsilon_def}
to obtain three terms;
the two terms obtained by $D_{(\mcI,p)}$ hitting $F_a$ and $\prod_{k\in K} \CY_{(\mcI,t_k)}$
match precisely
$\mathring\Upsilon^\cbF[\tilde\sigma\graft_p^{\root}\sigma]$ and
$\mathring\Upsilon^\cbF[\tilde\sigma\graft_p^{\poly}\sigma]$ respectively,
while the third term obtained
by $D_{(\mcI,p)}$ hitting $\prod_{j\in J}
\mathring\Upsilon^\cbF[\sigma_j]$
matches
$\mathring\Upsilon^\cbF[\tilde\sigma\graft_p^{\nonroot}\sigma]$ by an induction.

Finally, for $b\in[d]$,
we define a `raising' operator $\uparrow_{b}$ which maps every $\sigma\in \mfB$ of the form~\eqref{eq:tree_in_mfB} to a \textit{series} of trees (so not strictly an element of $\CB$ according to our definition)
given inductively by\label{uparrow_page_ref}
\begin{equs}
\uparrow_{b} \sigma
&= \sum_{p\in\N^d}\CJ_p[X^{p+e_b}] \sigma
+
\bone^a\Big(
\sum_{i \in K}
\CJ_{s_{i}}[X^{t_i+e_b}]
\prod_{\substack{k \in K\\ k\neq i}} \CJ_{s_k}[X^{t_k}]
\Big)
\Xi_L
\prod_{j\in J} \mcI_{m_j}[\sigma_j]
\\
&\quad
+ \mcL(\uparrow_b) \sigma
\;.
\end{equs}
\begin{example}
\begin{equs}
\uparrow_{b}
\begin{tikzpicture}  [scale=0.8,baseline=-0.3cm] 
    \node [dot,label=-90:{\scriptsize $\bone^a\CJ_{\alpha}[X^{r}]\CJ_{\beta}[X^{q}]$}] (m) at (-2, -1) {};
    \node [dot,label=90:{\scriptsize $\bone^b\CJ_{\gamma}[X^n]$}] (ml) at (-1, 0) {};
    \node [dot] (new) at (-3, 0) {};
    \draw (m) to node[sloped, midway,below]{\scriptsize $\CI_q$} (ml);
    \draw (m) to node[sloped, midway,below] {\scriptsize $\Xi_l$} (new);
\end{tikzpicture}
&=
\sum_{p\in\N^d}
\begin{tikzpicture} [
scale=0.8,baseline=-0.3cm] 
    \node [dot,label=-90:{\scriptsize $\bone^a\CJ_{p}[X^{p+e_b}]\CJ_{\alpha}[X^{r}]\CJ_{\beta}[X^{q}]$}] (m) at (-2, -1) {};
    \node [dot,label=90:{\scriptsize $\bone^b\CJ_{\gamma}[X^n]$}] (ml) at (-1, 0) {};
    \node [dot] (new) at (-3, 0) {};
    \draw (m) to node[sloped, midway,below]{\scriptsize $\CI_q$} (ml);
    \draw (m) to node[sloped, midway,below] {\scriptsize $\Xi_l$} (new);
\end{tikzpicture}
+
\begin{tikzpicture} [
scale=0.8,baseline=-0.3cm] 
    \node [dot,label=-90:{\scriptsize $\bone^a\CJ_{\alpha}[X^{r+e_b}]\CJ_{\beta}[X^{q}]$}] (m) at (-2, -1) {};
    \node [dot,label=90:{\scriptsize $\bone^b\CJ_{\gamma}[X^n]$}] (ml) at (-1, 0) {};
    \node [dot] (new) at (-3, 0) {};
    \draw (m) to node[sloped, midway,below]{\scriptsize $\CI_q$} (ml);
    \draw (m) to node[sloped, midway,below] {\scriptsize $\Xi_l$} (new);
\end{tikzpicture}
\\
&\quad +
\begin{tikzpicture} [
scale=0.8,baseline=-0.3cm] 
    \node [dot,label=-90:{\scriptsize $\bone^a\CJ_{\alpha}[X^{r}]\CJ_{\beta}[X^{q+e_b}]$}] (m) at (-2, -1) {};
    \node [dot,label=90:{\scriptsize $\bone^b\CJ_{\gamma}[X^n]$}] (ml) at (-1, 0) {};
    \node [dot] (new) at (-3, 0) {};
    \draw (m) to node[sloped, midway,below]{\scriptsize $\CI_q$} (ml);
    \draw (m) to node[sloped, midway,below] {\scriptsize $\Xi_l$} (new);
\end{tikzpicture}
+\sum_{p\in\N^d}
\begin{tikzpicture} [
scale=0.8,baseline=-0.3cm] 
    \node [dot,label=-90:{\scriptsize $\bone^a\CJ_{\alpha}[X^{r}]\CJ_{\beta}[X^{q}]$}] (m) at (-2, -1) {};
    \node [dot,label=90:{\scriptsize $\bone^b\CJ_{p}[X^{p+e_b}]\CJ_{\gamma}[X^n]$}] (ml) at (-1, 0) {};
    \node [dot] (new) at (-3, 0) {};
    \draw (m) to node[sloped, midway,below]{\scriptsize $\CI_q$} (ml);
    \draw (m) to node[sloped, midway,below] {\scriptsize $\Xi_l$} (new);
\end{tikzpicture}
\\
&\quad +
\begin{tikzpicture} [
scale=0.8,baseline=-0.3cm] 
    \node [dot,label=-90:{\scriptsize $\bone^a\CJ_{\alpha}[X^{r}]\CJ_{\beta}[X^{q}]$}] (m) at (-2, -1) {};
    \node [dot,label=90:{\scriptsize $\bone^b\CJ_{\gamma}[X^{n+e_b}]$}] (ml) at (-1, 0) {};
    \node [dot] (new) at (-3, 0) {};
    \draw (m) to node[sloped, midway,below]{\scriptsize $\CI_q$} (ml);
    \draw (m) to node[sloped, midway,below] {\scriptsize $\Xi_l$} (new);
\end{tikzpicture}
\end{equs}
\end{example}
Recalling the operator $\partial^b$ defined by~\eqref{eq:partial_def}, an inductive argument easily shows that
\begin{equ}\label{eq:uparrow_morph}
\mathring{\Upsilon}^\cbF[\uparrow_b\sigma] = \partial^b \mathring{\Upsilon}^\cbF[\sigma]\;.
\end{equ}
\begin{remark}\label{rem:redefine_B}
One should really think of $\uparrow_{b} \sigma$
as
$\sum_{p\in\N^d} X^{p+e_b} \graft_p \sigma$.
The trouble with the latter is that, according to our definition, the tree $X^{p+e_b}$ is not in $\mfB$.
By enlarging $\mfB$ and extending the definition of $\mathring{\Upsilon}^\cbF$
in a natural way (i.e. setting $\mathring{\Upsilon}^\cbF[X^p]=\CY_{(\mcI,p)}$),
we could avoid a separate definition for $\uparrow_{b}$.
We stick with our current definitions
partly because $\mfB$ was the natural space to use in the proof of Lemma~\ref{lem:coherent_alt_def}
and partly to be consistent with~\cite{BCCH21}.
\end{remark}
The next step in the proof of Lemma~\ref{lem:Upsilon_MF}
is to relate the identities~\eqref{eq:pre-Lie_morph} and~\eqref{eq:uparrow_morph}
to the space of trees $\mfT_\circ^\ex$ appearing in the lemma statement.
To this end, it becomes useful to describe the grafting and raising operators $\graft_p$ and $\uparrow_b$ in terms of their adjoints,
which is the purpose of the next exercise.

For an operator $\Phi \colon\CB\to\CB\otimes \CB$, let $\mcL(\Phi)\colon \CB\to\CB\otimes\CB$
be the linear operator defined for $\sigma$ of the form~\eqref{eq:tree_in_mfB} by\label{mcL_tensor_page_ref}
\begin{equ}
\mcL(\Phi) \sigma = \sum_{i \in J} (\id\otimes\mcI_{m_i})[\Phi\sigma_i]\Big(\bone\otimes Y\prod_{\substack{j \in J\\ j\neq i}} \mcI_{m_j}[\sigma_j]\Big)\;.
\end{equ}
That is, $\mcL(\Phi)$ lifts $\Phi$ to every branch $\sigma_i$ but adds the edge $\mcI_{m_i}$ only to the right factor of $\Phi \sigma_i$.
We make the same definition for $\Phi\colon\CV^\ex\to\CV^\ex\otimes\CV^\ex$.

\begin{exercise}
Let $\graft_p^*$ be the adjoint of $\graft_p$ for the inner products $\scal{\cdot,\cdot}_\CB$ and $\scal{\cdot,\cdot}_{\CB\otimes\CB}$.
For $\sigma\in\mfB$ of the form~\eqref{eq:tree_in_mfB}, show that $\graft_p^*\sigma \in \CB\otimes \CB$ and is given by
\begin{equs}
\graft_p^*\sigma
&= \sum_{0 \leq b \leq p} \sum_{i \in J} \frac{\delta_{k,m_i}}{(p-b)!}\sigma_i\otimes
Y \CJ_b[X^p] \prod_{\substack{j \in J\\ j\neq i}} \mcI_{m_j}[\sigma_j]
+ \mcL(\graft_p^*)\sigma\;,
%\sum_{\jmath =1}^m \sum\sigma_{\jmath}^{(1)}\otimes Y \mcI_{p_\jmath}[\sigma_\jmath^{(2)}] \prod_{\substack{1\leq j \leq m\\ j\neq \jmath}} \mcI_{p_j}[\sigma_j]
\end{equs}
where $\CJ_p[X^p]\eqdef \bone$.
%where we used the Sweedler-type notation $\graft_p^* \sigma_\jmath = \sum \sigma_\jmath^{(1)}\otimes \sigma_\jmath^{(2)}$.
Similarly, show that $\uparrow_b^*$, the adjoint of $\uparrow_b$, maps $\CB$ to $\CB$ and is given for
$\sigma\in\mfB$ of the form~\eqref{eq:tree_in_mfB} by
\begin{equs}
\uparrow_b^*\sigma &= \sum_{\substack{i \in K\\ t_i[b]> s_i[b]}}
(t_{i}[b]-s_{i}[b]) \CJ_{s_i}[X^{t_i-e_b}]
\Xi_L
\prod_{\substack{k \in K\\ k\neq i}} \CJ_{s_k}[X^{t_k}] \prod_{j\in J}\mcI_{m_j}[\sigma_j]
+ \mcL(\uparrow_b^*)\sigma \;,
\end{equs}
where again $\CJ_p[X^p]\eqdef \bone$.
\end{exercise}
\begin{remark}
The operator $\graft_p^*$ has a natural interpretation as a sum over cuts of edges, see~\cite[Lem.~4.10]{BCCH21} -- the tree above the cut is placed on the left of the original tree.
\end{remark}
We now define the linear `grafting' operator $\hgraft_p \colon \CV^\ex \otimes \CV^\ex \to \CV^\ex$
for $\tau\in\mfT^\ex$ of the form~\eqref{eq:general_tree_ext} and $\bar \tau\in\mfT^\ex$ inductively by\label{hgraft_page_ref}
\begin{equs}\label{eq:hgraft_def}
\bar \tau \hgraft_p \tau
&= \sum_{0\leq r \leq k} \binom{k}{r} X^{k-r}\bone^a
\Xi_L
\mcI_{p-r}[\bar\tau] \prod_{j\in J}\mcI_{m_j}[\tau_j]
+\mcL(\bar\tau\hgraft_p)\tau\;,
\end{equs}
where $\CI_{p-r}[\bar\tau]\eqdef 0$ if $p-r\notin\N^d$.
It is straightforward to verify that the adjoint operator $\hgraft_p^* \colon \CV^\ex\to\CV^\ex\otimes\CV^\ex$ is given by
\begin{equ}
\hgraft_p^* \tau = \sum_{0\leq n \leq p} \sum_{i\in J} \tau_i
\otimes \frac{\delta_{m_i,n}}{(p-n)!}
X^{p-n} \bone^a\Xi_L \prod_{\substack{j\in J\\ j\neq i}}\mcI_{m_j}[\tau_j] + \mcL(\hgraft_p^*)\tau\;,
\end{equ}
which again can be written as a sum over cuts, see~\cite[Def.~4.11]{BCCH21}.
The point behind this definition is the readily verified identity
\begin{equ}\label{eq:Q_morph_graft}
(Q\otimes Q) \graft_p^* = \hgraft_p^* Q
\end{equ}
(see~\cite[Lem.~4.14]{BCCH21}).

We similarly define for $b\in[d]$ the operator $\hat\uparrow_b \colon \CV^\ex\to\CV^\ex$ by\label{hat_uparrow_page_ref}
\begin{equ}
\hat\uparrow_b \tau = X^{e_b} \tau + \mcL(\hat\uparrow_b)\tau\;.
\end{equ}
The adjoint $\hat\uparrow_b^*\colon\CV^\ex\to\CV^\ex$ of $\hat\uparrow_b$ is given for $\tau$ of the form~\eqref{eq:general_tree_ext} by $\hat\uparrow_b^* \tau=0$ if $k[b]=0$ and, if $k[b]>0$,
\begin{equ}
\hat\uparrow_b^* \tau = k[b] X^{-e_b} \tau + \mcL(\hat\uparrow_b^*)\tau\;,
\end{equ}
where $X^{-e_b} \tau$ carries the obvious meaning (the polynomial decoration changes from $k$ to $k-e_b$).
As before, it is easy to see the morphism identity
\begin{equ}\label{eq:Q_morph_raise}
\hat\uparrow^*_b Q = Q \uparrow^*_b\;.
\end{equ}
Combining Lemma~\ref{lem:Upsilon_Q} with~\eqref{eq:pre-Lie_morph} and~\eqref{eq:Q_morph_graft}, we obtain for any extended non-linearity $\cbF$
\begin{equ}\label{eq:Upsilon_morph}
\Upsilon^\cbF[\bar \tau \hgraft_p \tau] = \Upsilon^\cbF [\bar \tau] \triangleleft_p \Upsilon^\cbF [\tau]\;,
\end{equ}
which is an analogue of~\eqref{eq:pre-Lie_morph} for trees in $\mfT^\ex$.
Similarly, combining Lemma~\ref{lem:Upsilon_Q} with~\eqref{eq:uparrow_morph} and~\eqref{eq:Q_morph_raise}, we obtain
\begin{equ}\label{eq:hat_uparrow_Upsilon}
\Upsilon^\cbF[\hat\uparrow_b\tau] = \partial^b\Upsilon^\cbF[\tau]\;.
\end{equ}

We now come to an important universality property of $(\CV^\ex,\{\hgraft_p\}_{p\in\N^d})$.
Define the set\label{mfG_page_ref}
\begin{equ}
\mfG = \{\bone^a X^k \Xi_L \,:\,
a\in \R\,,\, 
k\in\N^d\,,\,
L\in\mfN\}
\subset \mfT^\ex\;.
\end{equ}
We think of $\mfG$ as the set of \textit{generators} of $\CV^\ex$ for the following reason.

\begin{lemma}\label{lem:free_gen}
The space $\CV^\ex$ is freely generated by $\mfG$ and the operators $\{\hgraft_p\}_{p\in\N^d}$ in the following sense.
Consider a vector space $V$ with a family of bilinear operators $\{\triangleleft_\alpha\}_{\alpha\in A}$,
$\triangleleft_\alpha\colon V\times V\to V$ which satisfies the pre-Lie identity
\begin{equ}\label{eq:mult_pre_Lie}
x \triangleleft_\alpha( y \triangleleft_\beta z)
- (x \triangleleft_\alpha y) \triangleleft_\beta z
=
y \triangleleft_\beta (x \triangleleft_\alpha z)
- (y \triangleleft_\beta x) \triangleleft_\alpha z
\end{equ}
for all $\alpha,\beta\in A$ and $x,y,z\in V$.
Then, for any map $\Phi\colon\mfG\to V$
and $\Psi\colon \N^d \to A$, there exists a unique extension of $\Phi$ to a linear map $\Phi \colon \CV^\ex\to V$
which satisfies the pre-Lie morphism identity
\begin{equ}
\Phi(\bar\tau \hgraft_p \tau) = (\Phi\bar\tau) \triangleleft_{\Psi(p)} (\Phi\tau)
\end{equ}
for all $\tau,\bar\tau\in\mfT^\ex$ and $p\in\N^d$.
\end{lemma}

\begin{remark}
Lemma~\ref{lem:free_gen} extends the results of~\cite{Chapoton01,Dzhumadildaev02} discussed in Example~\ref{ex:graft_trees}.
The proof of this result is given in~\cite[App.~A.4]{BCCH21} and follows the strategy of~\cite{Chapoton01}.

A similar form of Lemma~\ref{lem:free_gen} holds for the space $(\CB,\{\graft_p\}_{p\in\N^d})$, although we do not state it here because it is the universality property of $\CV^\ex$ that will be used.
If one redefines $\mfB$ as in Remark~\ref{rem:redefine_B},
the set of generators simply becomes $\{X^k\,:\,k\in\N^d\} \sqcup \{\bone^a\Xi_L \,:\, a\in\R\,,\,L\in\mfN\}$.
\end{remark}

\begin{remark}
In the rest of the argument, we will exploit the \textit{generation} property of $\CV^\ex$ in Lemma~\ref{lem:free_gen} rather than the \textit{freeness} property, i.e.
it is the uniqueness of the extension of $\Phi$ that will be important below, rather than its existence.
\end{remark}

Remark that the definition of $\hgraft_p$, though very natural in light of~\eqref{eq:Q_morph_graft} and~\eqref{eq:Upsilon_morph},
is not expected to interact well with the negative renormalisation group $G_-^\ex$ because, even for $\bar\tau,\tau\in\mfT_\circ^\ex$, it is not always true that $\bar\tau\hgraft_p\tau$ is an element of $\CT^\ex = \Span_\R(\mfT^\ex_\circ) \subset \CV^\ex$.
On the other hand, recall that we aim to prove Lemma~\ref{lem:Upsilon_MF} only for $F$ that obey $R$ and for extended non-linearities $\cbF\in F^{G_-^\ex}$.
The following proposition and its corollary show that we can effectively restrict attention to trees in $\mfT^\ex_\circ$ in this case.

Recall from Exercise~\ref{ex:trees_in_mfT^ex_circ},
that for $\tau= T^{\mfn,\mfo}_\mfe \in\mfT^\ex$ it holds that $\tau\in\mfT^\ex_\circ$ if and only if $\CN(x)\in \tilde R(\mfo(x))$
where $\tilde R\colon \R\to \CP(\N^\CE)$.
The following proposition shows how the notion of \textit{obey} from Definition~\ref{def:obey} extends to $F^{G^\ex_-}$.

\begin{proposition}\label{prop:obey_extended}
Suppose $F$ obeys $R$ and let $\cbF\in F^{G^\ex_-}$.
Let $O = \{o_n\}_{n=1}^N \in \N^\CE$ be a multiset
and $a\in \R$
such that $O \notin\tilde R(a)$.
Then $\big(\prod_{o\in O} D_o\big) F_a = 0$.
\end{proposition}

\begin{proof}
Let $g \in G^\ex_-$ such that $\cbF = F^g$, where we identify $F$ with the reduced non-linearity in $F^{G^\ex_-}$,
and recall that $\Upsilon^\cbF[\bone^a] = \Upsilon^F[M_g^* \bone^a]$.
Consider the set
\begin{equ}
A_-\eqdef (-\infty,0]\cap \{|\tau|_\s\,:\, \tau\in\mfT_\circ\}\;,
\end{equ}
which we recall is finite.
It follows from the definition of $\mfT^\ex_\circ$ that
$\bone^a\in\mfT_\circ^\ex$ only if $a\in A_-$.
Therefore $F_a = 0$ if $a\notin A_-$,
so it suffices to assume henceforth that $a\in A_-$.

By the definition of obey, we have $\Upsilon^F[\tau]=0$ for all $\tau\in \mfT^\ex \setminus \mfT_\circ$,
so we may write
\begin{equ}\label{eq:bone^a_Upsilon}
\Upsilon^\cbF[\bone^a] = \sum_{i} c_i\Upsilon^F[\tau_i]\;,
\end{equ}
where $c_i\in\R$ and $\tau_i\in\mfT_\circ$ with $|\tau_i|_\s = a$ and the sum is finite.
Let us fix one index $i$ in this sum.

Let us introduce a new distinguished node with decoration $\square$.
Consider the sum of trees $\hat\tau$ formed by grafting $N$ nodes with decoration $\square$ to $\tau_i$, i.e.
\begin{equ}\label{eq:grafting_new_nodes}
\hat\tau
\eqdef 
\square\hgraft_{o_1}(\square\hgraft_{o_2}(\ldots (\square\hgraft_{o_{N}}\tau_i)\ldots)) = \sum_{j} b_j\sigma_j\;.
\end{equ}
Here $b_j\in \R$ and the sum is finite.
The trees $\sigma_j$ are not strictly elements of $\mfT^\ex$ and $\square$ may appear as a decoration of a noise node.
The grafting operator $\square\hgraft_{o_n}\sigma$ for $o_n$ of the form $\Xi_q$ is defined exactly as~\eqref{eq:hgraft_def}
\textit{except} there is no summation over $r\leq k$, i.e. we simply graft the node with decoration $\square$ onto every node of $\sigma$ using the edge $\Xi_q$.

We define $\Upsilon^F[\square]=1$ and extend the domain of $\Upsilon^F$ using~\eqref{eq:Upsilon_def} to all the trees $\sigma_j$ in~\eqref{eq:grafting_new_nodes}.
By a simple extension of the pre-Lie morphism property~\eqref{eq:Upsilon_morph}, we obtain  \begin{equ}\label{eq:hat_tau_Upislon}
\Upsilon^F[\hat\tau] = \Big(\prod_{o\in O} D_o\Big) \Upsilon^F[\tau_i]\;.
\end{equ}
Furthermore, if $\sigma_j$ has two adjacent nodes with decorations $\square$, then by the definition $\Upsilon^{F}[\square]=1$, we have $\Upsilon^{F}[\sigma_j]=0$.
On the other hand, if $\sigma_j$ has no adjacent nodes with decoration $\square$,
then necessarily $\CN(x) \notin R$ for some node $x\in \sigma_j$ by the assumption that $O\notin\tilde R(a)$ and by the definition of $\tilde R$.
Since $F$ obeys $R$, we obtain $\Upsilon^F[\sigma_j]=0$,
and thus 
\begin{equ}
\Upsilon^F[\hat\tau] = \Big(\prod_{o\in O} D_o\Big) \Upsilon^F[\tau_i]=0\;.
\end{equ}
Since the holds for every $\tau_i$ in~\eqref{eq:bone^a_Upsilon}, the conclusion follows.
\end{proof}
The following is now an immediate corollary of Exercise~\ref{ex:trees_in_mfT^ex_circ}, the definition of $\Upsilon^\cbF$, and Proposition~\ref{prop:obey_extended}.
\begin{corollary}\label{cor:vanish_bad_trees}
Suppose $F$ obeys $R$ and let $\cbF\in F^{G^\ex_-}$.
Then $\Upsilon^{\cbF}[\tau]=0$ for all $\tau\in\mfT^\ex\setminus\mfT^\ex_\circ$.
\end{corollary}
It is now natural to introduce one final grafting operator $\bgraft_p \colon \mcT^\ex\otimes\mcT^\ex \to \mcT^\ex$
defined by\label{bgraft_page_ref}
\begin{equ}
\bar\tau\bgraft_p \tau \eqdef \pi_R (\bar\tau \hgraft_p \tau)\;,
\end{equ}
where $\pi_R \colon\CV^\ex\to\CT^\ex$\label{pi_R_page_ref} is the canonical projection that annihilates every tree in $\mfT^\ex\setminus\mfT_\circ^\ex$.
The following exercise
shows that $\pi_R$ is a pre-Lie morphism
and that $(\CT^\ex,\{\bgraft_p\}_{p\in\N^d})$ is generated by\label{mfG_circ_page_ref}
\begin{equ}
\mfG_\circ \eqdef \mfG\cap \mfT_\circ^\ex
\end{equ}
(though not freely).
\begin{exercise}\label{exercise:pi_morph}
Show that, for all $\tau,\bar\tau\in\CV^\ex$ and $p\in\N^d$,
\begin{equ}
\pi_R(\bar\tau \hgraft_p \tau ) = 
(\pi_R \bar\tau) \bgraft_p (\pi_R\tau) \;.
\end{equ}
Conclude that the same statement as Lemma~\ref{lem:free_gen} holds with $(\CV^\ex,\mfG,\hgraft_p)$ replaced by $(\CT^\ex,\mfG_\circ,\bgraft_p)$
and ``there exists a unique extension'' replaced by ``there exists at most one extension''.
\end{exercise} 
The final ingredient we require for the proof of Lemma~\ref{lem:Upsilon_MF}
is the fact that $M_g^*$ for every $g\in G_-^\ex$ is a pre-Lie morphism on $\CT^\ex$ and commutes with $\hat\uparrow_b$.
\begin{lemma}\label{lem:M_pre-Lie_morph}
For all $g\in G_-^\ex$,  $\tau,\bar\tau\in\CT^\ex$, and $p\in\N^d$,
\begin{equ}\label{eq:M*_morphism}
(M_g^*\bar\tau) \bgraft_p (M_g^*\tau) = M_g^*(\bar\tau\bgraft_p\tau)\;.
\end{equ}
Moreover, for every $b\in[d]$,
\begin{equ}\label{eq:M*_commute_raise}
M^*_g\hat\uparrow_b = \hat\uparrow_b M_g^*\;.
\end{equ}
\end{lemma}
To prove Lemma~\ref{lem:M_pre-Lie_morph},
we identify $\bgraft_p^*$ with the linearisation of a coproduct similar to $\Deltap_\ex$,
where $\bgraft_p^*\colon \CT^\ex\to \CT^\ex\otimes \CT^\ex$ is the adjoint of $\bgraft_p^*$.
This is a version of the classical link between the Connes--Kreimer coproduct and the grafting operation on trees from Example~\ref{ex:graft_trees}.
Indeed,
one can verify that
\begin{equ}\label{eq:adjoint_bgraft}
\bgraft_p^* = \CM_{12} (\id \otimes \mfp_p) \Delta_2\;,
\end{equ}
where
\begin{itemize}
\item $\CM_{12}(\tau_1\otimes \tau_2) = \tau_2\otimes \tau_1$,

\item $\mfp_p \colon \CT_+^\ex\to\CT^\ex$ is the linear function that maps $\mcI_p[\tau]\in\mfT_+^\ex$ to $\tau\in\mfT_\circ^\ex$ and annihilates every tree in $\mfT_+^\ex$ not of the form $\mcI_p[\tau]$, and

\item $\Delta_2 \colon \CT^\ex \to \CT^\ex\hat\otimes \CV^\ex$ is defined identically to $\Deltap_\ex$ in Definition~\ref{def:Deltap_ex} except that in~\eqref{eq:Deltap_def} we take the sum over all $k\in\N^d$, not just $|k|_\s < |\mcI_p\tau|_+$,
and $\CT^\ex\hat\otimes \CV^\ex$ denotes an appropriate space of tensor series.
\end{itemize}
\begin{remark}
Although $\Delta_2 \tau$ is a series of tensor products, $(\id\otimes \mfp_p)\Delta_2 \tau$ is a finite sum.
\end{remark}
We similarly have
\begin{equ}\label{eq:adjoint_hat_uparrow}
\hat\uparrow_{b}^* = \CM_1(\id\otimes \mfp_{X^b})\Delta_2\;,
\end{equ}
where
\begin{itemize}
\item $\CM_1\colon \CT^\ex\otimes \CT_+^\ex\to\CT^\ex$ is the linear map defined for two trees $(\tau_1,\tau_2)\in \mfT^\ex\times \mfT_+^\ex$ by $\CM_1(\tau_1\otimes\tau_2)=\tau_1$, and

\item $\mfp_{X^b} \colon \CT_+^\ex\to\CT_+^\ex$ is the projection which maps $X^b$ to itself and annihilates all trees $\tau\neq X^b$.
\end{itemize}
A general cointeraction identity proved in~\cite[Prop.~3.27]{BHZ19} is\footnote{This cointeraction does not rely on the extended decorations because
$\Delta_2$ does not truncate the sum in~\eqref{eq:Deltap_def}. This is why the analogue of Theorem~\ref{thm:MU_coherent} holds on the reduced structure as mentioned in Remark~\ref{rem:Upsilon_reduced_structure}.}
\begin{equ}\label{eq:general_cointeract}
\mcM_{(13)(2)(4)} (\Deltam_\ex\otimes\Deltam_\ex)\Delta_2 = (\id\otimes\Delta_2)\Deltam_\ex\;.
\end{equ}
\begin{exercise}
Use~\eqref{eq:adjoint_bgraft} and~\eqref{eq:general_cointeract} to show that
\begin{equ}
\CM_{13}(\Deltam_\ex\otimes\Deltam_\ex)\bgraft_p^* = (\id\otimes \bgraft_p^*)\Deltam_\ex\;.
\end{equ}
Similarly, use~\eqref{eq:adjoint_hat_uparrow} and~\eqref{eq:general_cointeract} to show that
\begin{equ}
(\id\otimes\hat\uparrow_b^*)\Deltam_\ex = \Deltam_\ex \hat\uparrow_b^*\;.
\end{equ}
Hence prove
Lemma~\ref{lem:M_pre-Lie_morph}.
\end{exercise}
We are finally in a position to give the proof of Lemma~\ref{lem:Upsilon_MF}.
\begin{proof}[of Lemma~\ref{lem:Upsilon_MF}]
By the fact that $M_{g h}= M_{g}M_h$ and thus $M_{g h}^* = M_h^*M_g^*$, it suffices to consider the case that $\cbF=F$ is the (unique) reduced non-linearity in $F^{G^\ex_-}$.
Observe that, by Corollary~\ref{cor:vanish_bad_trees} and the identity~\eqref{eq:Upsilon_morph}, $\Upsilon^{\cbF^g}$ is a
pre-Lie morphism from $\CT^\ex$ to the space of non-linearities which takes the operator $\bgraft_p$ to $\triangleleft_p$ for every $p\in\N^d$,
i.e.
\begin{equ}
\Upsilon^{\cbF^g}[\bar\tau\bgraft_p\tau] = \Upsilon^{\cbF^g}[\bar\tau] \triangleleft_p \Upsilon^{\cbF^g} [\tau]\;.
\end{equ}
On the other hand, by Lemma~\ref{lem:M_pre-Lie_morph} and the identity~\eqref{eq:Upsilon_morph},
$\Upsilon^F \circ M_g^*$ is also a pre-Lie morphism,
i.e.
\begin{equ}
\Upsilon^{F}[M^*_g(\bar\tau\bgraft_p \tau)] = \Upsilon^{F}[M_g^*\bar\tau] \triangleleft_p \Upsilon^{F} [M_g^* \tau]\;.
\end{equ}
We now claim that, for all $\bone^a X^k\Xi_L\in\mfG_\circ$,
\begin{equ}\label{eq:agree_on_G}
\Upsilon^{\cbF^g}[\bone^a X^k\Xi_L] = \Upsilon^F M_g^* (\bone^a X^k\Xi_L)\;.
\end{equ}
Indeed, observe that $\hat\uparrow_b$ and $\hat\uparrow_{c}$ commute for all $b,c\in[d]$,
and thus $\hat\uparrow_k$ makes sense for all $k\in\N^d$ by composition.
Furthermore $\hat\uparrow_k$ commutes with $M_g^*$ by~\eqref{eq:M*_commute_raise}.
Hence, for every $\bone^a X^k\Xi_L\in\mfG_\circ$,
\begin{equs}
\Upsilon^{\cbF^g}[\bone^a X^k\Xi_L] &=\Upsilon^{\cbF^g}[ \hat\uparrow_k\bone^a\Xi_L]
= \partial^k\Upsilon^{\cbF^g}[ \bone^a \Xi_L] = \partial^k \Big(\prod_{l\in L}D_{(\Xi,l)}\Big)\Upsilon^{\cbF^g}[\bone^a]
\\
&=
\partial^k\Big(\prod_{l\in L}D_{(\Xi,l)}\Big)\Upsilon^{F}[M_g^*\bone^a]
= \partial^k\Upsilon^F[\hat\uparrow_k M_g^*(\bone^a\Xi_L)]
\\
&= \Upsilon^F [\hat\uparrow_kM_g^*(\bone^a \Xi_L)]
=\Upsilon^F M_g^*[\hat\uparrow_k \bone^a \Xi_L]
= \Upsilon^F M_g^*[\bone^a X^k \Xi_L]
\;,
\end{equs}
where we used~\eqref{eq:hat_uparrow_Upsilon} in the second and sixth equalities,
the definition~\eqref{eq:F^g_def} in the fourth equality,
and the fact that $M_g^*$ commutes with $\hat\uparrow_k$ in the seventh equality.
Furthermore, in the fifth equality, we used the same argument as in proof of Proposition~\ref{prop:obey_extended} (see~\eqref{eq:bone^a_Upsilon}-\eqref{eq:hat_tau_Upislon}) to obtain
\begin{equ}
\Big(\prod_{l\in L}D_{(\Xi,l)}\Big)\Upsilon^{F}[M_g^*\bone^a]
= \Upsilon^F[M_g^*(\bone^a\Xi_L)]\;.
\end{equ}
We have thus shown~\eqref{eq:agree_on_G}.
To conclude the proof, observe that $\Upsilon^{\cbF^g}$ and $\Upsilon^F \circ M_g^*$ are both pre-Lie morphisms on $\CT^\ex$ which agree on the generating set $\mfG_\circ$.
Hence, by Exercise~\ref{exercise:pi_morph}, $\Upsilon^{\cbF}=\Upsilon^F \circ M_g^*$ as claimed.
\end{proof}
\begin{remark}
Presumably, with enough work, one can rephrase the whole construction of the renormalisation groups $G_+^\ex,G_-^\ex$ and the associated Hopf algebras using the enlarged space of trees $\CB$ in place of $\CV^\ex$.
This would avoid having to come back to $\CT^\ex$ in the final step above and defining the rather complicated grafting operators $\hgraft$,
but would likely require reworking much of the theory from~\cite{BHZ19}.
\end{remark}
\begin{remark}
The role of pre-Lie products was recognised in a rough path context in~\cite{BCFP19}.
The situation therein is simpler due to the absence of polynomial decorations and derivatives in the trees.
Nonetheless, the rough path case serves as an instructive starting point in which the importance of the fact that $\CV^\ex$
is generated by the grafting operators is made clear.
\end{remark}

\appendix

\section{Symbolic index}

We collect in this appendix commonly used symbols of the article together with their meaning and, if relevant, the page where they first occur.
We do not specify the meaning of many `extended' spaces as they are defined analogously to the `reduced' spaces -- see Section~\ref{subsec:extended_struct}.

 \begin{center}
\renewcommand{\arraystretch}{1.1}
\begin{longtable}{lll}
\toprule
Symbol & Meaning & Page\\
\midrule
\endfirsthead
\toprule
Symbol & Meaning & Page\\
\midrule
\endhead
\bottomrule
\endfoot
\bottomrule
\endlastfoot
$|\cdot|_\s$ & Degree map $|\cdot|_\s \colon \mfT\to\R$ & \pageref{deg_page_ref}\\
$|\cdot|_+$ & Degree map $|\cdot|_+ \colon \mfT^\ex\to\R$ including extended decorations & \pageref{deg_ex_page_ref}\\
$\fancynorm{\cdot}$ & Kernel edge counting function $\fancynorm{\cdot} \colon \mfT\to\N$ & \pageref{kernel_count_page_ref}\\
$\partial^k$ & Spatial derivative on non-linearities for $k\in\N^d$ &
 \pageref{eq:partial_def}\\
$\triangleleft_p$ & Pre-Lie product on non-linearities associated to $p\in\N^d$ &
\pageref{pre_Lie_page_ref}\\
$\graft_p$ & Pre-Lie (grafting) product on $\CB$ associated to $p\in\N^d$ &
\pageref{graft_page_ref}\\
$\hgraft_p$ & Pre-Lie (grafting) product on $\CV^\ex$ associated to $p\in\N^d$ &
\pageref{hgraft_page_ref}\\
$\bgraft_p$ & Pre-Lie (grafting) product on $\CT^\ex$, $\bgraft_p = \pi_R\circ\hgraft_p$  &
\pageref{bgraft_page_ref}\\
$\uparrow_{b}$ & Polynomial raising operator on $\CB$ for $b\in[d]$ &
\pageref{uparrow_page_ref}\\
$\hat\uparrow_{b}$ & Polynomial raising operator on $\CV^\ex$ for $b\in[d]$ &
\pageref{hat_uparrow_page_ref}\\
$\bone$ & Unit for product on $\CV$ and $\CB$ & \pageref{bone_page_ref},~\pageref{bone_CB_page_ref}\\
$A$ & Index set of the regularity structure $(A,\mcT,G_+)$ & \pageref{A_page_ref}\\
$\mfB$ & Enlarged space of trees with extra decorations on polynomials & \pageref{mfB_pag_ref} \\
$\CB$ & Span $\CB = \Span_\R(\mfB)$ & \pageref{CB_pag_ref} \\
$\Deltap$ & Positive coaction/coproduct $\Deltap\colon \CH \to \CH\otimes \CT_+$, $\CH=\CT,\CT_+$ & \pageref{Deltap_page_ref},~\pageref{Deltap_page_ref_2}
\\
$\Deltam$ & Negative coaction/coproduct $\Deltam\colon \CH \to \CF_-\otimes \CH$, $\CH=\CT,\CF_-$ & \pageref{Delta_minus_page_ref},~\pageref{Delta_minus_page_ref_2}\\
$[d]$ & $[d]=\{1,\ldots, d\}$ & \pageref{[d]_page_ref}\\
$\e$ & Empty forest $\e=0\in\N^\mfT$ & \pageref{empty_forest_page_ref}\\
$\mcE$ & Edge types $\mcE = \mcL\times\N^d$ & \pageref{def:types}\\
$F$ & Non-linearity $F\colon \R^\mcE \to \R$ &
\pageref{def:obey}\\
$\F$ & Lift of $F$ to non-linearity on jets &
\pageref{eq:F_def}\\
$\cbF$ & Extended non-linearity $\cbF=\{F_a\}_{a\leq 0}$ &
\pageref{def:ext_non-linear}\\
$F^{G^\ex_-}$ & Orbit of $F$ under action of $G^\ex_-$ &
\pageref{def:F_orbit}\\
$\mfF$ & Set of forests $\mfF = \N^\mfT$ &
\pageref{forests_page_ref}\\
$\mfF_-$ & Subset $\mfF_-\subset\mfF$ of forests with trees in $\mfT_-$ &
\pageref{def:neg_trees}\\
$\mcF_-$ & Span $\mcF_- = \Span_\R(\mfF_-)$, which is the `negative' Hopf algebra  &
\pageref{mcF_minus_pageref}\\
$f_x$ & Positive character $f_x\in G_+$ of $\bfPi$ at $x\in \R^d$ & \pageref{Pi_x_page_ref}\\
$F_x$ & Linear map $F_x\in\GL(\CT)$ given by $F_x=\Gamma_{f_x}$ & \pageref{Pi_x_page_ref}\\
$\Gamma_f$ & Representation $G_+\ni f\mapsto \Gamma_f = (\id\otimes f)\Deltap \in \GL(\CT)$ & \pageref{ex:prod_G_+}\\
$G_+$ & Characters on $\CT_+$, i.e. structure group of regularity structure  &
\pageref{def:mfT_plus}\\
$G_-$ & Characters on $\CF_-$ &
\pageref{G_minus_page_ref}\\
$\mfG$ & Set of generators of $\CV^\ex$ & \pageref{mfG_page_ref}\\
$\mfG_\circ$ & Set of generators of $\CT^\ex$ given by $\mfG_\circ = \mfG\cap \mfT_\circ$ &
\pageref{mfG_circ_page_ref}\\
$\mcI$ & Integration type $\mcI\in\mfL$ & \pageref{def:types}\\
$\CJ$ & Decoration added to polynomials in $\mfB$ & \pageref{CJ_page_ref}\\
$K$ & Kernel $K\colon \R^d\setminus\{0\}\to\R$ & \pageref{K_page_ref}\\
$\mfL$ & Set of all types $\mfL = \{\Xi,\mcI\}$ & \pageref{def:types}\\
$\mcL(\Phi)$ & Lift to branches of operator $\Phi$ & \pageref{def:mcL_lift},~\pageref{mcL_tensor_page_ref}\\
$M_g$ & Linear map  $M_g = (g\otimes \id)\Deltam \in \GL(\CH)$ with $\CH=\CT,\CT_+$ & \pageref{M_g_page_ref}\\ 
$\N$ & Non-negative integers $\N=\{0,1,\ldots\}$ & \pageref{mfN_page_ref}\\
$\mfN$ & Set of functions $\N^d\to\N$ with finite support & \pageref{mfN_page_ref}\\
$\mcN$ & Multiset of edges types at a node or root of tree & \pageref{eq:mcN}\\
$\bfPi$ & Linear map $\bfPi\colon\CT\to C^\infty$ &
\\
$\Pi_x$ & Recentred map $\Pi_x=\bfPi \Gamma_{f_x}$ at $x\in \R^d$ & \pageref{Pi_x_page_ref}\\
$\pi_R$ & Projection $\pi_R\colon\CV^\ex\to\CT^\ex$ annihilating trees in $\mfT^\ex\setminus\mfT^\ex_\circ$ & \pageref{pi_R_page_ref}\\
$Q$ & Projection $Q\colon\mfB\to\mfT^\ex$ & \pageref{Q_page_ref}\\
$R$ & Rule $R\subset \N^\mcE$ & 
\pageref{def:rule}\\
$\s$ & Scaling $\s=(\s_1,\ldots,\s_d)\in [1,\infty)^{d}$ & \pageref{scaling_page_ref}\\
$\mfT$ & Set of rooted, decorated trees & \pageref{def:mfT}\\
$\mfT^\ex$ & Set of rooted, decorated trees with extended decorations & \pageref{mfT^ex page ref}\\
$\mfT_\circ$ & Subset $\mfT_\circ\subset \mfT$ of trees strongly conforming to $R$ & \pageref{def:conform}\\
$\mfT_+$ & Free commutative monoid generated by planted `positive trees' & \pageref{def:mfT_plus}\\
$\mfT_-$ & Subset $\mfT_-\subset \mfT_\circ$ of negative, non-planted trees & \pageref{def:neg_trees}\\
%$\mfT^\ex_\ast$ & Subset $\mfT_\ast^\ex \subset \mfT^\ex$ of trees with more than one node & \pageref{mfT_star_pageref}\\
$\mcT$ & $\mcT=\Span_\R (\mfT_\circ)$, i.e. model space of regularity structure  &
\pageref{def:mcT}
\\
$\mcT_+$ & Span $\mcT_+=\Span_\R (\mfT_+)$, which is the `positive' Hopf algebra &
\pageref{def:mfT_plus}
\\
$\mcT_-$ & Span of negative, non-planted trees,  $\mcT_-=\Span_\R (\mfT_-)$ &
\pageref{mcT_minus_pageref}\\
$U$ & Jet &
\pageref{eq:jet}\\
$U^P$ & Polynomial part of $U$ &
\pageref{eq:jet}\\
$U^R$ & Non-polynomial part of $U$ &
\pageref{eq:jet}\\
$\Upsilon^\cbF$ & Coherence map of extended non-linearity $\cbF$ &
\pageref{def:Upsilon}\\
$\mathring\Upsilon^\cbF$ & Lifted coherence map on $\CB$ of $\cbF$ &
\pageref{eq:mathring_Upsilon_def}\\
$\CV$ & Span of all trees $\CV = \Span_\R(\mfT)$ &
\pageref{CV_page_ref}\\
$\CV^\ex$ & Span of all trees with extended decorations $\CV^\ex = \Span_\R(\mfT^\ex)$ &
\pageref{def:Upsilon}\\
$\Xi$ & Noise type $\Xi\in \mfL$& \pageref{def:types}\\
$\mcY_o$ & Non-linearity $\mcY_o\colon \R^\mcE\to\R$ evaluating the $o\in\mcE$ component &
\pageref{ex:Phi43_rule},~\pageref{mcY_page_ref}
\end{longtable}
 \end{center}

\endappendix

\subsection*{Acknowledgements}

{\small
The author is grateful to the ESI for its hospitality during the Master Class and Workshop ``Higher Structures Emerging from Renormalisation'' held on November 8 - 19, 2021.
The author thanks the participants of the workshop for questions and comments
and is grateful to Yvain Bruned for discussions and feedback on a first draft of this manuscript.}
	
	\bibliographystyle{Martin}
	\bibliography{refs}
	
\end{document}